\documentclass[11pt]{article}
\usepackage{fullpage,amssymb}
\usepackage{caption}
\usepackage{subcaption}
\usepackage{graphicx}
\usepackage{amsmath}
\usepackage{float}
\usepackage{listings}
\usepackage{xcolor}
\usepackage{amsthm}
\usepackage{algorithmic}
\usepackage{algorithm}
\usepackage[colorlinks,
linkcolor=blue,
anchorcolor=blue,
citecolor=blue]{hyperref}
\usepackage{natbib}
\usepackage{comment}
\usepackage{bbm}
\usepackage{multirow}
\usepackage{array}
\usepackage[shortlabels]{enumitem}
\usepackage{romannum}
\AtBeginDocument{\pagenumbering{arabic}}
\usepackage{multirow}
\usepackage{stmaryrd}
\usepackage{amsmath}

\setcitestyle{authoryear,round}

\newtheorem{theorem}{Theorem}
\newtheorem{lemma}{Lemma}
\newtheorem{definition}{Definition}
\newtheorem{proposition}{Proposition}
\newtheorem{corollary}{Corollary}
\newtheorem{assumption}{Assumption}
\newtheorem{remark}{Remark}

\newtheorem{example}{Example}

\def\EE{{\mathbb E}}

\def\II{{\mathbb I}}

\def\PP{{\mathbb P}}

\def\RR{{\mathbb R}}
\def\SS{{\mathbb S}}

\def\g{{\mathbf g}}

\def\u{{\mathbf u}}
\def\v{{\mathbf v}}

\def\x{{\mathbf x}}

\def\z{{\mathbf z}}

\def\A{{\mathbf A}}
\def\B{{\mathbf B}}

\def\D{{\mathbf D}}
\def\E{{\mathbf E}}

\def\G{{\mathbf G}}

\def\I{{\mathbf I}}

\def\R{{\mathbf R}}

\def\U{{\mathbf U}}
\def\V{{\mathbf V}}

\def\X{{\mathbf X}}
\def\Y{{\mathbf Y}}
\def\Z{{\mathbf Z}}

\def\calA{{\mathcal A}}

\def\calC{{\mathcal C}}

\def\calE{{\mathcal E}}
\def\calF{{\mathcal F}}

\def\calH{{\mathcal H}}

\def\calN{{\mathcal N}}

\def\calS{{\mathcal S}}
\def\calT{{\mathcal T}}
\def\calU{{\mathcal U}}

\def\calX{{\mathcal X}}

\def\Bbeta{{\boldsymbol{\beta}}}
\def\bSigma{{\boldsymbol{\Sigma}}}
\def\eps{\varepsilon}
\def\Ca{C_{a}}
\def\Cb{C_{b}}

\def\ComSQ{\textsf{CoM}_{2}}
\def\ComOne{\textsf{CoM}_{1}}
\def\ComMax{\textsf{CoM}_{\infty}}
\def\OffDiagS{\lambda_{\max,s}^{\textsf{Off}}}
\def\OffDiagO{\lambda_{\max,1}^{\textsf{Off}}}

\def\l{\left}
\def\r{\right}

\begin{document}
	\title{SGD with Dependent Data: \\Optimal Estimation, Regret, and Inference}
    \author{Yinan Shen$^1$,~Yichen Zhang$^2$,~Wen-Xin Zhou$^3$ \\ 
    ~\\
    $^1$Department of Mathematics, University of Southern California  \\
    $^2$Daniels School of Business, Purdue University  \\
    $^3$Department of Information and Decision Sciences, University of Illinois Chicago  \\
    }
    \date{}
	\maketitle
	\begin{abstract}
		This work investigates the performance of the final iterate produced by stochastic gradient descent (SGD) under temporally dependent data. We consider two complementary sources of dependence: $(i)$ martingale-type dependence in both the covariate and noise processes, which accommodates non-stationary and non-mixing time series data, and $(ii)$ dependence induced by sequential decision making. Our formulation runs in parallel with classical notions of (local) stationarity and strong mixing, while neither framework fully subsumes the other. Remarkably, SGD is shown to automatically accommodate both independent and dependent information under a broad class of stepsize schedules and exploration rate schemes.  
	
	Non-asymptotically, we show that SGD simultaneously achieves statistically optimal estimation error and regret, extending and improving existing results. In particular, our tail bounds remain sharp even for potentially infinite horizon $T=+\infty$. Asymptotically, the SGD iterates converge to a Gaussian distribution with only an $O_{\PP}(1/\sqrt{t})$ remainder, demonstrating that the supposed estimation-regret trade-off claimed in prior work can in fact be avoided. We further propose a new ``conic'' approximation of the decision region that allows the covariates to have unbounded support. For online sparse regression, we develop a new SGD-based algorithm that uses only $d$ units of storage and requires $O(d)$ flops per iteration, achieving the long term statistical optimality. Intuitively, each incoming observation contributes to estimation accuracy, while aggregated summary statistics guide support recovery.
	\end{abstract}


\section{Introduction}
\label{sec:intro}

Stochastic gradient descent (SGD) is a fundamental algorithm for online learning and sequential decision making. Despite its simplicity, understanding its statistical behavior becomes considerably more challenging when the incoming data exhibit dependence, either because the covariates or noise form a time series, or because the data are collected adaptively through the learner's decision. Such dependence is pervasive in modern applications, including contextual bandits, financial and economic forecasting \citep{black1986noise,ubukata2009estimation,jacod2017statistical}, and streaming prediction problems \citep{robbins1951stochastic,lai1987adaptive,wu2006robust,wu2007strong,wu2007inference}. Consequently, a central methodological and theoretical question is how to design and analyze SGD procedures that automatically adapt to potentially mixed sources of independent and dependent observations, without requiring explicit knowledge of their underlying dependence structure.

A substantial literature has developed tools for analyzing dependent processes, spanning (local) stationarity and physical or functional dependence measures \citep{dedecker2000functional, wu2007strong, xiao2012covariance, nagaraj2020least, srikant2024rates}, $\rho$-mixing and other related weak dependence conditions \citep{masry1997local, doukhan1994functional, merlevede2009bernstein, lu2022almost}, and density-based characterizations of conditional distributions \citep{ tartakovsky2024nearly}. While these frameworks provide powerful theoretical guarantees, they typically rely on precise characterizations of temporal dependence, such as stationarity, mixing rates, or conditional density. However, modern online environments frequently violate these assumptions: covariates may be neither stationary nor mixing, noise distributions may evolve over time, and adaptive data collection introduces additional dependence channels that are not adequately captured by classical stochastic process theory \citep{azuma1967weighted,koltchinskii1994komlos, van2002hoeffding, lauer2023uniform}.

This paper develops an estimation and inference framework for SGD that accommodates the simultaneous presence of independent and dependent observations, without attempting to classify or disentangle their sources. Our formulation parallels or extends classical dependence concepts but is more flexible. In particular, covariates may evolve according to processes that are neither independent across time nor locally stationary nor strongly mixing, yet still fall within the scope of our theory. Moreover, adaptively collected rewards in contextual bandit settings introduce a second source of dependence: the evolving decision rule interacts with the data-generating mechanism, jointly shaping the observed sequence. Our framework incorporates both types of dependence within a unified analysis. To illustrate the wide range of dependence structures encompassed by our analysis, consider a sequence $d$-dimensional covariates $\{\X_t\}$ evolving according to either of the following dynamics: (i) $\X_t=\nu_{t} \X_{t-1} / \|\X_{t-1}\| +\E_t$, or (ii) $\X_0\sim\mathrm{Unif}(\SS^{d-1})$ and $\X_t =\nu_{t,0}\X_0+\cdots+\nu_{t,t-1}\X_{t-1}+\E_t$ for $t\geq 1$, where $\{\nu_t\}$, $\{\nu_{t,i}\}$ and $\{\E_t\}$ are random but not necessarily stationary or independent, with $\sum_{i=0}^{t}|\nu_{t,i}|\leq 1/2$.  In such settings, classical stationary or mixing analyses are generally inapplicable. Nevertheless, we show that under suitable regularity conditions on the covariates and noise, the non-asymptotic performance of SGD with these dependent covariates can be established at a level comparable to the idealized i.i.d. Gaussian setting. Furthermore, under mild convergence and finite-moment conditions, we derive the asymptotic distribution of the SGD iterates.

 The same type of dependence may also arise in the noise process, which is likewise accommodated by our framework. More broadly, the dependence considered in this work is two-fold: in addition to the intrinsic temporal dependence of covariates and noise, we incorporate the dependence induced by decision making. This latter form of dependence is intrinsic to contextual bandit problems \citep{goldenshluger2013linear, hao2020high, bastani2020online,  bastani2021mostly, chen2020statistical, chen2021statisticalb, chen2021statistical, han2022online, duan2024online, han2024ucb}, where the learner observes covariates, selects an action, and receives a reward accordingly.  The balance between exploration (random pulls) and exploitation (pulling the empirically best arm) is crucial. Existing exploration strategies fall into two categories. The first relies on a prespecified grid of exploration times \citep{goldenshluger2013linear, hao2020high, bastani2020online, bastani2021mostly, ren2024dynamic}, which enables elegant offline analysis but requires storing and repeatedly using historical observations, as well as prior knowledge of the total horizon. The second category consists of $\eps$-greedy-type schemes \citep{chen2020statistical, chen2021statistical, chen2021statisticalb, han2022online, chen2022online,duan2024online}, which specify the exploration probability $\pi_t$ without knowledge of the horizon and do not store the full data stream. While this leads to practically appealing fully online algorithms, it also introduces substantial theoretical challenges. In particular, existing online analyses under independent data typically require relatively high exploration rates, often assuming either $\lim_{t\to\infty}\pi_t > 0$ or $\pi_t=t^{-\alpha}$ with $\alpha\in(0,1)$, resulting in regret rates of order $O(T)$ or $O(T^{2/3})$, which are strictly worse than offline guarantees.

This work aims to close this gap. We allow a broad class of exploration schedules, including choices of the form $\pi_t=f(t)$ for functions $f$ specified in Definition~\ref{def:exploration rate}. Despite this flexibility, our algorithm attains simultaneously optimal estimation error and regret. To incorporate dependent streaming data without storing past observations, we update parameters via SGD and introduce a new family of stepsize schedules. As shown in Section~\ref{sec:bandit}, when $f(t)$ satisfies Definition~\ref{def:exploration rate} and the growth condition $\int_0^T f(t) {\rm d}t \le O(\sqrt{T})$, both the regret and estimation error achieve minimax-optimal rates, even for discrete covariates. Although the covariates and noise may be dependent, our non-asymptotic guarantees match the optimal offline rates that assume i.i.d. data.

Another challenge in fully online analysis, without retaining historical data, is the accumulation of tail probabilities \citep{jin2016provable, chen2020statistical, chen2021statistical, li2023online, duan2024online, shen2024online}. This issue forces prior work to restrict the horizon to $T = O(d^C)$ or $T\le \exp(cd)$. Our analysis avoids this limitation: under our SGD algorithm, all non-asymptotic guarantees remain valid for $T=+\infty$. In Section~\ref{sec:LR}, by leveraging the lower bound results of \cite{ma2024high}, we further show that our result is statistically optimal among algorithms that allow $t=1,\ldots,+\infty$. Asymptotically, we obtain a Bahadur-type representation
$$
    \sqrt{t}\l(\Bbeta_t-\Bbeta^*\r) = \text{a term that converges to Gaussian in distribution}+O_{\PP}\bigg( \frac{1}{\sqrt t} \bigg) ,
$$
which is sharp for an online procedure. Existing online methods, even under i.i.d. data and with weaker regret guarantees, incur residuals of order $t^{1/5}$ due to dependence effects. The multifold optimality achieved here stems from the coordinated design of both the algorithms and the analysis.

Online sparse linear regression has long remained a challenging problem. How to estimate a sparse parameter in an online manner is largely open. Conceptually, online sparse linear regression integrates estimation in Euclidean norm \citep{blumensath2009iterative,bogdan2015slope,bellec2018slope,sun2020adaptive} with support recovery \citep{fan2008sure,candes2009near,fan2014adaptive,zhang2010nearly}. As streaming data arrive, the goal is to achieve a monotonically decreasing estimation error while eventually recovering the entire support. However, existing offline theories and assumptions do not directly transfer to the online setting. For example, offline analyses often require a lower bound on the minimal signal strength \citep{zhang2010nearly}, such as $\min_{j: \beta_j \neq 0} |\beta_j| \geq O(1/\sqrt{n})$ for a fixed sample size $n$. In contrast, online learning may involve an unknown or unbounded horizon. To circumvent this difficulty, \cite{fan2018statistical} assume that the initialization step already identifies the true support, though they still require a minimal signal strength condition. \cite{han2024online} and \cite{yang2023online} develop online LASSO algorithms, but they recompute a full LASSO estimator each time a new observation arrives. All three studies focus on sparse estimation under i.i.d. observations.

The algorithm of SGD, or stochastic approximation, dates back to \cite{robbins1951stochastic,kiefer1952stochastic,polyak1992acceleration}. Although the final iterate of SGD is widely used in practice, its statistical performance, especially whether it can match that of offline estimators in terms of estimation error and non-asymptotic guarantees, remains largely unclear. This work investigates the behavior of the last iterate of SGD without storing historical data and without averaging over past iterates. Below, we summarize our main contributions and Section~\ref{sec:numeric} numerically verifies our theoretical results.
\begin{enumerate}
    \item {\it Dependent and Independent Data:} Our dependence structure runs in parallel to classical stationary or mixing frameworks, and neither class contains the other. Specifically, we generalize the i.i.d. Gaussian setting to random variables with bounded conditional Orlicz norms. Interestingly, even when the covariates and noise exhibit dependence, the non-asymptotic performance of SGD remains essentially the same as in the i.i.d. Gaussian case. Asymptotic results are established under mild convergence and moment conditions, and the Markov structure does not degrade the asymptotic performance of SGD. The resulting Bahadur-type remainder is as small as $1/\sqrt{t}$.

    \item {\it Performance of SGD in Linear Regression:} Whether online SGD can yield estimators that are statistically comparable to offline approaches has remained unknown even for linear regression. In Section~\ref{sec:LR}, we investigate its non-asymptotic and asymptotic performance in the presence of potentially dependent covariates and noise, without involving decision making. From a non-asymptotic perspective, we show that the tail probability 
    $$
       \max_{\Bbeta_t(\{\X_l,Y_l\}_{l=1}^t)}\PP\bigg(\cup_{t=0}^T\bigg\{\|\Bbeta_t-\Bbeta^*\|^2\leq \frac{C d}{t}\bigg\}\bigg)
    $$ 
    is at least $1-T\exp(-cd)$, matching the best offline results. Moreover, we show that the trajectory
    $$
        \cup_{t=0}^{+\infty}\bigg\{\|\Bbeta_t-\Bbeta^*\|^2\leq C\frac{d+\log t}{t}\bigg\}
    $$
    holds with probability at least $1-\exp(-cd)$. Borrowing results from \cite{ma2024high}, we further illustrate that our estimation error and tail probability are optimal, up to constants. Existing studies focus on the i.i.d. observations and require $T\leq d^C$ or $T\leq \exp(cd)$ \citep{jin2016provable, chen2020statistical, chen2021statistical, li2023online, duan2024online,shen2024online}.
    
    \item {\it Decision Making and Linear Bandit:} Whether SGD can simultaneously achieve statistical optimality in estimation and regret in an online manner has been unknown, even under i.i.d. covariates and noise. This work closes this gap and establishes such results under dependent covariates and noise. Our guarantees hold under minimal assumptions for a broad class of exploration rates and stepsizes, including rates generated by general functions in Definition~\ref{def:exploration rate}. We further show that the SGD iterates under our stepsize scheme converge to a multivariate Gaussian. In addition, we establish the first inference results for both decaying and zero exploration rates by fully incorporating dependent data. The proposed stepsize scheme is also new.
    
    \item {\it Online Sparse Learning:} Designing an online algorithm for sparse linear regression that neither stores the full data nor solves a new optimization problem whenever a new sample arrives has remained open. This work advances the state of the art by proposing a new SGD-based algorithm for sparse linear regression with storage cost $d$ and per-iteration computation $O(d)$. Intuitively, each individual observation contributes to improving estimation accuracy, while summary statistics are used for variable selection. Our algorithm achieves statistically optimal iterates in the long run. 
    
    \item {\it Conic Decision Region:} This work introduces a new conic-shaped approximation of the true decision making region, which allows the covariates to have possibly unbounded support and to substantially cover a neighborhood around the origin. Additionally, our asymptotic results are based on the conic measurement, which removes the continuous conditions of the distribution in existing works \citep{chen2022online,chen2020statistical,chen2021statistical}. When covariates are discretely distributed, our characterization coincides with those used in \cite{bastani2020online} and \cite{bastani2021mostly}.
\end{enumerate}

In addition, the statistical theory developed in this work may be of independent interest. We conclude this section by introducing the notation used throughout the paper.

\paragraph*{Notation} We use bold symbols to denote vectors and matrices (e.g., $\g,\X,\Y$), and calligraphic font to denote sets, operators, or $\sigma$-algebras (e.g., $\calA,\calH,\calF$). Throughout, $[K]$ denotes the set $\{1,2,\ldots,K\}$. For two symmetric matrices $\A$ and $\B$, $\A\succeq\B$ means that $\A-\B$ has non-negative eigenvalues. The operator $\|\cdot\|$ denotes the $L_2$ norm for vectors and the operator norm for matrices, while $\|\cdot\|_{\infty}$ denotes the maximum absolute entry of a vector. For a symmetric matrix $\A=\U\boldsymbol{\Lambda}\U^\top\in\RR^{d\times d}$ with $\boldsymbol{\Lambda}=\operatorname{diag}(\lambda_1,\ldots,\lambda_d)$ and a function $f(\cdot)$, we define $f(\A)=\U f(\boldsymbol{\Lambda})\U^{\top}$ and $f(\boldsymbol{\Lambda})=\operatorname{diag}(f(\lambda_1),\ldots,f(\lambda_d))$. We write $\rightsquigarrow$ for convergence in distribution.
	
For a vector $\X\in\RR^d$, $[\X]_i$ denotes its $i$-th entry, and for a set $\calS\subseteq[d]$, $[\X]_{\calS}$ denotes the subvector of $\X$ with indices in $\calS$. In the context of sparse linear regression, the operator $\calH_{\calS}(\cdot):\RR^d\to \RR^d$ retains only the entries in $\calS$ and sets the entries in $\calS^c$ to zero; that is, $[\calH_{\calS}(\X)]_{i} = [\X]_i$ for $i\in \calS$ and $0$ for $i\in\calS^c$.

\section{Linear Regression with Dependent Data}
\label{sec:LR}

We begin by introducing the conditional Orlicz norm, which plays a central role in our analysis. This concept is not new; rather, it generalizes several cases. Foundational work such as \citet{koltchinskii1994komlos}, \citet{van1995exponential}, \citet{van2002hoeffding}, \citet{shamir2011variant}, and \citet{lauer2023uniform}, along with the classical Azuma's inequality \citep{azuma1967weighted}, has explored concentration inequalities for dependent random variables that are either bounded or satisfy bounded conditional sub-Gaussian tails. More recently, as non-asymptotic bounds for the estimation error of stochastic gradient descent (SGD) have attracted increasing attention, conditional Orlicz norm-based techniques have been employed in works such as \citet{han2022online} and \citet{shen2024online}. However, these studies focus on independent covariates and noise. Our work substantially generalizes the framework by allowing for more general dependence structures. In particular, we show that under suitable conditions, SGD under dependent data performs comparably to the idealized i.i.d. Gaussian setting.

\subsection{Preliminaries: Orlicz norm and Conditional Orlicz Norm}
    
Recall that the Orlicz norm of a univariate random variable is defined as 
$$
\|X \|_{\Psi_\alpha}:=\inf\big\{u>0: \, \EE\exp( |X/u|^\alpha )\leq 2\big\}, \quad \alpha\geq 1.
$$ 
Let $\calF$ be a $\sigma$-field. A random variable $K_{\calF}$ is said to be $\calF$-measurable if $K_{\calF}\in\calF$.
    
\begin{definition}[Conditional Orlicz Norm]
Suppose $\calF$ is a $\sigma$-field and $K_{\calF} >0$ is $\calF$-measurable. A random variable $X$ satisfies $\|X |\calF\|_{\Psi_\alpha}\leq K_{\calF}$ if and only if 
\begin{equation}
\EE\big\{ \exp( |X/K_{\calF}|^\alpha)|\calF \big\} \leq 2\quad a.s.
\end{equation}  
\end{definition}
       
Note that the conditional expectation above is itself $\calF$-measurable random variable. We illustrate the concept of the conditional Orlicz norm with the following examples.
       
\begin{example}
If $X$ is $\calF$-measurable, then $\|X|\calF\|_{\Psi_\alpha}\leq  (\log 2)^{-1/\alpha }\cdot|X|$. If $X$ is independent of $\calF$, then $\|X|\calF\|_{\Psi_\alpha}\leq \|X\|_{\Psi_2}$.
\end{example}
        
For a random vector $\X\in\RR^d$, define 
$$
\|\X\|_{\Psi_\alpha}:=\sup_{\u\in\SS^{d-1}} \|\u^{\top}\X\|_{\Psi_\alpha}.
$$ 
The conditional version is defined analogously as follows.

\begin{definition}[Conditional Orlicz Norm of Random Vectors]
Suppose $\calF$ is a $\sigma$-field and $K_{\calF} >0$ is $\calF$-measurable. A random vector $\X$ satisfies $\|\X |\calF\|_{\Psi_\alpha}\leq K_{\calF}$ if and only if 
\begin{equation}
\sup_{\u\in\SS^{d-1},\u\in\calF}  \EE\big\{ \exp( |\u^{\top}\X /K_{\calF} |^\alpha)|\calF \big\} \leq 2\quad \text{a.s.}
\end{equation}
\end{definition}

Additionally, for a random matrix $\X\in\RR^{d_1\times d_2}$, we define $$\|\X\|_{\Psi_\alpha}:=\sup_{\u\in\SS^{d_1-1}}\sup_{\v\in\SS^{d_2-1}}\|\u^{\top}\X\v\|_{\Psi_\alpha},$$ with conditional norms defined similarly. These norms are well defined and satisfy the usual norm properties; details are deferred to Appendix~\ref{sec:proofLR}. Finally, we present another illustrative example of the conditional Orlicz norm, which can be verified directly.
       
\begin{example}
Let $\calF=\sigma(X)$. Define $Y_1=X+\xi$ and $Y_2=\xi X$, where $\xi$ is independent of $X$ and $\|\xi\|_{\Psi_\alpha} < \infty$. Then 
$$
\|Y_1|\calF\|_{\Psi_\alpha}\leq  |X| / \log 2 +\|\xi\|_{\Psi_{\alpha}},\quad \|Y_2|\calF\|_{\Psi_\alpha}\leq |X|\cdot\|\xi\|_{\Psi_\alpha}.
$$
If $|X|\leq K$ almost surely,  $\|Y_1|\calF\|_{\Psi_\alpha}\leq K/\log 2 +\|\xi\|_{\Psi_{\alpha}} $ and $\|Y_2|\calF\|_{\Psi_\alpha}\leq K\cdot\|\xi\|_{\Psi_\alpha}$.
\end{example}

\subsection{Linear Regression with Dependent Covariates and Noise}

    To isolate the effects of dependence from those arising from decision making, we first study dependent-data linear regression via SGD. At each time $t$, we observe $(\X_t, Y_t) \in \RR^d \times \RR$ generated from the linear model
    \begin{align*}
		Y_t=\X_t^{\top}\Bbeta^*+\xi_t,
	\end{align*}
	where $\X_t$ and $\xi_t$ may depend on past observations. The goal is to estimate $\Bbeta^* \in \RR^d$ in an online manner, processing each data point as it arrives without storing the entire history. We allow an arbitrary initialization $\Bbeta_0\in\mathbb{R}^{d}$. At time $t$, suppose the current iterate is $\Bbeta_{t-1}$ and a new data pair $(\X_t,Y_t)$ becomes available. We then update the parameter using the standard SGD rule:
    \begin{align}
	\Bbeta_{t+1}=\Bbeta_{t}-\eta_{t}\cdot(\X_t^{\top}\Bbeta_{t}-Y_t)\X_t,
	\label{alg:regression}
    \end{align}
    where $\eta_{t}$ is the stepsize to be specified later. Let $\calF_t = \sigma(Y_{t-1}, \X_{t-1}, \ldots, Y_0, \X_0)$ denote the $\sigma$-field generated by all randomness up to time $t-1$, and define $\calF_t^+ = \sigma(\X_t, Y_{t-1}, \X_{t-1}, \ldots, Y_0, \X_0)$. The update in \eqref{alg:regression} ensures that $\Bbeta_t$ is $\calF_t$-measurable, and that $\calF_t \subseteq \calF_t^+ \subseteq \calF_{t+1}$. We next establish the convergence of \eqref{alg:regression} under the following assumptions.

\begin{assumption}[Covariates: Martingale Difference with Conditional Sub-Gaussian Norm]
\label{assm:cov}
	The covariate sequence $\X_t\in\RR^d$ satisfies $\EE \{\X_t|\calF_t \}=0$ a.s., and $\|\X_t|\calF_t\|_{\Psi_2}^2\leq \lambda_{\max}$. Moreover, there exists $\lambda_{\min}>0$ such that $\EE \{\X_t\X_t^{\top}|\calF_t \}\succeq \lambda_{\min}\cdot\I_{d}$ a.s. Here, $\lambda_{\min}$ and $\lambda_{\max}$ are constants. 
\end{assumption}

\begin{assumption}[Noise: Martingale Difference with Conditional Sub-Gaussian Norm]
 \label{assm:noise}
	The noise variable $\xi_t\in\RR$ satisfies $\EE \{\xi_t|\calF_t^+ \}=0$ a.s., and $\|\xi_t|\calF_t^+\|_{\Psi_2}\leq \sigma$, where $\sigma$ is a constant. 
\end{assumption}

\noindent 
Assumptions \ref{assm:cov} and \ref{assm:noise} allow both $\X_t$ and $\xi_t$ to depend on the historical data. This generalizes the commonly imposed assumptions of i.i.d. or uniformly bounded covariates, as seen in \cite{perchet2013multi}, \cite{fang2018online}, \cite{chen2020statistical}, \cite{han2022online}, \cite{shao2022berry}, \cite{li2023online}, \cite{agrawalla2025statistical}, \cite{wei2025online}, \cite{han2024online} and \cite{shen2024online}. We emphasize that our framework is neither contained in nor contains the classical stationary or strong-mixing settings; the two regimes are fundamentally incomparable \citep{wu2007strong,xiao2012covariance}.

	\begin{example}
	    Let $\X_0$ be any centered random vector satisfying $\|\X_0\|_{\Psi_2}< \infty$ and $\lambda_{\min}(\EE\{\X_0\X_0^{\top}\})>0$. For $t\geq 1$, consider the process $\X_t=a_t\cdot\frac{\X_{t-1}}{\|\X_{t-1}\|}+\E_t$, where $\EE\{a_t|\calF_t\}=0$, $|a_t|\leq 1$, $\|\E_t|\calF_t\|_{\Psi_2}\leq \lambda$, and $\E_t\perp\X_{t-1}|\calF_{t}$, with $\lambda_{\min}(\EE  \{\E_t\E_t^{\top} \})>0$. In this construction, $\{\X_t\}$ is dependent and may fail to be stationary or locally stationary. Nevertheless, it satisfies Assumptions \ref{assm:cov} and \ref{assm:noise}.
	\end{example}

    The next theorem establishes the non-asymptotic performance of SGD under two stepsize schemes. Interestingly, SGD can well incorporate the dependent data as if they were i.i.d. Gaussian, which is also verified with numeric experiments (Section~\ref{sec:numeric}).

    \begin{theorem}[Non-asymptotic Performance]
    \label{thm:nonsymLR}
	Suppose Assumptions \ref{assm:cov} and \ref{assm:noise} hold. The iterates generated by update \eqref{alg:regression} satisfy:
	\begin{enumerate}
		\item \textbf{Constant Stepsize.} With a constant stepsize $\eta_{t}:=\eta \leq \lambda_{\min}/(d\lambda_{\max}^2)$, the following holds with probability at least $1-(t+1)\exp(-c_1d)$:
		\begin{align*}
			\|\Bbeta_{t+1}-\Bbeta^*\|^2\leq 5(1-\eta \lambda_{\min}/2)^{t+1}\|\Bbeta_{0}-\Bbeta^*\|^2 +5\frac{\lambda_{\max}}{\lambda_{\min}}\eta\sigma^2d.
		\end{align*}
	In particular, after $t_1 \ge \log(\|\Bbeta_{0}-\Bbeta^*\|/(\sqrt{\eta d}\sigma))$, we have $\|\Bbeta_{t_1}-\Bbeta^*\|^2 \le C\frac{\lambda_{\max}}{\lambda_{\min}}\eta\sigma^2 d \le C\frac{\sigma^2}{\lambda_{\max}}$.
	\item  \textbf{Decaying Stepsize.} Suppose some iteration $t_1$ satisfies $\|\Bbeta_{t_1}-\Bbeta^*\|^2 \leq\sigma^2/\lambda_{\max}$.\footnote{This condition is imposed for clarity of presentation. The proof shows a more general dynamic, 
    $$
        \|\Bbeta_{t+1}-\Bbeta^*\|^2 \leq C\l(\frac{d}{t+1+\Cb d}\r)^{\Ca-2}\|\Bbeta_0-\Bbeta^*\|^2+  \frac{C}{\lambda_{\min}}  \frac{\max\{d,\delta_{t+1}\}}{t+1+\Cb d}  \sigma^2,
    $$
    where the effect of the initial error shrinks rapidly when $\Ca$ is large. Since a constant stepsize yields linear decay in the initial error, we assume the simplified condition $\|\Bbeta_{t_1}-\Bbeta^*\|^2 \leq\sigma^2/\lambda_{\max}$ for exposition.\label{note1}} Take $\eta_{t}=\frac{\Ca/\lambda_{\min}}{t-t_1+\Cb d}$ with $\Ca\geq 2$ and $\Cb\geq 3\Ca^2(\lambda_{\max}^2/\lambda_{\min}^2)$. For any non-decreasing sequence $\{\delta_t\}$ satisfying $(i)$ $0\leq\delta_t\leq\delta_{t+1}$, $(ii)$ $\delta_{t}\leq C t/\Cb$, we have, with probability at least $1-C\exp(-cd)-\sum_{l=t_1}^t\exp(-c(d+\delta_l))$, that for all $t\in[t_1,+\infty]$,
	\begin{align*}
			\|\Bbeta_{t+1}-\Bbeta^*\|^2\leq C^{*}\ \frac{1}{\lambda_{\min}} \frac{\max\{d,\delta_{t+1}\}}{t+1+\Cb d}  \sigma^2, \quad \text{where } C^{*}=2C\Ca\frac{\lambda_{\max}}{\lambda_{\min}}.
	\end{align*}
	\end{enumerate}
\end{theorem}

The proof of part (1) is deferred to Appendix~\ref{sec:proofLR}. Part (2) is included as a special case of Theorem~\ref{thm:LBD-estimation}, and thus its proof is omitted. Theorem~\ref{thm:nonsymLR} provides a two-phase online estimation strategy. The first phase, using a constant stepsize, ensures rapid linear contraction toward a neighborhood of $\Bbeta^*$; the radius of this neighborhood scales with the specified constant stepsize $\eta$. The second phase adopts a decaying stepsize and achieves statistically optimal performance, with error on the order of $O((d+\delta_t)/t)$. Taking $\delta_t = c \log t$ yields, with probability at least $1 - C \exp(-c d)$, for all $t= t_1,\ldots,+\infty$,
\begin{align}
\|\Bbeta_{t}-\Bbeta^*\|^2\leq C^{*} \frac{1}{\lambda_{\min}}  \frac{\max\{d,\log(t)\}}{t+\Cb d}  \sigma^2 . \label{eq:LR_logt}
\end{align}
Compared with offline linear regression, the additional factor $\log(t)/ t$ in online learning is unavoidable to allow $t=+\infty$. Indeed, \cite{ma2024high} shows that if $[\X_1,\ldots,\X_t]^\top$ satisfies suitable conditions, then
\begin{align*}
    \PP \bigg\{   \|\widehat{\Bbeta}_t(\{\X_l,Y_l\}_{l=1}^t)-\Bbeta^*\|\geq c\sigma^2\cdot\frac{d+\log(\delta)}{t} \bigg\} \geq\frac{1}{\delta}.
\end{align*}
Setting $\delta=t$ gives
    \begin{align*}
    \sum_{t=1}^{+\infty}\PP \bigg\{ \|\widehat{\Bbeta}_t(\{\X_l,Y_l\}_{l=1}^t)-\Bbeta^*\|\geq c\frac{\sigma^2}{\lambda_{\min}}\cdot\frac{d+\log(t)}{t} \bigg\} =+\infty ,
    \end{align*}
so the $\log(t)/t$ term in Theorem~\ref{thm:nonsymLR} is optimal, up to constants. Even when $\widehat{\Bbeta}_t$ are independent, the second Borel-Cantelli lemma implies
\begin{align*}
    \PP \bigg\{ \|\widehat{\Bbeta}_t(\{\X_l,Y_l\}_{l=1}^t)-\Bbeta^*\|\geq c\frac{\sigma^2}{\lambda_{\min}}\cdot\frac{d+\log(t)}{t}, \, {\rm i.o.} \bigg\} =1 .
\end{align*}

\begin{remark}[Tail Probability Accumulation]
    Non-asymptotic analyses for online estimation and SGD typically apply concentration bounds at each time step, since data are processed sequentially without storage. This inherently leads to accumulation of tail probabilities \citep{jin2016provable,han2022online,li2023online,shen2024online,liurevisiting}. Theorem~\ref{thm:nonsymLR} advances the state of the art by accommodating an infinite horizon $T\leq+\infty$, whereas earlier results were restricted to horizons of size at most $T \leq C\exp(d)$ (e.g., \citealp{shen2024online}). Enabling a non-asymptotic analysis for potentially unbounded horizons is a substantial and nontrivial improvement.
\end{remark}

\begin{remark}[Regret]
    Various definitions of regret exist for online linear regression. Typically, regret measures cumulative predictive performance rather than estimation quality at a single time point. Since Theorem~\ref{thm:nonsymLR} ensures optimal estimation error uniformly for all sufficiently large $t$, the corresponding regret is also statistically optimal. To avoid confusion with bandit-style regret definitions, we omit further discussion here.
\end{remark}

\noindent We now turn to the asymptotic behavior of $\Bbeta_t$ under Markovian covariates and noise with convergent second moments. The result shows that SGD iterates converge in distribution to a multivariate Gaussian. We use $O_{\PP}(\cdot)$ to denote stochastic boundedness.

\begin{theorem}[Asymptotic Performance]
\label{thm:asympLR}
    Assume $\EE\{\X_t|\calF_t\}=0$, $\EE\{\xi_t|\calF_t^+\}=0$, and $\xi_t \perp \X_t|\calF_t$. Suppose there exist $\bSigma^*$ and $\sigma^*$ such that 
	$$
     \displaystyle\lim_{t\to+\infty}\EE\big\{ \big\|\EE\{\X_t^{\top}\X_t|\calF_t\}-\bSigma^* \big\| \big\}=0,\,\text{ and } \lim_{t\to+\infty}\EE\big\{ \big|\EE\{\xi_t^2|\calF_t\}-\sigma_*^2 \big| \big\}=0.$$ 
     Additionally, suppose there exist $\lambda_{\max},\lambda_{\min},\sigma>0$ such that 
     $$
     \lambda_{\min}\I\preceq\EE\{\X_t\X_t^{\top}|\calF_t\},\quad \displaystyle \sup_{\V\in\SS^{d-1}:\,\V\in\calF_t}\EE\{(\X_t^{\top}\V)^4|\calF_t\}\leq\lambda_{\max}^2,\quad \EE\{\xi_t^4|\calF_t^+ \}\leq \sigma^4\quad\text{a.s.}
     $$ 
     With stepsize $\eta_{t}=\frac{\Ca/\lambda_{\min}}{t-t_1+\Cb d}$, where $\Ca\geq 2$ and $\Cb\geq 3\Ca^2(\lambda_{\max}^2/\lambda_{\min}^2)$, we have the decomposition
    \begin{align*}
        \Bbeta_t-\Bbeta^*=\sum_{j=0}^{t}\prod_{l=j+1}^t\l(\I-\eta_l\bSigma^*\r)\eta_j\cdot\xi_j\cdot\X_j+\R_t ,
    \end{align*}
    where $R_t=O_{\PP}(1/t)$. Moreover, 
    $$
        \sqrt{t}\sum_{j=0}^{t}\prod_{l=j+1}^t\l(\I-\eta_l\bSigma^*\r)\eta_j\cdot\xi_j\cdot\X_j \rightsquigarrow N(\boldsymbol{0},\U^*\boldsymbol{\Lambda}_{\Ca}^*\U^{*\top}),
    $$
    where $\bSigma^*=\U^*\boldsymbol{\Lambda}^*\U^{*\top}$ is the eigen-decomposition, $\boldsymbol{\Lambda}^*=\operatorname{diag}\{\lambda_1,\ldots,\lambda_d\}$, and $$\boldsymbol{\Lambda}_{\Ca}^*=\operatorname{diag}\l\{\frac{\sigma_*^2}{\lambda_i}\cdot\frac{(\Ca\lambda_i/\lambda_{\min})^2}{2\Ca\lambda_i/\lambda_{\min}-1}\r\}.$$
\end{theorem}

 Theorem~\ref{thm:asympLR} characterizes the asymptotic distribution of SGD iterates under dependent covariates and noise, encompassing the i.i.d. Gaussian case. The limiting covariance reflects an interplay between $\lambda_i$ and $\lambda_{\min}$. Note also that the limit law depends on $\Ca$ but not on $\Cb$: asymptotically, $t+ \Cb d = O(t)$, whereas Theorem~\ref{thm:nonsymLR} requires a sufficiently large $\Cb$ to ensure sharp non-asymptotic bounds and tail probabilities. The bandit setting (Section~\ref{sec:asymptotic}) extends Theorem~\ref{thm:asympLR}; its proof is therefore omitted and we defer our comparisons with existing literature to Section~\ref{sec:literature-bandit}. For related high-dimensional inference results under i.i.d. data and stepsize decaying as $1/t^\alpha$ with $\alpha \in (1/2, 1)$, we refer to \cite{agrawalla2025statistical} and \cite{shao2022berry}.

\section{Sparse Linear Regression with Dependent Data}
\label{sec:sparse}

In this section, we study sparse linear regression in the presence of dependent covariates and noise. The covariates and responses satisfy the same linear model as in Section~\ref{sec:LR}. In contrast to that setting, the unknown parameter vector $\Bbeta^*$ is now assumed to lie in a low-dimensional subspace induced by a sparsity constraint; specifically,
\begin{align*}
Y_t = \X_t^{\top}\Bbeta^* + \xi_t, \quad \text{with} \quad |\operatorname{supp}(\Bbeta^*)| \leq s .
\end{align*}
We impose the following assumption on the covariates, which is weaker than Assumption~\ref{assm:noise} due to the sparsity of $\Bbeta^*$.

\begin{assumption}
	The covariate $\X_t$ satisfies $\EE\{\X_t|\calF_t\}=0$ almost surely. For some $s\leq d$, there exists $\lambda_{\min} >0$ such that 
    $$
    \min_{\calS\subseteq[d] : \, 1\leq|\calS|\leq s}\lambda_{\min}\Big( \EE\big\{ [\X_t]_{\calS}[\X_t]_{\calS}^{\top}\big| \calF_t \big\}\Big)\geq \lambda_{\min},\quad {\rm a.s.},
    $$
	and there exists $\lambda_{\max}>0$ such that
	\begin{align*}
		\max_{\calS\subseteq[d] :\, |\calS|\leq s}\big\|[\X_t]_{\calS}\big|\calF_t \big\|_{\Psi_2}^2\leq\lambda_{\max},\quad {\rm a.s.}.
	\end{align*}
    where $\lambda_{\min},\lambda_{\max}$ are constants. \label{assm:cov-sparse}
\end{assumption}
This assumption differs from the dense setting in Assumption~\ref{assm:cov}: here the restricted eigenvalue condition is imposed only on $s$-dimensional subvectors, making it strictly weaker. Furthermore, Assumption~\ref{assm:cov-sparse} guarantees the existence of finite constants $\OffDiagS,\OffDiagO\geq 0$ such that
\begin{align*}
	\max_{\calS'\subseteq[d]: |\calS'|\leq s}\;\max_{\calS\subseteq[d]: |\calS|\leq s}\|[\X_t]_{\calS'}[\X_t]_{\calS\setminus\calS'}^{\top}|\calF_t\|_{\Psi_1}\leq \OffDiagS,\quad \max_{i\notin \calS^*}\;\max_{\calS\subseteq[d]: |\calS|\leq s}\|[\X_t]_i[\X_t]_{\calS\setminus i}^\top |\calF_t\|_{\Psi_1}\leq\OffDiagO.
\end{align*}
These quantities measure correlations among subsets of covariates. They vanish, for example, when any two entries of the vector $\X_t$ have disjoint supports. For simplicity, we assume $$\|\Bbeta_{0}-\Bbeta^*\|^2\leq C \frac{\sigma^2}{\lambda_{\max}},$$ which can be ensured by an offline initialization or by running a short warm-up SGD phase with either a constant stepsize (as in Theorem~\ref{thm:nonsymLR}) or the $\frac{\Ca}{t+\Cb d}$ stepsize scheme described in footnote~\ref{note1}. 

Our analysis proceeds in three steps. Section~\ref{subsec:FixedSupport} studies the behavior of SGD under a fixed support, highlighting its dependence on $\OffDiagS$ even if the initial support does not include the true support. Section~\ref{subsec:SupportRecovery} introduces a statistic for support recovery and establishes conditions under which missing support elements can be detected. Section~\ref{subsec:SparseSGD} integrates the estimation and support-recovery procedures into a unified algorithm. The key observation is that while  a single data point suffices for estimation updates, accurate support recovery fundamentally requires aggregating historical information.

\subsection{Sparse SGD: Fixed Support}
\label{subsec:FixedSupport}

In this section, we analyze the behavior of sparse SGD when the support is fixed throughout the iterations. Let $\calS_0$ denote the initial support set and $\calS^* := \operatorname{supp}(\Bbeta^*)$ be the true support. Importantly, we impose no structural relationship between $\calS_0$ and $\calS^*$. In particular, neither $|\calS^*| = |\calS_0|$ nor in the inclusion $\calS^* \subseteq \calS_0$ is required. This stands in sharp contrast to settings such as \cite{fan2018statistical}, where the initial support is assumed to contain the true one.

The update rule operates as follows. At each iteration, the algorithm first computes a stochastic gradient using the newly received data point, updates the current estimate accordingly, and then projects the iterate back on the fixed support $\calS_0$ by zeroing out all other coordinates. Formally, the update is given by
\begin{align}
	\Bbeta_{t+1}=\calH_{\calS_0}\big( \Bbeta_{t} - \eta_{t}\cdot \g_t \big),
	\label{alg:sparse}
\end{align}
where $\g_t:=(\X_t^{\top}\Bbeta_{t}-Y_t)\X_t$ is the instantaneous gradient and $\calH_{\calS_0}$ denotes the hard-thresholding operator that retains only the coordinates in $\calS_0$. Although the support remains fixed, this scheme differs substantially from the dense setting considered in Section~\ref{sec:LR}. Even when $\calS_0$ is misspecified, the estimation error on $\calS_0$ can be affected by components in $\calS_0\setminus\calS^*$ due to correlations among the subvectors $[\X_t]_{\calS_0}$, $[\X_t]_{\calS^*\setminus\calS_0}$, and $[\X_t]_{\calS^*}$. The next lemma makes this dependence explicit and illustrates how misspecification of the support influences convergence.

\begin{lemma}
	Suppose Assumptions~\ref{assm:noise} and \ref{assm:cov-sparse} hold with $s$ where $|\calS_0|\leq s$ and $|\calS^*|\leq s$. Consider the stepsize $\eta_{t}=\frac{1}{\lambda_{\min}}\frac{\Ca}{t+\Cb s\log(2d/s)}$ with $\Ca\geq 2$ and $\Cb\geq C (\Ca \lambda_{\max} / \lambda_{\min})^2$. Then, 
	\begin{align*}
		\|[\Bbeta_{t}-\Bbeta^*]_{\calS_0^c}\|^2=\|[\Bbeta^*]_{\calS^*\setminus\calS_0}\|^2,
	\end{align*}
	and for any tail-probability sequence $\{\delta_{t}\}$ satisfying the requirements in Theorem~\ref{thm:nonsymLR}, with probability at least $1-\sum_{l=0}^{t}\exp(-c\{ s\log(2d/s)+l/\Ca \} )-\sum_{l=0}^{t}\exp(-c\{ s\log(2d/s)+\delta_{l} \})$, we have
	\begin{align*}
		\|[\Bbeta_{t+1}-\Bbeta^*]_{\calS_0}\|^2\leq \frac{C^*}{\lambda_{\min}}\frac{s\log(2d/s)+\delta_{t+1}}{t+1+\Cb s\log(2d/s)} \sigma^2+C\bigg(\frac{\OffDiagS}{\lambda_{\min}}\bigg)^2\cdot\|[\Bbeta^*]_{\calS^*\setminus\calS_0}\|^2.
	\end{align*}
where $C^*=C\Ca^2 \lambda_{\max} /\lambda_{\min} +\Cb \lambda_{\min}/ \lambda_{\max}$.
	\label{lem:FixedSupport}
\end{lemma}

This result highlights two key phenomena. First, the estimation error on $\calS_0$ decays at the minimax rate as in the dense case, if $\calS^*\subseteq\calS_0$. Second, and more importantly, the second term reveals the persistent effect of omitted relevant covariates. When $\calS_0$ fails to contain the full true support, correlations captured by $\OffDiagS$ introduce a nonvanishing bias term that cannot be eliminated through additional iterations. This underscores the necessity of a support-recovery mechanism, developed in Section~\ref{subsec:SupportRecovery}.

\subsection{Sparse SGD: Support Recovery}
\label{subsec:SupportRecovery}

As shown in Lemma~\ref{lem:FixedSupport}, a fixed but misspecified support inevitably induces a persistent bias. It is therefore essential to adaptively recover the support of $\Bbeta^*$. In this section, we develop a support-update mechanism based on a statistic that aggregates the interaction between the covariates and the residuals over a given index set. For any index set $\calT$ (e.g., $\calT=\{1,\ldots,T\}$), define
\begin{align*}
	\G(\{\Bbeta_t,\X_t,Y_t\}_{t\in\calT}):=\sum_{t\in\calT}(\X_t\Bbeta_{t}-Y_t)\X_t .
\end{align*}
The vector $\G(\cdot)$ encodes the alignment between the residuals and the covariates and coincides with the sum of stochastic gradients under squared loss. Importantly, even when other loss functions are used, this quantity remains informative. We will show that $\G$ is particularly effective in identifying coordinates in $\calS^* \setminus \calS_0$, thereby revealing missing components of the true support. To make this role explicit, we expand $\G$ using the model structure:
\begin{align*}
	\G(\{\Bbeta_t,\X_t,Y_t\}_{t\in\calT})= \sum_{t\in\calT} \X_t\X_t^{\top} (\Bbeta_t-\Bbeta^* )-\sum_{t\in\calT} \xi_t\cdot\X_t.
\end{align*}
We say that the sequence $\{\Bbeta_{t}\}_{t\in\calT}$ has a common support $\calS_{\calT}$ if $[\Bbeta_{t}]_{\calS_{\calT}^c}=\boldsymbol{0}$ for all $t \in \calT$. For such a sequence, we introduce the event
$$
\calE_{\{\calS_{\calT},V_{\calT} , W_{\calT} \}}:=\big\{ \|\Bbeta_{t}-\Bbeta^*\|\leq V_t, \, \|[\Bbeta_{t}-\Bbeta^*]_{\calS_{\calT}} \| \leq W_t, ~{\rm for~all}~ t \in \calT \big\} ,
$$
where $V_t$ and $W_t$ are deterministic bounds (specified later) controlling the global and restricted estimation errors along the trajectory. The next lemma establishes lower and upper bounds on the magnitudes of entries of $\G$ when restricted to different subsets of coordinates. These bounds create a separation that enables reliable support detection.

\begin{lemma} 
\label{lem:sparse-gradient}
	Suppose Assumptions~\ref{assm:noise} and \ref{assm:cov-sparse} hold. Let $\{\Bbeta_{t}\}_{t\in\calT}$ be a sequence with common support $\calS_{\calT}$, where $|\calT|\geq Cs$ and $|\calS_{\calT}|\leq s$. Under the event $ \calE_{\{\calS_{\calT},V_\calT,W_\calT\}}$, for any $u,\delta>0$, the following statements hold:
    \begin{itemize}
        \item (Lower bound on uncovered support coordinates) For the coordinates in $\calS^*\setminus\calS_{\calT}$,
        \begin{multline*}
		\|[	\G(\{\Bbeta_t,\X_t,Y_t\}_{t\in\calT})]_{\calS^*\setminus\calS_{\calT}}\|\geq |\calT|\cdot\lambda_{\min}\| [\Bbeta^*]_{\calS^*\setminus\calS_{\calT}}\|-\OffDiagS\sum_{t\in\calT} W_t\\-u-C\sigma\sqrt{|\calT|}\sqrt{\lambda_{\max}}\sqrt{|\calS^*\setminus\calS_{\calT}|\log(2d/|\calS^*\setminus\calS_{\calT}|)+\delta} 
        \end{multline*}
        with probability at least 
        \begin{align*} 
            1 & -\exp\l(Cs-\min\l\{\frac{u^2}{(\OffDiagS)^2\sum_{t\in\calT} W_t^2+|\calT|\lambda_{\max}^2\|[\Bbeta^*]_{\calS^*\setminus\calS_\calT}\|^2},\r.\r.\\
            &~~~~~~~~~~~~~~~~~~~~~~~~~~~~~~~~\l.\l.\frac{u}{\OffDiagS\max_{t\in\calT} W_t+\lambda_{\max}\|[\Bbeta^*]_{\calS^*\setminus\calS_\calT}\|}\r\}\r) \\
            & -\exp\l(-c\min\l\{\delta, \sqrt{\frac{\delta|\calT|}{|\calS^*\setminus\calS_{\calT}|\log(2d/|\calS^*\setminus\calS_{\calT}|)}}\r\}\r).
        \end{align*}

        \item (Upper bound on irrelevant coordinates) For the coordinates in $\calS^{*c}\cap\calS_{\calT}^c$, 
        \begin{align*}
		      \|[	\G(\{\Bbeta_t,\X_t,Y_t\}_{t\in\calT})]_{\calS^{*c}\cap\calS_{\calT}^c}\|_{\infty}\leq  \OffDiagO\sum_{t\in\calT} V_t+u+C\sigma\sqrt{\lambda_{\max}}\sqrt{\delta|\calT|}
	    \end{align*}
        with probability at least 
        $$
            1-d\exp\l(-\min\l\{\frac{u^2}{(\OffDiagO)^2\sum_{t\in\calT}V_t^2},\frac{u}{2\lambda_{\max,1}\max_t V_t}\r\}\r)-d\exp\Big(-\min\big\{\delta, \sqrt{|\calT|\delta} \big\} \Big).
        $$
    \end{itemize}
\end{lemma}

Lemma~\ref{lem:sparse-gradient} shows that the entries of $\G$ corresponding to missing true-support coordinates ($\calS^*\setminus\calS_{\calT}$) grow at order $\Theta(|\calT|)$, whereas those on inactive coordinates ($\calS^{*c}\cap\calS_{\calT}^c$) remain much smaller. This separation is crucial for correctly identifying missing support elements. To guarantee a nontrivial gap between these bounds, and hence ensure detectable separation, we require the following condition on the covariates.

\begin{assumption} 
\label{assm:cov-support}
The covariates $\X_t$ satisfy
$$
    \lambda_{\min}\geq C\sqrt{s}\OffDiagO ~~\mbox{ and }~~ \lambda_{\min}\geq C\OffDiagS ,
$$
where the constants $(\lambda_{\min}, \OffDiagO, \OffDiagS)$ are such that $\displaystyle \min_{\calS\subseteq [d],\,1 \leq |\calS|\leq s}\lambda_{\min} (\EE \{[\X_t]_{\calS}[\X_t]_{\calS}^{\top}|\calF_t \}  )\geq \lambda_{\min}$, $\displaystyle \max_{\calS'\subseteq[d]: |\calS'|\leq s}\;\max_{\calS\subseteq[d]: |\calS|\leq s} \|[\X_t]_{\calS'}[\X_t]_{\calS\setminus\calS'}^{\top}|\calF_t\|_{\Psi_1}\leq \OffDiagS$ and $\displaystyle \max_{i\notin \calS^*}\;\max_{\calS\subseteq[d]: |\calS|\leq s}\|[\X_t]_i[\X_t]_{\calS\setminus i}^\top |\calF_t\|_{\Psi_1}\leq\OffDiagO $.
\end{assumption}

Assumption~\ref{assm:cov-support} requires different entries of $\X_t$ to have correlation bounded by $\lambda_{\min}$, which in fact is equivalent to the conditions in classic offline sparse linear regression \citep{zhang2010nearly}. Lemma~\ref{lem:FixedSupport} provides valid choices for the sequences $\{ V_t \}$ and $\{ W_t \}$ under the update rule \eqref{alg:sparse}. Specifically, 
\begin{align*}
	W_t :=\frac{\sqrt{C^*} }{\sqrt{\lambda_{\min}}}\sqrt{\frac{s\log(2d/s)+\delta_{t}}{t+\Cb s\log(2d/s)}}\sigma+C\bigg(\frac{\OffDiagS}{\lambda_{\min}}\bigg)\cdot\|[\Bbeta^*]_{\calS^*\setminus\calS_{0}}\|
\end{align*}
and 
\begin{align*}
	V_t :=\frac{\sqrt{C^*}}{\sqrt{\lambda_{\min}}}\sqrt{\frac{s\log(2d/s)+\delta_{t}}{t+\Cb s\log(2d/s)}}\sigma+C\bigg(\frac{\OffDiagS}{\lambda_{\min}}\bigg)\cdot\|[\Bbeta^*]_{\calS^*\setminus\calS_{0}}\|+\|[\Bbeta^*]_{\calS^*\setminus\calS_{0}} \|.
\end{align*} 
Substituting $\calS_{\calT} =\calS_{0}$, $\calT=\{0,1,\ldots,t\}$, and the bounds above into Lemma~\ref{lem:sparse-gradient} yields the following corollary.

\begin{corollary} 
Suppose Assumptions \ref{assm:noise}, \ref{assm:cov-sparse}, and \ref{assm:cov-support} hold. Take $\delta_{l}=C\log(l)+\log(d/s)$ and suppose that the estimation error dynamics in Theorem~\ref{thm:nonsymLR} holds. Then, for any $t$ satisfying 
$$
    \frac{t}{\log(t)}\geq C|\calS^*\setminus\calS_0|\log\l(\frac{2d}{|\calS^*\setminus\calS_0|}\r)\cdot\frac{\sigma^2}{\lambda_{\min}\|[\Bbeta^*]_{\calS^*\setminus\calS_0}\|^2},
$$ 
with probability at least $1-C (t d/s)^{-100}$,
	\begin{align*}
		\l\|[	\G(\{\Bbeta_t,\X_t,Y_t\}_{t\in\calT})]_{\calS^*\setminus\calS_{0}}\r\|> \sqrt{|\calS^*\setminus\calS_0|}\cdot\l\|[	\G(\{\Bbeta_t,\X_t,Y_t\}_{t\in\calT})]_{\calS^{*c}\cap\calS_{0}^c}\r\|_{\infty},
	\end{align*}
	which in turn implies that ${\rm argmax}_{i\in\calS_{\calT}^c}\big| [\G(\{\Bbeta_t,\X_t,Y_t\}_{t\in\calT}) ]_{i}\big|\in\calS^*\setminus\calS_{0}$.	\label{cor:variable-selection}
\end{corollary}

Corollary~\ref{cor:variable-selection} shows that once $t$ is sufficiently large, the statistic $\G$ reliably distinguishes the missing coordinates of the true support, enabling accurate support recovery. We now proceed to integrate this selection mechanism with the estimation updates.

\subsection{Sparse SGD: Estimation and Support Recovery}
\label{subsec:SparseSGD}

Intuitively, each incoming observation contributes to reducing the estimation error, but a single observations is insufficient for reliable variable selection. In contrast, aggregating information across historical data, via the summary statistic $\G$, enables us to identify missing support elements once the sample size is large enough. This motivates an iterative procedure that alternates between (i) estimation on a fixed support and (ii) support expansion using $\G$. We present below an algorithm that integrates estimation and support recovery in an online fashion.

\vspace{0.2cm}

\noindent
\textbf{Algorithm.} Over times $t\in\{0,1,\ldots,T\}$, where $T$ may be $+\infty$, choose update times $\tau_0=0$, $0<\tau_1<\cdots<\tau_{\alpha}$. During each interval $[\tau_{l-1},\tau_{l}-1]$, the iterates $\{\Bbeta_{t}\}$ share a common support $\calS_l$ and evolve according to the sparse update rule 
$$
    \Bbeta_{t+1}=\calH_{\calS_l} (\Bbeta_{t}-\eta_t\cdot\g_t ),
$$
where $\g_t=(\X_t^{\top}\Bbeta_{t}-Y_t)\X_t$ denotes the stochastic gradient and $\calH_{\calS_l}$ is the hard-thresholding operator projecting onto support $\calS_l$. The sequence of supports satisfies $\calS_{0}\subseteq \calS_1\subseteq \cdots\subseteq \calS_\alpha$, and the update at stage $l$ is performed via 
$$
    \calS_l=\calS_{l-1}\cup \big\{ {\rm argmax}_{i\in\calS_{l-1}^c}\big| [\G(\{\Bbeta_t,\X_t,Y_t\}_{t\in\calT_l}) ]_{i}\big| \big\}, 
$$
where $\calT_l \subseteq [t]$ denotes the set of times indices used to compute the statistic $\G$ at step $l$. The choice of $\{ \tau_l \}$ and $\{ \calT_l \}$ is not unique and will be discussed in detail below. We consider two natural choices for the index set $\calT_l$: (1) a local window, $\calT_l:=\{\tau_{l-1},\ldots,\tau_{l}-1\}$, and (2) a cumulative window, $\calT_l:=\{0,1,\ldots,\tau_l-1\}$. From a different perspective, if each update enlarges the support by exactly one element in $\calS^*\setminus\calS_{0}$, then full recovery of the true support requires at most $|\calS^* \setminus \calS_{0}|$ times. The following proposition characterizes the estimation error dynamics under this integrated update scheme.

\begin{proposition}
	Suppose Assumptions~\ref{assm:noise}, \ref{assm:cov-sparse}, and \ref{assm:cov-support} hold. Let the stepsize scheme be $\eta_{t}=\frac{\Ca/\lambda_{\min}}{t+\Cb s\log(2d/s)}$, with constants satisfying $\Ca\geq 5$ and $\Cb\geq C (\Ca \lambda_{\max} / \lambda_{\min} )^2$. Assume also that $\calS_l\setminus\calS_{l-1} \subseteq  \calS^*$ for all $l\in\{1,2,\ldots,\alpha\}$ and that the tail probability sequence $\{\delta_l\}$ satisfies the conditions in Theorem~\ref{thm:nonsymLR}. Then, for any $t\in[\tau_{l},\tau_{l+1}-1]$, with probability at least $1-\sum_{l=0}^{t}\exp(-c \{ s\log(2d/s)+l/\Ca \} )- \sum_{l=0}^{t}\exp(-c \{ s\log(2d/s)+\delta_l \} )$, the following bound holds:
	\begin{align*}
		\|[\Bbeta_{t+1}-\Bbeta^*]_{\calS_l}\|^2 
        & \leq C\frac{s\log(2d/s) +\delta_{t+1}}{t+1+\Cb s\log(2d/s)}\frac{\sigma^2}{\lambda_{\min}}+C\bigg(\frac{\OffDiagS}{\lambda_{\min}}\bigg)^2\|[\Bbeta^*]_{\calS^*\setminus\calS_l}\|^2 \\
		& ~~~~~~~~~~~~~~~~~~~+C\sum_{i=1}^{l}\bigg(\frac{\tau_i+\Cb s\log(2d/s)}{t+1+\Cb s\log(2d/s)}\bigg)^{ \Ca- 2}\cdot\|[\Bbeta^*]_{\calS_i\setminus\calS_{i-1}} \|^2.
	\end{align*}\label{prop:sparse-select}
\end{proposition}

Proposition~\ref{prop:sparse-select} quantifies the interplay between support recovery and estimation accuracy. The second term in the bound disappears if either $\OffDiagS=0$ or the current support already contains the true support, i.e., $\calS^*\subseteq\calS_{t}$. The third term reveals a key design principle: early and accurate inclusion of missing support coordinates is crucial for minimizing overall estimation error. However, Corollary~\ref{cor:variable-selection} shows that accurate support detection requires a sufficiently large sample size within the window $\calT_l$. Thus, the timing of support updates must balance two competing considerations. Updating too late slows convergence because unselected true coordinates remain uncorrected, whereas updating too early risks selecting incorrect variables before the statistic $\G$ has achieved reliable separation. By combining Proposition~\ref{prop:sparse-select} with Lemma~\ref{lem:sparse-gradient}, we now obtain an integrated guarantee for simultaneous estimation and support recovery.

\begin{theorem}
\label{thm:sparse-est-supprecv}
	Suppose Assumptions~\ref{assm:noise}, \ref{assm:cov-sparse}, and \ref{assm:cov-support} hold. Let the stepsize scheme be $\eta_{t}=\frac{\Ca/\lambda_{\min}}{t+\Cb s\log(2d/s)}$ with $\Ca\geq 5$ and $\Cb\geq C(\Ca \lambda_{\max} / \lambda_{\min} )^2$. Assume the initial estimate satisfies $\|\Bbeta_{0}-\Bbeta^*\|^2\leq \sigma^2/\lambda_{\min}$, and let $s_+ := |\calS^*\setminus\calS_0|$. Suppose the support update times $\{\tau_l\}$ and index sets $\{\calT_l\}$ satisfy:
	\begin{enumerate}
		\item $\tau_1/\log(\tau_1)\geq Cs_+\log(d/s_+)\cdot\frac{\sigma^2}{\lambda_{\min}\|[\Bbeta^*]_{\calS^*\setminus\calS_0}\|^2}$;
        
		\item For each $l\in  \{ 2,\ldots,s_+ \}$, let $s_l:=|\calS^*\setminus\calS_l|$.
        \begin{enumerate}
            \item If $\calT_l = \{\tau_{l-1},\ldots,\tau_l-1\}$ (local window), then require  
            $$
                \tau_{l}\geq \max\l\{\l(\frac{C}{\tau_{l-1}^{\Ca-4}\|[\Bbeta^*]_{\calS^*\setminus\calS_{l}}\|^2}\sum_{i=1}^{l-1}\tau_i^{\Ca-2}\|[\Bbeta^*]_{\calS_{i}\setminus\calS_{i-1}}\|^2\r)^{\frac{1}{2}},\tau_{l-1}+\frac{Cs_l\log(2d/s_l)\cdot\sigma^2}{\lambda_{\min}\|[\Bbeta^*]_{\calS^*\setminus\calS_l}\|^2}\r\},
            $$
        
            \item If $\calT_l =\{0,1,\ldots,\tau_l-1\}$ (cumulative window), then require
            $$
                \tau_{l}\geq \max\l\{\frac{C}{\|[\Bbeta^*]_{\calS^*\setminus\calS_{l}}\|}\sum_{i=1}^{l-1}\tau_i\|[\Bbeta^*]_{\calS_{i}\setminus\calS_{i-1}}\|,\frac{Cs_l\log(2d/s_l)\cdot\sigma^2}{\lambda_{\min}\|[\Bbeta^*]_{\calS^*\setminus\calS_l}\|^2}\r\} .
            $$
        \end{enumerate}
	\end{enumerate}
Then, with probability at least $1-C d^{-100}$, for all $t\in[ \tau_{s_0}, +\infty]$, we have $\calS^* \subseteq \calS_t$, and the estimation error satisfies
\begin{align*}
		\|\Bbeta_{t+1}-\Bbeta^*\|^2 & \leq C^*\frac{s\log(2d/s)}{t+1+\Cb s\log(2d/s)}\frac{\sigma^2}{\lambda_{\min}}\\
	& ~~~~ +C\sum_{i=1}^{s_0}\bigg(\frac{\tau_i+\Cb s\log(2d/s)}{t+1+\Cb s\log(2d/s)}\bigg)^{\Ca-2}\|[\Bbeta^*]_{\calS_i\setminus\calS_{i-1}} \|^2.
\end{align*}
\end{theorem}

Theorem~\ref{thm:sparse-est-supprecv} shows once the support is fully recovered, the estimation error decays at the statistically optimal rate. The contributions from earlier support-mismatch stages diminish rapidly owing to the exponent $C_a - 2$, and the long-term performance matches the minimax-optimal rate for sparse linear regression under dependent data. In practice, one may update the support whenever the statistic $\G$ exhibits clear $O(t)$ growth on coordinates outside the current support, as suggested by Corollary~\ref{cor:variable-selection}. Theorem~\ref{thm:sparse-est-supprecv} is numerically verified in Section~\ref{sec:numeric}. To this end, our procedure provides a computationally efficient alternative to existing approaches such as online LASSO algorithms \citep{yang2023online, han2024online}, which require solving a regularized optimization problem and maintaining $O(d^2)$ summary statistics at each step.

\section{Contextual Linear Bandit with Dependent Data}
\label{sec:bandit}

We now turn to the online contextual linear bandit problem. Relative to the regression setting, this problem introduces an additional and fundamental source of dependence: beyond the inherent dependence in the data stream $(\X_t, \xi_t)$, the sequence of observations is now adaptively generated through the decision-making policy. Each action taken at time $t$ depends on the current estimate, which in turn depends on all past outcomes. This feedback loop creates a twofold dependence structure, temporal dependence in the covariates and noise, and adaptive dependence induced by exploration-exploitation decisions. Our goal in this section is to show that the proposed SGD framework is sufficiently robust to handle this full dependence structure. In particular, the same techniques developed for dependent linear regression can be integrated with an $\varepsilon$-greedy exploration strategy to achieve simultaneously (i) statistically optimal estimation error for each arm and (ii) statistically optimal regret. This extends classical analyses for i.i.d. covariates and noise \citep{goldenshluger2013linear, gao2019batched, bastani2020online, bastani2021mostly, chen2021statistical, chen2022online, duan2024online} to the substantially more general dependent-data setting.

At each time $t$, the decision-maker observes a covariate vector $\X_t$ and must select one of the $K$ arms. Each arm $i \in [K]$ is parameterized by an unknown vector $\Bbeta_i^*$. Pulling arm $i$ yields the reward
$Y_t=\X_t^{\top} \Bbeta_{i}^* + \xi_t$, where $\xi_t$ denotes the noise term, which is allowed to be dependent. Let $a_t$ denote the arm selected at time $t$. Over a (possibly infinite) horizon $T$, the goal is to maximize the cumulative reward, or equivalently, to minimize the regret
\begin{align}
	\textsf{Regret}(T):=\EE\l\{\sum_{t=1}^T \l(\max_{i\in[K]} \X_t^{\top}\Bbeta_{i}^* - \X_t^{\top}\Bbeta_{a_t}^*\r)\r\}.
	\label{eq:regret}
\end{align}
Equivalently, the objective is to learn a decision rule for $a_t$ that minimizes regret, or to identify the oracle decision region that maps $\X_t$ to the optimal arm. Although this objective differs from minimizing estimation error at a single time point (as in Section~\ref{sec:LR}), we will show that our approach achieves both goals: it yields statistically optimal estimation error for each arm and statistically optimal regret. To accomplish this, we consider a natural integration of SGD with an $\varepsilon$-greedy exploration mechanism. The algorithm is given below.

\noindent\rule{\linewidth}{0.4pt}
\paragraph{Algorithm}  
\vspace{-.3cm}
Let $\{\Bbeta_{i}^{(0)}\}$ be arbitrary initial estimates, and let $\{\eta_{t}\}$ and $\{ \pi_t \}$ denote the stepsize and exploration probability sequences, respectively. At time $t$, we observe the covariate $\X_t$ and draw $\alpha_t \sim \mathrm{Bernoulli}(\pi_t)$. If $\alpha_t=0$ (the exploitation step), we pull the arm that maximizes the estimated reward, 
$$ 
    a_t\in \arg\max_{i\in[K]} \X_t^{\top}\Bbeta_i^{(t)}.
$$ 
If $\alpha_t=1$ (the exploration step), we select an arm uniformly at random: 
$$
    a_t\sim\mathrm{Unif}\{1,2,\ldots,K\}.
$$ 
After pulling arm , we observe the reward $Y_t$, and update only the selected arm using SGD: 
\begin{align}
    \Bbeta_{i}^{(t+1)} = \Bbeta_{i}^{(t)} - \eta_t\ \II_{\{a_t=i\}}\cdot (\X_t^{\top}\Bbeta_{i}^{(t)} - Y_t)\ \X_t,\quad i\in[K].
    \label{alg:LBD}
\end{align}
Remark that $\{\pi_t\}$ and $\{\eta_{t}\}$ are to be specified.

\noindent\rule{\linewidth}{0.4pt}
\medskip

Denote $\calF_t=\sigma(Y_{t-1},\alpha_t,\X_{t-1},\ldots,Y_0,\alpha_0,\X_0)$. Clearly, $\Bbeta_{i}^{(t)}\in\calF_t$. Define the augmented $\sigma$-field $\calF_t^+=\sigma(\X_t,Y_{t-1},\alpha_t,\X_{t-1},\ldots,Y_0,\alpha_0,\X_0) $, so that $\calF_t\subseteq\calF_t^+\subseteq\calF_{t+1}$. At time $t$, with probability $\pi_t$, the algorithm selects an arm uniformly at random from $[K]$; this constitutes the exploration step and is used to acquire information about all arms. With probability $1-\pi_t$, the algorithm selects the arm that maximizes the estimated reward based on past data; this is the exploitation step. The exploration rate $\pi_t$ governs the fundamental trade-off: full exploration yields the smallest estimation error but incurs large regret, whereas pure exploitation risks committing to incorrect arms. Consequently, $\pi_t$ influences both the estimation accuracy of each $\Bbeta_i^*$ and the cumulative regret. In this work, we allow for a broad class of exploration schedules, including functionally defined choices of $\pi_t$ as formalized in Definition~\ref{def:exploration rate}. The stepsize sequence $\{ \eta_t \}$ serves as an additional algorithmic parameter to be specified. 

\begin{remark}[$K=2$ versus $K\geq 2$]
    Throughout, we consider $K\geq 2$. Importantly, the multi-arm case $K>2$ is not a trivial extension of the two-arm setting. When $K=2$, each arm is optimal for a nonempty region of covariates. In contrast, when $K\geq 3$, some arms may be globally sub-optimal, never achieving maximal reward for any covariate. For instance, let $\Bbeta_1^*= (1, 0)^{\top}$, $\Bbeta_2^*= (0.1, 0)^{\top}$, $\Bbeta_3^*= (-1, 0)^{\top}$, $\Bbeta_4^*= (0, 1)^{\top}$, $\Bbeta_5^*=(0, -1)^{\top}$. Then for every $\X \in \RR^2$, arm 2 is never optimal. Such sub-optimal arms introduce additional complexity that does not arise when $K=2$ \citep{chen2022online,chen2021statistical}.
\end{remark}

\subsection{Decision Region Characterizations}

In this section, we introduce a conic approximation of the arm-specific decision regions. For arm $i$, the oracle decision region is defined as 
$$
    \calU_i^*=\bigg\{\X\in\RR^d:\ \X^{\top}\Bbeta_{i}^*> \max_{j\neq i}\X^{\top}\Bbeta_{j}^*\bigg\}.
$$ 
It is straightforward to verify that $\calU_i^*$ is either a cone or an empty set. Indeed, if $\x\in \calU_i^*$, then for any positive scalar $a$, we have $a\x\in \calU_i^*$. To minimize the regret, we need to estimate and approximate the oracle region $\calU_i^*$. A line of influential studies \citep{goldenshluger2013linear, wang2018minimax, bastani2020online, hao2020high, chen2020statistical, bastani2021mostly} assume that the covariates lie in a bounded domain and approximate $\calU_i^*$ using bounded polyhedral regions constructed from separating hyperplanes. Specifically, they consider
\begin{align}
	\widetilde{\calU}_i(h;D):=\bigg\{\|\X\|\leq D:\ \X^{\top}\Bbeta_{i}^*- \max_{j\neq i}\X^{\top}\Bbeta_{j}^*\geq h\bigg\}.
	\label{eq:tilde ui}
\end{align}

The set $\widetilde{\calU}_i(h;D)$ excludes the entire neighborhood of the origin. To address this limitation, we propose an alternative characterization, which we term the conic approximation of the decision region:
\begin{align}
	\calU_i(h):=\bigg\{\X\in\RR^{d}:\ \frac{ \X^{\top}\Bbeta_{i}^*- \max_{j\neq i}\X^{\top}\Bbeta_{j}^*}{\|\X\|}\geq h\bigg\}.
	\label{eq:ui}
\end{align}
The region $\mathcal{U}_i(h)$ is again either conic or empty, mirroring the geometric structure of $\mathcal{U}_i^*$. Moreover, the family $\mathcal{U}_i(h)$ naturally accommodates unbounded covariate support and retains coverage of neighborhoods near the origin. This feature is particularly advantageous in settings where the covariate distribution is highly concentrated around the origin, in which case $\widetilde{\mathcal{U}}_i(h; D)$ may discard substantial probability mass by removing the entire origin. In contrast, when covariates take values from a discrete set, the two constructions $\mathcal{U}_i(h)$ and $\widetilde{\mathcal{U}}_i(h; D)$ coincide. To illustrate the geometric differences, Figure~\ref{fig:decision_region} depicts the bounded approximation $\widetilde{\mathcal{U}}_i(h; D)$ and the conic approximation $\mathcal{U}_i(h)$. The next lemma formally characterizes the relationships between $\mathcal{U}_i^*$, $\mathcal{U}_i(h)$, and $\widetilde{\mathcal{U}}_i(h; D)$, and establishes several key geometric properties of the proposed conic region.
        
        \begin{figure}
            \centering
            \begin{subfigure}[b]{0.45\textwidth}
                \centering
                 \includegraphics[width=\linewidth]{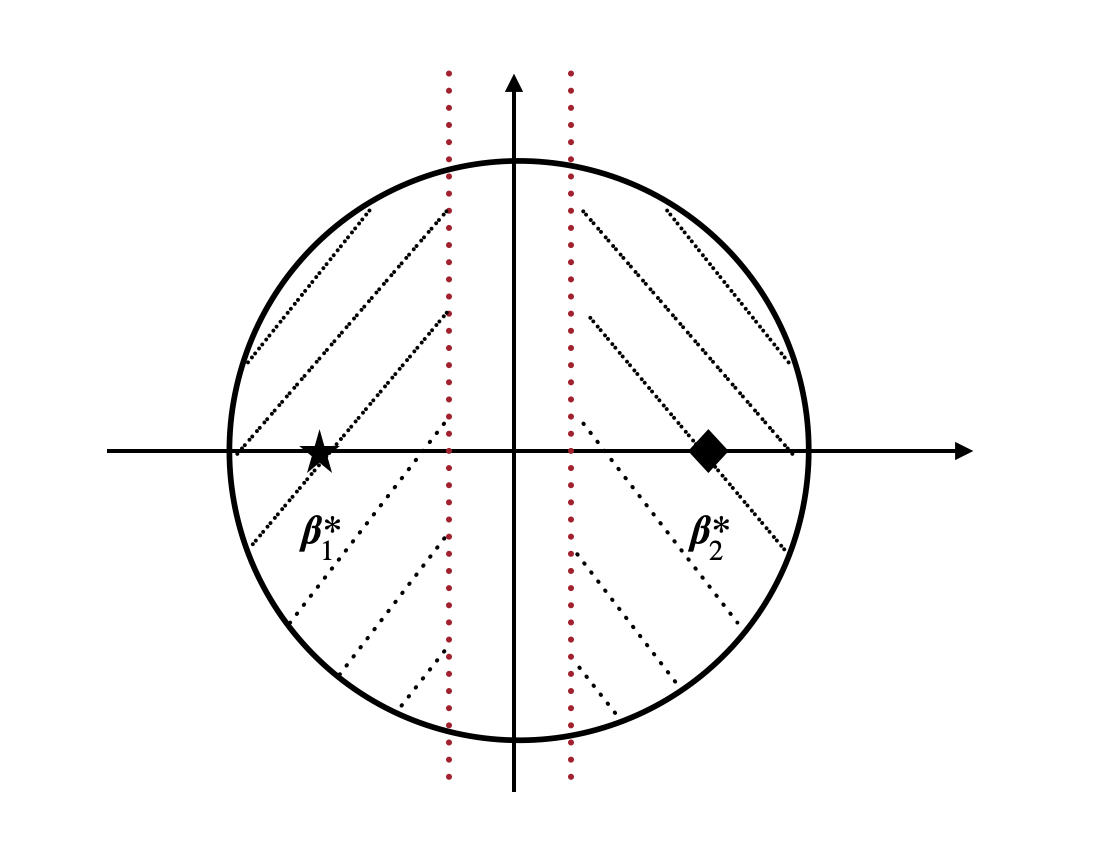}
            \caption{Region of $\widetilde{\calU}_i$}
            \label{fig:tilde_Ui}
            \end{subfigure}
            \hfill
           \begin{subfigure}[b]{0.45\textwidth}
               \centering
               \includegraphics[width=\linewidth]{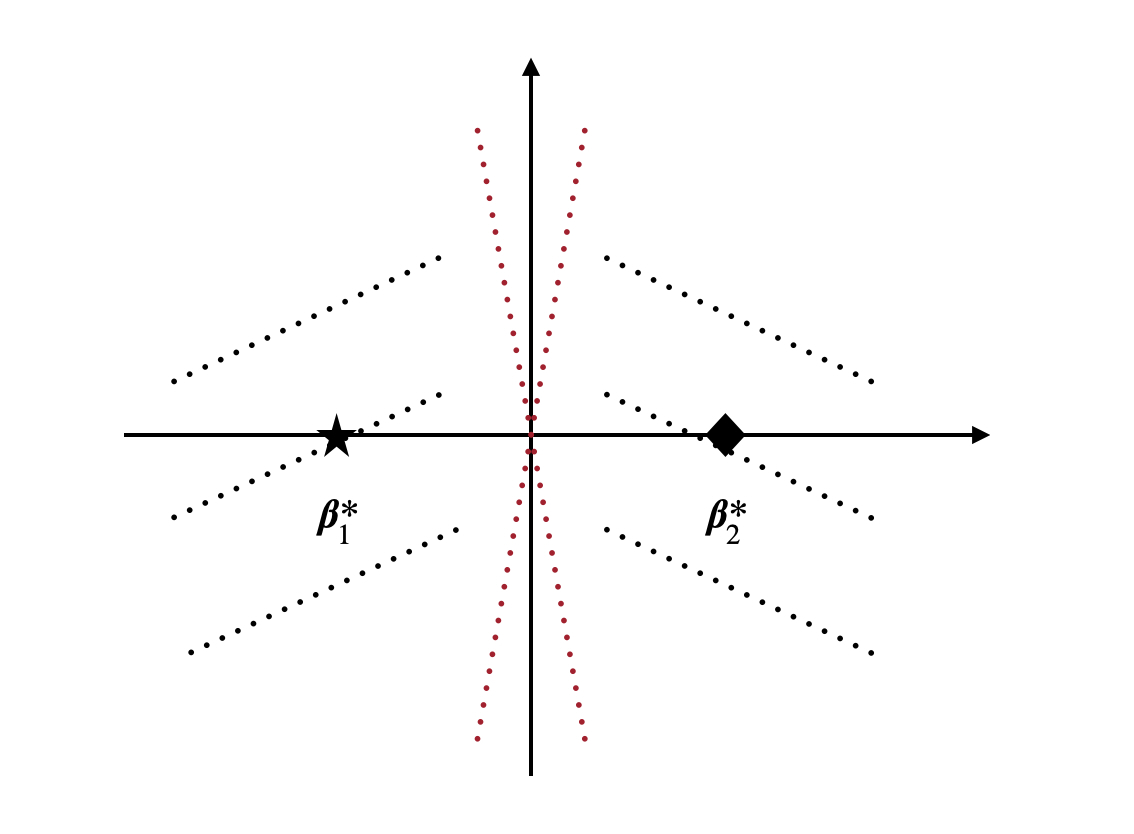}
               \caption{Region of $\calU_i$}
               \label{fig:Ui}
           \end{subfigure}
           \caption{Two approximations of the oracle decision region  $\calU_i^*$. The left panel illustrates the bounded polyhedral approximation $\widetilde{\mathcal{U}}_i$, shown as the shaded region enclosed by several line segments. The right panel shows the conic approximation $\mathcal{U}_i$, which forms an unbounded, scale-invariant region.}
           \label{fig:decision_region}
        \end{figure}

\begin{lemma} \label{lem:region Ui}
	 For any $h>0$ and $D>0$, the following set relations hold:
	\begin{enumerate}
		\item[(i)] For $\widetilde{\calU}_i(h;D)$, it has $\widetilde{\calU}_i(h;D)\subsetneq \calU_i(h/D)\cap  \{\X:\|\X\|\leq D \}\subsetneq \calU_i^*\cap \{\X:\|\X\|\leq D \}$.
        
		\item[(ii)] Conversely, it has $\calU_i(h)\cap  \{\X:\|\X\|\leq D \}\nsubseteq \widetilde{\calU}_i(hD;D)$.
        
		\item[(iii)] Moreover, $\calU_i(h)\subseteq\calU_i^*$, and in particular $\calU_i(0)$ is the closure of $\calU_i^*$. For any $h_1>h_2$, it has $\calU_i(h_1)\subseteq\calU_i(h_2)$.
	\end{enumerate}
Furthermore, if $\calU_i^*\neq\emptyset$, then for any 
$$
    h\in \bigg( 0, \, \min_{i\in\calA}\sup_{\|\z\|=1}(\z^{\top}\Bbeta_{i}^*-\max_{j\neq i}\z^{\top}\Bbeta_{j}^*) \bigg),
$$
we have $\calU_i(h)\neq \emptyset$. Moreover, if $\calU_j^*=\emptyset$, then
\begin{align*}
	\X^{\top}\Bbeta_{j}^*\leq \max_{i\neq j}\X^{\top}\Bbeta_i^* ~~\text{for all }~ \X\in\RR^d.
\end{align*}
\end{lemma}

\begin{proof}
	We prove only part (i). For brevity, write $\calX:=\l\{\X:\|\X\|\leq D\r\}$. Clearly, $\widetilde{\calU}_i(h;D)\subsetneq\calX$. We show $\widetilde{\calU}_i(h;D)\subsetneq \calU_i(h/D)$. For any $\x\in\widetilde{\calU}_i(h;D)$, $\x^{\top}\Bbeta_{i}^*-\max_{j\neq i} \x^{\top}\Bbeta_{j}^*\geq h\geq \frac{h}{D}\|\x\|$, implying $\widetilde{\calU}_i(h;D)\subset \calU_i(h/D)$. Conversely, the vector $\frac{h}{2(\x^{\top}\Bbeta_{i}^*-\max_{j\neq i}\x^{\top}\Bbeta_{j}^*)}\cdot \x $ lies in $\calU_i(h/D)$ but not in $ \widetilde{\calU}_i(h;D)$, completing the proof that $\widetilde{\calU}_i(h;\calX)$ is strictly contained in $\calX\cap\calU_i(h)$.
\end{proof}

Next, let ${\boldsymbol{\beta}_i}$ denote estimators of ${\boldsymbol{\beta}_i^*}$. The corresponding empirical decision region in \eqref{alg:LBD} is
\begin{align}
	\calX\l(i,\{\Bbeta_{j}\}_{j\in[K]}\r):=\bigg\{\X\in\RR^d:\; \X^{\top}\Bbeta_{i}>\max_{j\neq i}\X^{\top}\Bbeta_{j} \bigg\}.
\end{align}
During exploitation steps, the algorithm selects arm $i$ whenever $\X_t\in\calX (i,\{\Bbeta_{j}^{(t)}\}_{j\in[K]} )$. Importantly, $\calX (i,\{\Bbeta_{j}\}_{j\in[K]} )$ is a conic region: if $\mathbf{X}$ belongs to the set, then for any $a>0$, $a\mathbf{X}$ also belongs to it. The next result formalizes the relationship between the empirical region and the oracle approximation $\mathcal{U}_i$.

\begin{lemma}
\label{lem:decision region}
    If $\| \Bbeta_{i} - \Bbeta_{i}^* \| \le h_0$ holds for all $i\in[K]$, then we have
	$$
        \calU_i(2h_0)\subseteq\calX\Big(i,\l\{\Bbeta_{j}\r\}_{j\in[K]}\Big)\subseteq \calU_i(-2h_0).
    $$
\end{lemma}

Lemma~\ref{lem:decision region} links the empirical decision regions used in the update rule~\eqref{alg:LBD} to the conic approximations of the oracle regions $\calU_i^*$, with an explicit and interpretable dependence on the estimation error $\|\Bbeta_{i}-\Bbeta_{i}^* \|$. This result is a key structural component of both our finite-sample and asymptotic analyses. It ensures that sufficiently accurate parameter estimates yield decision regions that correctly approximate the true optimal regions, enabling control of both regret and estimation error within a unified framework.

The classical notion of separability between optimal and suboptimal arms in contextual bandits \citep{bastani2020online, bastani2021mostly} is expressed through bounded-domain approximations such as $\widetilde{\calU}_i$. Our next assumption adapts this separability condition to the scale-invariant regions $\calU_i(h)$.

\begin{assumption}[Arm Optimality]
\label{assm:arm opt}
There exists $h>0$ and a partition $\calA\cup\calA^c=[K]$ such that for all $t$ and $i \in \calA$, $\EE (\X_t \X_t^{\top}\cdot\II\{\X_t\in\calU_i(h) \} \,|\, \calF_t )\succeq \lambda_{\min}\cdot\I_{d}$, and for every $j\in\calA^c$,
$$
    \PP\bigg(\X_t:\; \frac{\max_{i\in [K]}\X_t^{\top}\Bbeta_{i}^*-\X_t^{\top}\Bbeta_{j}^*}{\|\X_t\|} \geq h\bigg| \calF_t\bigg)=1 .
$$	
\end{assumption}

\medskip
\begin{remark}[On the Characterization of Optimal Regions]
Assumption~\ref{assm:arm opt} is formulated using our scale-invariant definition of the optimal decision region, $\calU_i(h)$ in \eqref{eq:ui}. This represents a subtle but important departure from conventional formulations in the contextual bandit literature. For example, influential works such as \cite{bastani2020online} and \cite{bastani2021mostly} employ a bounded-domain characterization, which corresponds to our construction of $\widetilde{\calU}_i(h;D)$ in \eqref{eq:tilde ui}. While such bounded-domain approximations are appropriate when covariates lie within a compact set, they introduce two limitations. First, they require prior knowledge of a bounded support, an assumption that may be violated in many modern applications where covariates are naturally unbounded. Second, they exclude an entire neighborhood around the origin. This exclusion can be problematic for well-behaved designs, such as multivariate Gaussian covariates, whose probability mass is heavily concentrated near the origin. In these settings, removing the region near zero may eliminate a substantial portion of the informative design space, thereby reducing applicability or statistical efficiency. Our definition of $\calU_i(h)$ circumvents these issues. By construction, $\calU_i(h)$ is scale-invariant (i.e., a cone), making it well suited for unbounded covariate distributions while retaining the essential geometric properties of the oracle region $\calU_i^*$. Notably, $\calU_i(h)$ involves the origin rather than excluding it, ensuring that no disproportionately large region of high-probability covariate mass is discarded. As a result, Assumption~\ref{assm:arm opt} provides a more general and broadly applicable foundation for contextual bandit analysis, allowing our theoretical guarantees to hold under more general data-generating mechanisms than those accommodated by bounded-domain assumptions.
\end{remark}

Additionally, if Assumption~\ref{assm:arm opt} holds for some $h$, then it automatically holds for any $h'\in(0,h)$. In the special case $K=2$, this assumption is always satisfied and thus requires no further conditions. We now proceed to establish the regularity properties of the problem, which form a key component in analyzing the convergence dynamics of the proposed method. The intuition underlying these properties is partly inspired by \cite{bastani2020online}. Notably, our framework does not require boundedness on $\X$ or the parameters $\Bbeta^*$, nor does it rely on an i.i.d. assumption on $\{ \X_t \}$.

\begin{lemma}[Regularity Properties]
	Suppose Assumption~\ref{assm:arm opt} holds with $h>0$. Then the following properties hold. 
    \begin{enumerate}
	    \item If the estimates satisfy $\{\|\Bbeta_{i}-\Bbeta_{i}^*\|\leq \frac{h}{2},\ i\in[K] \}$, then with probability one, for any $j\in\calA^{c}$, $\calX(j,\{\Bbeta_i\}_{i\in[K]})=\emptyset$, that is,
	\begin{align*}
		{\rm argmax}_{i\in[K]} \X_t^\top\Bbeta_{i}\subseteq \calA,\quad {\rm a.s.}
	\end{align*}
    Moreover, for every $i\in\calA$,
    $$
      \calU_i\bigg(2\max_{j\in\calA}\|\Bbeta_j-\Bbeta_j^*\|\bigg)\subseteq\calX\Big(i,\l\{\Bbeta_{j}\r\}_{j\in[K]} \Big) \subseteq \calU_i\bigg(-2\max_{j\in\calA}\|\Bbeta_j-\Bbeta_j^*\| \bigg).
    $$
    
    \item For any $i\in\calA$ and any $\Bbeta_{i}\in\calF_t$, 
	\begin{multline} \label{eq:lem 1}
		\EE\bigg\{\II \bigg\{\X_t^{\top}\Bbeta_{i}\geq \max_{j} \X_t^{\top}\Bbeta_{j} \bigg\}\cdot \l(\X_t^{\top}\Bbeta_{i}-\X_t^{\top}\Bbeta_{i}^*\r)^2\bigg| \, \calF_t \bigg\}\\ \geq \II\bigg\{\max_{j}\|\Bbeta_{j}-\Bbeta_{j}^*\|\leq\frac{h}{2} \bigg\}\cdot\lambda_{\min}\|\Bbeta_{i}-\Bbeta_{i}^*\|^2. 
	\end{multline}
	\end{enumerate} 
    \label{lem:regularity}
\end{lemma}

Lemma~\ref{lem:regularity} shows that when the estimation error satisfies $\|\Bbeta_i-\Bbeta_i^*\|\leq \frac{h}{2}$, the quantity in \eqref{eq:lem 1} admits a sharp quadratic lower bound in $\| \Bbeta_i - \Bbeta_i^* \|$. Moreover, in contrast to Lemma~\ref{lem:region Ui}, under Assumption~\ref{assm:arm opt} and the same error condition, the approximation of the empirical decision region depends only on the estimation accuracy of the arms in $\calA$. Technically, Lemma~\ref{lem:regularity} plays a central role in enabling the exploitation-phase observations to contribute effectively to parameter refinement. When  $h' < h$, the region $\calU_i(h')$ more closely approximates the oracle region $\calU_i^*$, compared to $\calU_i(h)$, and the union $\cup_{i\in[K]}\calU_i(h')$ covers a larger subset of $\RR^d$. However, the lemma also indicates that attaining a strictly positive quadratic lower bound in \eqref{eq:lem 1} requires controlling the estimation error at a scale proportional to $h'$. Thus, reducing $h$ expands the decision region but simultaneously requires more accurate parameter estimates in order to maintain the curvature lower bound in \eqref{eq:lem 1}. This illustrates an inherent trade-off in selecting $h$: increasing the probability mass $\PP(\calU_i)$ comes at the cost of weakening the guaranteed quadratic lower bound.

\subsection{Non-asymptotic Performances}
\label{sec:LBD-nonasymptotic}

In this section, we investigate the non-asymptotic statistical performance of the proposed method. A notable feature of our approach is that both the estimation error for the parameters associated with arms in $\mathcal{A}$ and the cumulative regret achieve statistically optimal rates, while requiring only minimal storage, specifically, the most recent estimates and the total number of observations collected thus far. We now detail the convergence dynamics underlying this result.

For clarity of exposition, we assume throughout that the initial estimates satisfy $\|\Bbeta_{i}^{(0)}-\Bbeta_{i}^*\|^2\leq C\sigma^2/\lambda_{\max}$. Such an initialization can be obtained using either offline procedures or online methods. For example, as shown in Theorem~\ref{thm:nonsymLR}, SGD with a sufficiently small constant stepsize yields an $O(1)$ estimation error with linear convergence, thus meeting the above requirement. We next introduce a general prescription for the exploration rate $\pi_t$ and the associated sequence of tail probabilities $\{ \delta_t \}$. The following theorem establishes the resulting non-asymptotic estimation error dynamics.

\begin{theorem}[Non-asymptotic Error Dynamics with Tail Level $\delta_t$]
\label{thm:LBD-estimation}
    
Suppose Assumptions~\ref{assm:cov}, \ref{assm:noise}, and \ref{assm:arm opt} hold. Consider the stepsize $\eta_{t}=\frac{1}{\lambda_{\min}}\frac{\Ca}{t+\Cb d}$, and let $\pi\in (0,1]$ be any constant. For any sequence $\{\delta_t\}$ satisfying $(i)$ $0\leq\delta_t\leq\delta_{t+1}$ and $(ii)$ $\delta_{t} \leq Ct/\Cb$, the following statements hold:
\begin{enumerate}[(1)]

\item If $\pi_t\geq\pi$, $\Ca\geq 2K/\pi$, and $\Cb\geq 3\Ca^2(\lambda_{\max}^2/\lambda_{\min}^2)$, then with probability at least $1-K\sum_{l=0}^{t} \exp(-c(d+l/\Ca))-K\sum_{l=0}^{t}\exp(-c(d+\delta_l))$, we have for all $i\in [K]$,
\begin{align} \label{eq1:thm:LBD-estimation}
\|\Bbeta_{i}^{(t)}-\Bbeta_{i}^*\|^2\leq \ \frac{C^*}{\lambda_{\min}}\frac{d+\delta_{t}}{t+\Cb d}  \sigma^2  ~~\text{with}~~ C^*=C\Ca^2\frac{\lambda_{\max}}{\lambda_{\min}}+\Cb\frac{\lambda_{\min}}{\lambda_{\max}} .
\end{align}
        
\item For any $t\geq t_1:=16C^*d\sigma^2/(\lambda_{\min}h^2) -\Cb d$ and any exploration schedule $\pi_t\in[0,1]$, we have with probability at least $1-|\calA|\sum_{l=t_1}^{t+1} \exp(-c(d+l/\Ca))-|\calA|\sum_{l=t_1}^{t+1}\exp(-c(d+\delta_l))$ that for all $i\in \calA$,
\begin{align} 
\label{eq2:thm:LBD-estimation}
	\|\Bbeta_{i}^{(t)}-\Bbeta_{i}^*\|^2\leq \frac{C^{**}}{\lambda_{\min}}\frac{d+\delta_{t}}{t+\Cb d}  \sigma^2 ~~\text{with}~~ C^{**}=C^*+2C\Ca\frac{\lambda_{\max}}{\lambda_{\min}}.
\end{align}
Moreover, with probability at least $1-|\calA^c|\exp(-cd)\cdot\sum_{l=t_1}^{t+1}\II\l\{\pi_l\neq 0\r\}$, we have for all $i\in \calA^c$,
\begin{align}  
\label{eq3:thm:LBD-estimation}
\|\Bbeta_{i}^{(t)}-\Bbeta_{i}^*\|^2\leq\frac{h^2}{4} .
\end{align}\end{enumerate}
\end{theorem}

\medskip
\begin{remark}[On Decaying Exploration and Its Relation to Exploration-Free Bandits]
A key implication of Theorem \ref{thm:LBD-estimation}(2) is that statistically optimal estimation rates for the optimal arms (those in $\calA$) remain attainable even when the exploration rate $\pi_t$ decays to zero. This provides rigorous theoretical justification for adaptive or gradually diminishing exploration schedules, and connects our framework to the emerging literature on \emph{exploration-free} bandit algorithms. In particular, \cite{bastani2021mostly} introduce a class of ``mostly exploration-free'' algorithms that achieve optimal regret under a geometric condition on the covariate distribution, known as \emph{covariate diversity}. This condition ensures that the randomness in the covariates supplies sufficient intrinsic exploration, thereby eliminating the need for explicit forced exploration. Our result offers a complementary and more general viewpoint: rather than relying on geometric assumptions about the covariates, we show that exploration becomes unnecessary once the estimation error satisfies $\max_i\|\Bbeta_i^{(t)}-\Bbeta^*\|\leq h/2$. In other words, explicit exploration is required only until the estimation error falls below a critical threshold; thereafter, the algorithm can operate under a near-greedy policy without compromising long-run performance. This insight highlights an algorithmic mechanism for substantially reducing or eliminating exploration, even in settings where geometric conditions such as covariate diversity fail or are difficult to verify.
\end{remark}

Theorem~\ref{thm:LBD-estimation} characterizes the estimation error rates in the presence of tail probabilities. Part (1) describes the convergence behavior when the exploration rate $\pi_t$ is sufficiently large, guaranteeing statistically optimal estimation. Part (2) extends the analysis to settings where exploration rates are small, possibly decreasing or even vanishing, and establishes that long-term statistically optimal rates are still attainable for arms in $\calA$, while the suboptimal arms maintain uniformly bounded estimation error. This result extends existing literature on estimation error and exploration schedules; a detailed comparison is provided in Section~\ref{sec:literature-bandit}. The following corollary provides a sharper characterization when $\delta_t=C\log t$.

\begin{corollary}[Estimation Error Dynamics --- $\delta_{t}=C\log t$]
Under the same conditions and the same $t_1$ and $\{\eta_{t}\}$ as in Theorem~\ref{thm:LBD-estimation}, we have:
	\begin{enumerate}[(1)]
		\item If $\pi_t\geq\pi$ and $\Ca\geq 2K/\pi$, $\Cb\geq C\Ca^2\lambda_{\max}^2/\lambda_{\min}^2$, then with probability at least $1-C K \Ca e^{- c d }$, $\|\Bbeta_{i}^{(t)}-\Bbeta_{i}^*\|^2 \leq \frac{C^*}{\lambda_{\min}}\frac{d+\log t}{t+\Cb d}\sigma^2$ for all $t$.
        
		\item For any $\pi_t\in[0,1]$ and $t\geq t_1$, with probability at least $1-|\calA| e^{- c d } -|\calA^c| e^{- c d } \cdot\sum_{l=t_1}^{t+1}\II\l\{\pi_l\neq 0\r\}$, we have for all $i\in\calA$, $\|\Bbeta_{i}^{(t)}-\Bbeta_{i}^*\|^2\leq \frac{C^{**}}{\lambda_{\min}}\frac{d+\log t}{t+\Cb d}\sigma^2$, and for all $i\in \calA^c$, $ \|\Bbeta_{i}^{(t)}-\Bbeta_{i}^*\|^2\leq \frac{h^2}{4}$.
	\end{enumerate}
	\label{cor:estimation error1}
\end{corollary}

Corollary~\ref{cor:estimation error1}$(1)$ implies that if $\pi_t\geq \pi$, then for all $T\leq +\infty$, the estimator satisfies, with probability at least $1-CK\Cb\exp(-cd)$,
\begin{align}
	\|\Bbeta_{i}^{(T)}-\Bbeta_{i}^*\|^2\leq C^{*}   \frac{\max\{d,\log T\}}{\lambda_{\min}(T+\Cb d)}  \sigma^2, \quad T=0,1,\cdots,+\infty.
	\label{eq1:cor:estimation error1}
\end{align}
If $T\geq \exp(d)$, the extra $\log T$ term becomes unavoidable, as shown in Section~\ref{sec:LR}. Corollary~\ref{cor:estimation error1}$(2)$ further covers settings where the exploration rate may be arbitrarily small. Under appropriate conditions, e.g., where $\pi_t = 0$ for all $t\geq t_2$ or $\calA^c=\emptyset$, the estimator for arms $i\in\calA$ satisfies the same bound for all $T\leq \infty$.

Before presenting the regret bound, we introduce the following measures of model complexity:
\begin{align*}
	\ComSQ:=\sum_{j=1}^{K}\sum_{i=1}^{K} \|\Bbeta_{i}^*-\Bbeta_{j}^*\|,\quad \ComOne:=\sum_{j=1}^K \max_{i}\|\Bbeta_{i}^*-\Bbeta_{j}^*\|,\quad\ComMax:=\max_{i,j} \|\Bbeta_{i}^*-\Bbeta_{j}^*\|.
\end{align*}
When the pairwise distances among the parameters $\{\Bbeta_i^*\}$ are large, the parameter $h$ in Assumption~\ref{assm:arm opt} can also be chosen large, potentially simplifying the identification of optimal arms. However, incorrect decisions in such settings can incur larger regret. The next theorem formalizes how these complexity measures influence regret. Define the accumulated tail probability
$$ 
	\textsf{Tail}(T,d,\{\delta_{l}\}): = e^{-cd} \bigg(    c'\Cb +K\sum_{l=0}^{t_1} e^{ -c \delta_{l} } +|\calA|\sum_{l=t_1+1}^{T} e^{ -c \delta_{l} } + |\calA^c| \sum_{l=t_1+1}^{T}\II\{\pi_l\neq 0\} \bigg).
$$

\begin{theorem}[Regret Bound under General Learning Rates]
\label{thm:regret}
	Under the same assumptions, stepsize choices, and the same $t_1$ and parameter conditions as in Theorem~\ref{thm:LBD-estimation}, suppose $\pi_t\geq 1/3$ for $t\leq t_1$ and $\pi_t\in[0,1]$ for $t\geq t_1$. Then
	\begin{multline*}
		\textsf{Regret}(T)\leq \underbrace{\sqrt{\lambda_{\max}}\frac{\sum_{t=0}^{T}\pi_t}{K} \min\{\ComSQ, \sqrt{d}\ComOne \}}_{R_1}\\+\underbrace{\sqrt{\lambda_{\max}}\textsf{Tail}(T,d,\{\delta_{t}\})\min\{\ComSQ, \sqrt{d} \ComMax \}}_{R_2}\\+\underbrace{C\sqrt{\frac{\lambda_{\max}}{\lambda_{\min}}} \sigma \min\{ K, \sqrt{d}\}\sum_{t=0}^{T}\sqrt{\frac{d+\delta_{t}}{t+\Cb d}}}_{R_3}.
	\end{multline*}
\end{theorem}

Theorem~\ref{thm:regret} decomposes the cumulative regret into three components. The first term, $R_1$, captures regret due to exploration and scales with $\sum_{t=0}^{T}\pi_t$, which represents the approximate number of exploration iterations. The second term, $R_2$, arises from the failure probability in Theorem~\ref{thm:LBD-estimation} and is typically negligible. The third term, $R_3$, captures the dominant contribution and scales as $\sqrt{T}$. To minimize overall regret, one should therefore keep $\sum_{t=0}^T \pi_t$ as small as possible while still satisfying the required conditions. As discussed in Theorem~\ref{thm:LBD-estimation}, a convenient choice is to set $\pi_t=f(t)$ for some $f$ in the class $\Pi(\tau,\pi)$.

%the approximate total number of exploration iterations. The second term, $R_2$, corresponds to the probability of failure in the convergence result of Theorem~\ref{thm:LBD-estimation} and is proportional to a tail probability, which is typically negligible. The third term, $R_3$, arises from the high-probability convergence case and generally scales as $\sqrt{T}$. Consequently, the overall regret is expected to be dominated by $R_3$. To minimize the total \textsf{Regret}, the quantity $\sum_{t=0}^{T}\pi_t$ should therefore be kept as small as possible while still satisfying the required conditions. As discussed in Theorem~\ref{thm:LBD-estimation}, the selection of $\{\pi_t\}$ can be quite flexible. A convenient exploration rate selection is taking function value, i.e., $\pi_t=f(t)$. 

\begin{definition}
	A function $f(\cdot):\RR^+\to\RR^+$ belongs to $\Pi(\tau, \pi)$ if $(i)$ $f$ is decreasing; $(ii)$ $0\leq f(x)\leq 1$; and $(iii)$ $f(x)\geq \pi$ for all $x\leq \tau$.
	\label{def:exploration rate}
\end{definition}

Since $f$ is decreasing, we have $\sum_{t=1}^{T}\pi_t\leq \int_{0}^{T}f(x)  {\rm d}x$. The following are some examples of the exploration rate $\pi_t$ that satisfy Definition \ref{def:exploration rate}.

\begin{example}[Example Functions satisfying Definition~\ref{def:exploration rate}] 
The following functions belong to $\Pi(\tau,1/3)$.
\begin{enumerate}
		\item Let $ f(x) =\frac{C_{\pi}}{x+2C_{\pi}}$, where $C_{\pi}\geq \tau$. Then $f\in\Pi(\tau,1/3)$ and $\sum_{t=1}^{T}\pi_t\leq C_{\pi}\ln\l(\frac{T+2C_{\pi}}{T}\r)$;

		\item Let $f(x)= (\frac{C_{\pi}}{x+2^{1/p}C_{\pi}}  )^{p}$ with $0< p<1$ and $C_{\pi}\geq 3^{1-1/p}\tau$. Then $f\in\Pi(\tau,1/3)$ and $\sum_{t=1}^{T}\pi_t\leq C_{\pi}^{p} (T+2^{1/p}C_{\pi})^{1-p}$.
	\end{enumerate}
\label{exp:regret}
\end{example}

Substituting $\delta_t=0$ or $\delta_t=\log t$ into Theorem~\ref{thm:regret} gives the following corollary.

\begin{corollary}[Regret: Unaware of $h$] 
\label{cor:regret 1}
Assume the same assumptions and parameters as in Theorem~\ref{thm:regret}. Let $\eta_{t}=\frac{1}{\lambda_{\min}}\frac{\Ca}{t+\Cb d}$ and $\pi_t =f(t)\in\Pi(t_1,1/3)$. If $\calA^c\neq\emptyset$, then for all $T\leq \exp(Cd)$, 
$$
    \textsf{Regret}(T)\leq \sqrt{\lambda_{\max}}\frac{\int_{0}^{T}f(t) {\rm d} t}{K} \min \{\ComSQ, \sqrt{d}\ComOne \} +C\sqrt{C^*}\sqrt{\frac{\lambda_{\max}}{\lambda_{\min}}}   \frac{\sigma \min \{ K, \sqrt{d} \} T\sqrt{d} }{\sqrt{T+\Cb d} + \sqrt{\Cb d}} .
$$
If $\calA^c=\emptyset$, then for all $T\leq +\infty$, 
$$
	\textsf{Regret}(T)\leq \sqrt{\lambda_{\max}}\frac{\int_{0}^{T}f(t) {\rm d} t}{K}  \min \{\ComSQ, \sqrt{d}\ComOne \} +C\sqrt{C^*}\sqrt{\frac{\lambda_{\max}}{\lambda_{\min}}}  \frac{\sigma \min \{ K, \sqrt{d} \}T\sqrt{d+\log T} }{\sqrt{T+\Cb d} + \sqrt{\Cb d}}  .
$$
\end{corollary} 

The cumulative regret bound consists of two main components. The first depends on the horizon $T$ via the integral $\int_0^T f(s){\rm d}s$, determined by the exploration schedule and potentially growing more slowly than $\sqrt{T}$. The second term scales as $\sqrt{T}$. For instance, choosing $\pi_t= (\frac{C_\pi}{t+2^{1/p}C_\pi} )^p$ with $1/2\le p\le 1$ (Example~\ref{exp:regret}) yields
$$
    \textsf{Regret}(T)\leq O(\sqrt{T}).
$$
Thus, Theorems \ref{thm:LBD-estimation} and \ref{thm:regret} together show that optimal regret and estimation error can be achieved simultaneously under a wide class of exploration schedules. Moreover, the dependence on $T$ and $d$ is statistically optimal. In some settings, e.g., as discussed in \cite{bastani2020online}, the parameter $h$ may be treated as known. When $h$ is known, Theorems \ref{thm:LBD-estimation} and \ref{thm:regret} imply that exploration is unnecessary once the estimation error satisfies $\max_{i\in[K]} \| \Bbeta_i^{(t)}-\Bbeta_i^{*} \| \le h/2$. Thus, the exploration rate may be set to $\pi_t=0$ beyond this point, as formalized below.

\begin{corollary}[Regret: Aware of $h$] 
\label{cor:regret 2}
Under the same assumptions and conditions as Theorem~\ref{thm:regret}, take $\eta_t=\frac{1}{\lambda_{\min}}\frac{\Ca}{t+\Cb d}$ and set $\pi_t\ge1/2$ for $t<t_1$ and $\pi_t=0$ for $t\ge t_1$. Then for all $T\le\infty$,
$$
		\textsf{Regret}(T)\leq \sqrt{\lambda_{\max}}\frac{t_1}{K} \min \{\ComSQ, \sqrt{d}\ComOne \}\\
		+C\sqrt{C^*}\sqrt{\frac{\lambda_{\max}}{\lambda_{\min}}}   \frac{\sigma \min \{ K, \sqrt{d} \} T\sqrt{d+\log T} }{\sqrt{T+\Cb d} + \sqrt{\Cb d}}   .
$$
\end{corollary}

Setting $\pi_t=0$ ensures that suboptimal arms in $\calA^c$ are never selected. Compared with Corollary \ref{cor:regret 1}, this result removes the restriction $T\le C\exp(cd)$ even when $\calA^c \neq \emptyset$. The improvement stems from the fact that with $\pi_t=0$, exploration-induced errors no longer accumulate, allowing non-asymptotic guarantees to extend to an infinite horizon.

\subsection{Asymptotic Analysis and Inference}
\label{sec:asymptotic}

In this section, we investigate the asymptotic behavior of the iterates  $\Bbeta_{i}^{(t)}$, from which we derive confidence intervals for $\Bbeta_{i}^*$. There is a growing literature on inference for SGD iterates in contextual linear bandits; see, for example, \cite{chen2021statisticalb}, \cite{han2022online}, \cite{chen2022online}, and \cite{duan2024online}, which study inference under i.i.d. covariates and noise with $1/t^\alpha$ stepsize schemes, constant exploration rates, and related settings. A detailed comparison with these approaches is provided in Section~\ref{sec:literature-bandit}.

Throughout, we maintain the assumptions $\EE\{\xi_t|\calF_t^+\}=0$ and $\EE\{\X_t|\calF_t\}=\boldsymbol{0}$ almost surely. To ensure asymptotic convergence of the SGD iterates, we require the existence of limiting second-moment matrices for both the covariates and the noise.

\begin{assumption}[Covariates and Noise for Asymptotic Convergence]
\label{assm:cov asymp}
There exist matrices $\bSigma^*$ and $\bSigma_{i}^*$ such that
$$
    \displaystyle\lim_{t\to+\infty}\EE\big\{ \|\EE\{\X_t^{\top}\X_t|\calF_t\}-\bSigma^* \|\big\}=0 ~~\mbox{ and }~~  \displaystyle\lim_{t\to+\infty} \EE \big\{ \|\EE \{\X_t\X_t^\top\cdot\II \{\X_t\in\calU_i^* \}|\calF_t \}- \bSigma_i^* \| \big\}=0.
$$ 
In addition, there exist constants $\lambda_{\max},\lambda_{\min}>0$ such that almost surely,
$$
    \lambda_{\min}\I\preceq\EE \{\X_t\X_t^{\top}|\calF_t  \} ~~\mbox{ and }~~ \displaystyle \sup_{\V\in\SS^{d-1}:\,\V\in\calF_t}\EE \{(\X_t^{\top}\V)^4|\calF_t \}\leq\lambda_{\max}^2 .
$$ 
Assume $\xi_t\perp\X_t|\calF_t$. There exist $\sigma,\sigma_*>0$ such that $\displaystyle\lim_{t\to+\infty}\EE \{ |\EE\{\xi_t^2|\calF_t\}-\sigma_*^2 | \}=0$ and $\EE \{\xi_t^4|\calF_t^+ \}\leq \sigma^4$ almost surely.
\end{assumption}

Assumption~\ref{assm:cov asymp} is needed to establish the asymptotic distribution of $\Bbeta_{i}^{(t)}$. However, near the boundary $\partial \calU_i^*$, the behavior $\X_t$ may still cause difficulties. To illustrate, consider the case $\Bbeta_1^* = (3, 0)^\top$ and $\Bbeta_2^* = (-1, 0)^\top$, and suppose
$$
\PP\{ \X = (0, 1)^\top \} = \PP \{ \X = (0, -1)^\top \} = \frac{1}{2} 
\PP\{ \X = (2, 0)^\top \} = \PP \{ \X = (-1, 0)^\top \} = \frac{1}{5}.
$$
Here, $\partial \calU_1^* = \{(0, y)^\top: y \in \RR\}$ is the $y$-axis, but $\PP(\X \in \partial\calU_1^*) = 2/5$. Although our non-asymptotic guarantees in Theorems~\ref{thm:LBD-estimation} and~\ref{thm:regret} continue to hold, the asymptotic behavior depends on sensitivity on the frequency with which $\X_t$ lies near the decision boundary. In particular, Assumption~\ref{assm:cov asymp} alone is insufficient. To address this, we introduce spherical measures that quantify how $\X_t$ spreads on the unit sphere. For any measurable $A \in \sigma(\SS^{d-1})$, define
$$
    \nu_X(A;t):=\EE \big(\|\X_t\|\cdot\II \{ \X_t / \|\X_t\| \in A \} \big|\calF_t \big),\quad \kappa_X(A;t):=\EE \big(\|\X_t\|^2\cdot\II \{ \X_t / \|\X_t\| \in A \}  \big|\calF_t \big).
$$ 
We impose the following mild regularity condition near $\partial\calU_i^*$.

\begin{assumption}
\label{assm:density asymp}
In addition to Assumption~\ref{assm:arm opt}, there exists $h_0\in(0,h]$ and constants $\kappa_0, \nu_0>0$ such that for any $h_1, h_2\in[-h_0,h_0]$, $ \kappa_X(\calU_i(h_1)\backslash\calU_i(h_2))\leq \kappa_0|h_1-h_2|$ and $ \nu_X(\calU_i(h_1)\backslash\calU_i(h_2))\leq\nu_0|h_1-h_2|$.
\end{assumption}

Assumption~\ref{assm:density asymp} includes the conditions of covariates in \cite{chen2021statistical,han2022online}. The following remark provides an explicit example for Assumption~\ref{assm:density asymp}.

\begin{remark}
Assumption~\ref{assm:density asymp} encompasses a broad class of distributions. For simplicity of presentation, we take $\X_t$ to be i.i.d., though the arguments extend to the dependent case. In this remark, we write $\kappa_X(\cdot)$ and $\nu_X(\cdot)$ without the index $t$. Let $\mu$ denote the Lebesgue measure on the sphere $\SS^{d-1}$. By Lebesgue's Decomposition Theorem, the measure $\nu_X$ admits the decomposition
$$
    \nu_X=\nu_X^{(1)}+\nu_X^{(2)},
$$
where $\nu_X^{(1)}\ll \mu$ and $\nu_X^{(2)}\perp \mu$. Radon–Nikodym's theorem then guarantees the existence of the derivative $\frac{{\rm d} \nu_X^{(1)}}{{\rm d} \mu}$, and for any $A\in\sigma(\SS^{d-1})$, 
$$
    \nu_X^{(1)}(A)=\int_{A} \frac{{\rm d} \nu_X^{(1)}}{{\rm d} \mu}  {\rm d} \mu.
$$ 
If $\nu_X^{(2)}\bigl(\calU_i(-h_0)\setminus\calU_i(h_0)\bigr)=0$ and $\frac{d\nu_X^{(1)}}{d\mu}\le \kappa_0$, then Assumption~\ref{assm:density asymp} holds for $\nu_X$. The same reasoning applies to $\kappa_X$. Common distributions such as multivariate Gaussian and multivariate Student's $t$ distributions fall within the scope of Assumption~\ref{assm:density asymp}.
\end{remark}

If Assumption~\ref{assm:arm opt} and Assumption~\ref{assm:density asymp} hold with parameters $h$ and $h_0$, respectively, then they also hold with any smaller $h'<\min\{ h, h_0\}$. Thus, for simplicity, we assume throughout that Assumptions~\ref{assm:arm opt} and \ref{assm:density asymp} hold with a common parameter $h$. We define
\begin{align*}
	\bSigma_{i}(\pi^*):=\frac{\pi^*}{K}\bSigma^*+(1-\pi^*)\bSigma_{i}^*,
\end{align*}
and write its singular value decomposition as $\bSigma_{i}(\pi^*) =\U_{i}(\pi^*)\boldsymbol{\Lambda}_i(\pi^*)\U_{i}(\pi^*)^{\top}$, where $\U_{i}(\pi^*)$ is orthogonal and $\boldsymbol{\Lambda}_i(\pi^*)=\operatorname{diag}\{\lambda_j(\pi^*)\}$. Let $\Ca'=\Ca/\lambda_{\min}$, so that the stepsize is $\eta_{t}=\Ca'/(t+\Cb d)$. We further define
$$
    \boldsymbol{\Lambda}_{i}(\pi^*,\Ca',\sigma_*):=\operatorname{diag} \bigg\{ \frac{\sigma_*^2}{\lambda_j(\pi^*)}\cdot\frac{(\Ca'\lambda_{j}(\pi^*))^2}{2\Ca'\lambda_{j}(\pi^*)-1} \bigg\}.
$$ 
We now present a Bahadur-type representation for the SGD iterates. Note that $\{ X_t\}$ and $\{ \xi_t \}$ form a Markov process rather than an i.i.d. sequence. Throughout, $O_{\PP}(\cdot)$ denotes stochastic boundedness.

\begin{proposition}[Bahadur-Type Representation]
\label{prop:Bahadur}
Suppose Assumptions~\ref{assm:arm opt}, \ref{assm:cov asymp}, and \ref{assm:density asymp} hold for some common $h$. \footnote{We assume a common $h$ for clarity; if Assumption~\ref{assm:arm opt} holds for $h$, it automatically holds for any smaller $h'<h$.} Under the same $\{\eta_{t}\}$ and $\{\pi_t\}$ as in Theorem~\ref{thm:regret}, if  $\lim_{t\to+\infty}\pi_t=\pi^*$, then
	\begin{align*}
		\Bbeta_{i}^{(t)}-\Bbeta_{i}^*=\sum_{j=0}^t \prod_{l=j+1}^{t}\Big\{ \I-\eta_{l}\cdot\Big(\frac{\pi_l}{K}\bSigma+(1-\pi_l)\bSigma_{i}^*\Big)\Big\}  \eta_{j}\cdot\xi_j\cdot\X_j\cdot\II\l\{a_j=i\r\}+\R_{i}^{(t)}.
	\end{align*}
    Moreover, under either of the following conditions:
    \begin{enumerate}
        \item[(1)] $\pi^*=0$, $i\in\calA$ and on the event $\cup_{t=t_1}^{+\infty}\{\max\|\Bbeta_i^{(t)}-\Bbeta_i^*\|\leq \frac{h}{2}\}$;
        
        \item[(2)] $\pi^*>0$, $\Ca\geq 4K/\pi^*$, and $i\in[K]$,
    \end{enumerate}
we have $\R_i^{(t)}=O_{\PP}(1/t)$ and 
\begin{align*}
    & \sqrt{t}\sum_{j=0}^t \prod_{l=j+1}^{t}\Big\{ \I-\eta_{l}\cdot\Big(\frac{\pi_l}{K}\bSigma+(1-\pi_l)\bSigma_{i}(\pi^*)\Big)\Big\} \eta_{j}\xi_j\X_j\cdot\II \{a_j=i \} \\
    & ~~~~~~~~~~~~~~~~~~~~~~~~~~~~~~~~~~~~\rightsquigarrow N(0,\U_i(\pi^*)\boldsymbol{\Lambda}_{i}(\pi^*,\Ca',\sigma_*)\U_i(\pi^*)^{\top}).
\end{align*}
\end{proposition}

The asymptotic distribution notably does not depend on $\Cb$. Proposition~\ref{prop:Bahadur} allows the construction of confidence intervals for $\Bbeta_i^*$. Since $\pi^*$ and $\Ca'$ are known, the remaining task is to estimate $\bSigma$ and $\bSigma_i^*$.

\begin{lemma}[Estimation of $\bSigma$ and $\bSigma_{i}^*$]
\label{lem:asymp estimate cov}
	Under Assumption~\ref{assm:cov asymp}, the estimator
    \begin{align*}
		\widehat{\bSigma}^*(t):=\frac{1}{t}\sum_{l=1}^{t}\X_l\X_l^{\top}  
	\end{align*}
    converges in probability to $\bSigma^*$ under Frobenius norm.\footnote{Any norm equivalent to the Frobenius norm may be used.} Furthermore, if $\Bbeta_i^{(t)}\to\Bbeta_i^*$ in probability, then
    \begin{align*}
        \widehat{\bSigma}_i(t):=\frac{1}{(1-\pi^*)t}\sum_{l=1}^{t} \X_l\X_l^{\top}\cdot\II\{a_l=i,\alpha_l=0\} \to \bSigma_i^*,\quad \widehat{\sigma}_*^2(t):=\frac{1}{t}\sum_{l=1}^t(Y_l-\X_l^{\top}\Bbeta_{a_l})^2\to\sigma_*^2 ,   
    \end{align*}
    both in probability.
\end{lemma}

Since $\pi^*$ and $\Ca'$ are known, we construct
$$
    \widehat{\bSigma}_i(\pi^*,t):=\frac{\pi^*}{K}\widehat{\bSigma}^*(t)+(1-\pi^*)\widehat{\bSigma}_i(t) , 
$$
and let its singular value decomposition be $\widehat{\U}_i(\pi^*,t)\widehat{\boldsymbol{\Lambda}}_i(\pi^*,t)\widehat{\U}_i(\pi^*,t)^{\top}$. By Lemma~\ref{lem:asymp estimate cov} and and the Davis–Kahan theorem \citep{davis1970rotation,yu2015useful},
\begin{align*}
    \widehat{\U}_i(\pi^*,t)\to \U_i(\pi^*),\quad \widehat{\boldsymbol{\Lambda}}_i(\pi^*,t)\to \boldsymbol{\Lambda_i}(\pi^*),
\end{align*}
in probability. Consequently,
$$
    \widehat{\boldsymbol{\Lambda}}_{i}(t):=\operatorname{diag}\bigg\{ \frac{\widehat{\sigma}_*^2(t)}{\widehat{\lambda}_j(\pi^*,t) }\cdot\frac{(\Ca'\widehat{\lambda}_{j}(\pi^*,t))^2}{2\Ca'\widehat{\lambda}_{j}(\pi^*,t)-1} \bigg\}\to \boldsymbol{\Lambda}_{i}(\pi^*,\Ca',\sigma_*)\quad\text{in probability}.
$$
For any set $\calC\subseteq\RR^d$ with probability mass $1-\alpha$ under $N(0,\I)$, 
$$
    \lim_{t\to\infty}\PP\big\{ \sqrt{t} \widehat{\boldsymbol{\Lambda}}_{i}(t)^{-1/2}(\Bbeta_i^{(t)}-\Bbeta_i^*)\in\calC\big\} =1-\alpha,
$$
which yields asymptotically valid confidence intervals for $\Bbeta_i^*$.

\subsection{Related Literature}
\label{sec:literature-bandit}

In this section, we review the literature on linear contextual bandits and related online learning approaches. The pioneering works of \cite{goldenshluger2013linear}, \cite{bastani2020online}, and \cite{bastani2021mostly} estimate the underlying parameters of multi-arm contextual bandits using ordinary least squares (OLS) or LASSO applied to stored historical observations. Their analyses are non-asymptotic and rely on bounded i.i.d. covariates together with i.i.d. noise. These studies impose a margin condition under which the optimal regret rate is $\log(T)$. In contrast, our framework does not require this condition; instead, the minimax-optimal regret rate in our setting is of order $\sqrt{T}$. Additional offline contextual bandit algorithms can be found in \citet{dimakopoulou2017estimation}, \citet{kang2023heavy}, and \citet{rusmevichientong2010linearly}. \cite{khamaru2021near} and \cite{chen2021statistical} investigate the asymptotic properties of Markov linear regression, but do not address regret. Both works rely on OLS estimators. In particular, \citet{chen2021statistical} focus on bounded i.i.d. covariates and i.i.d. noise in the special case of two arms ($K=2$). To date, no offline algorithm is known to achieve optimal non-asymptotic performance when covariates and noise follow a Markov process. Our work contributes to this gap by providing both non-asymptotic and asymptotic analysis for this setting in an \emph{online} fashion, without storing historical observations.

More recently, \citet{chen2021statisticalb} and \citet{chen2022online} take computational and memory constraints into account by employing SGD. Their analyses also focus on the two-arm setting ($K=2$), where the suboptimal arm set $\calA^c$ is empty. They assume i.i.d. covariates and noise. \citet{chen2022online} further study scenarios with a constant exploration rate, which yields $O(T)$ regret but enables asymptotically Gaussian inference for the estimator. It is important to emphasize that the general multi-arm setting with $K>2$ is \emph{not} a straightforward extension of the two-arm case, either asymptotically or non-asymptotically; see, for example, Proposition~\ref{prop:Bahadur}. By applying Theorem~\ref{thm:LBD-estimation} and Theorem~\ref{thm:regret} to the two-arm case with i.i.d. covariates and noise, our approach improves upon existing online contextual bandit results, sharpening the regret from $\textsf{Regret}(T)=O(T^{2/3})$ to the minimax-optimal $\sqrt{T}$ rate. Moreover, we strengthen the Bahadur-type representation by obtaining a residual term that is stochastically $O(1/t)$. Our framework removes the classical exploration-exploitation trade-off by fully leveraging dependent data, enabling the estimation error for arms in $\calA$ to match the performance of fully exploratory designs. In contrast to prior work, which requires the exploration rate not to decay too quickly, our results accommodate a broad range of exploration schemes and even allow $\pi_t = 0$ for sufficiently large $t$; see Theorem~\ref{thm:regret}, Proposition~\ref{prop:Bahadur}, and Corollaries~\ref{cor:regret 1} and \ref{cor:regret 2}.

Another line of work studies online bandits with low-rank matrix structures \citep{han2022online,duan2024online,li2023online}. These methods also rely on SGD updates, but use stepsizes of order $1/t^\alpha$ with $\alpha\in(0,1)$, a strategy different from ours. Their regret bounds typically scale as $T^{2/3}$, which exceeds the minimax lower bound $\sqrt{T}$. Whether the low-rank structure introduces additional technical challenges remains an open question and is an interesting direction for future research.

\section{Numeric Experiments}
\label{sec:numeric} 

This section presents a series of numerical experiments evaluating the performance of the proposed algorithms. Across all settings, the empirical results align closely with the theoretical predictions. Unless otherwise specified, we set the ambient dimension to $d = 99$ and the time horizon to $T = 999{,}999$.

\paragraph{Linear Regression.} We begin with experiments for online linear regression (Theorem~\ref{thm:nonsymLR}, Section~\ref{sec:LR}). The procedure employs an initial constant stepsize to stabilize the residual scale, followed by a decaying stepsize of the form $\frac{\Ca}{t-t_1+\Cb d}$. Figure~\ref{fig:LR_Ca} reports convergence dynamics for several choices of $\Ca \in \{ 3,10,50 \}$ under a common $\Cb$. The resulting estimation trajectories are remarkably similar across the different stepsize magnitudes. Figures~\ref{fig:LR_iid_Ca} and \ref{fig:LR_dependent_Ca} then compare two data-generating processes: (i) i.i.d. covariates and noise, and (ii) dependent covariates and noise generated according to
\begin{align}
    \X_t=\text{Independent Rad(1/2)}\times \frac{1}{\|\X_{t-1}\|}\times \X_{t-1}+N(\boldsymbol{0},\I_d),\label{eq:numeric:covariate}
\end{align}
\begin{align}
    \xi_t=\text{Independent Rad(1/2)}\times \frac{\operatorname{sign}(\xi_{t-1})\min\{|\xi_{t-1}|,1\}}{|\X_{t-1}(1)|}\times \X_{t-1}(1)+N(\boldsymbol{0},\I_d),\label{eq:numeric:noise}
\end{align}
 where Rad$(1/2)$ denotes Rademacher random variables and $N(\boldsymbol{0},\I_d)$ represents standard multivariate Gaussian. Both settings use the same signal-to-noise ratio. Notably, dependence in the data does not impede estimation accuracy, consistent with our theoretical guarantees.
 
 Figure~\ref{fig:LR_Cb} investigates sensitivity with respect to $\Cb$ while fixing $\Ca$. Even under substantial variations of $\Cb$, the long-run covariance behavior remains nearly identical. Figure~\ref{fig:LR_dependent_Cb} shows that this robustness persists when the covariates and noise follow the dependent processes \eqref{eq:numeric:covariate}–\eqref{eq:numeric:noise}, again mirroring the i.i.d. results in Figure~\ref{fig:LR_iid_Cb}.
 
\begin{figure}
    \centering
    \begin{subfigure}[b]{0.45\textwidth}
        \centering
        \includegraphics[width=\linewidth]{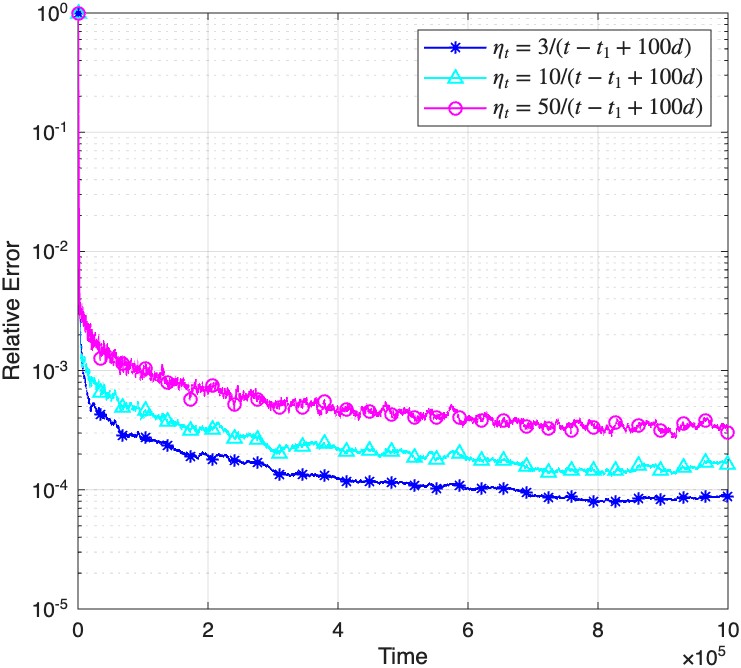}
    \caption{SGD for Independent Observations}
    \label{fig:LR_iid_Ca}
    \end{subfigure}
    \hfill
    \begin{subfigure}[b]{0.45\textwidth}
        \centering
        \includegraphics[width=\linewidth]{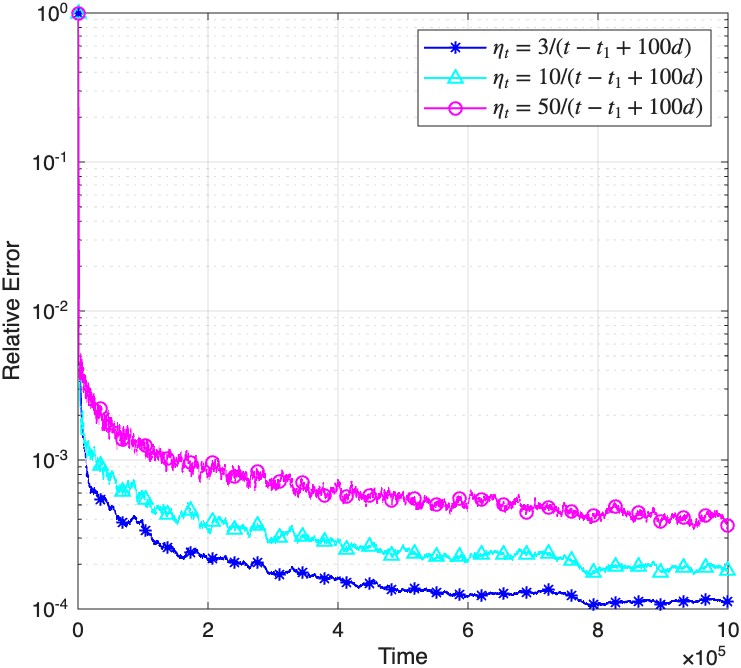}
    \caption{SGD for Dependent Observations}
    \label{fig:LR_dependent_Ca}
    \end{subfigure}
    \caption{Sensitivity with respect to $\Ca$: SGD in online linear regression under both i.i.d. and dependent observations. Each curve depicts the convergence dynamics for $\Ca \in \{ 3,10,50 \}$.}
    \label{fig:LR_Ca}
\end{figure}

\begin{figure}
    \centering
    \begin{subfigure}[b]{0.45\textwidth}
        \centering
        \includegraphics[width=\linewidth]{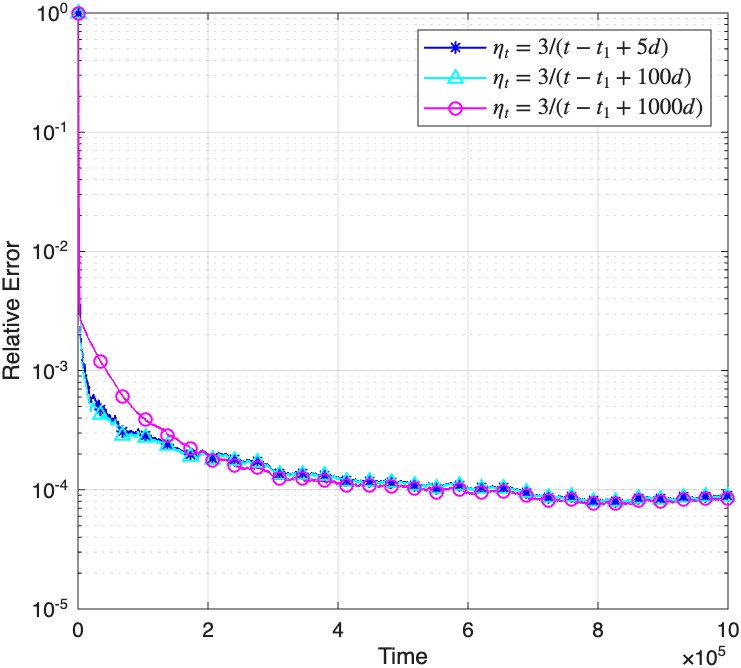}
    \caption{SGD for Independent Observations}
    \label{fig:LR_iid_Cb}
    \end{subfigure}
    \hfill
    \begin{subfigure}[b]{0.45\textwidth}
        \centering
        \includegraphics[width=\linewidth]{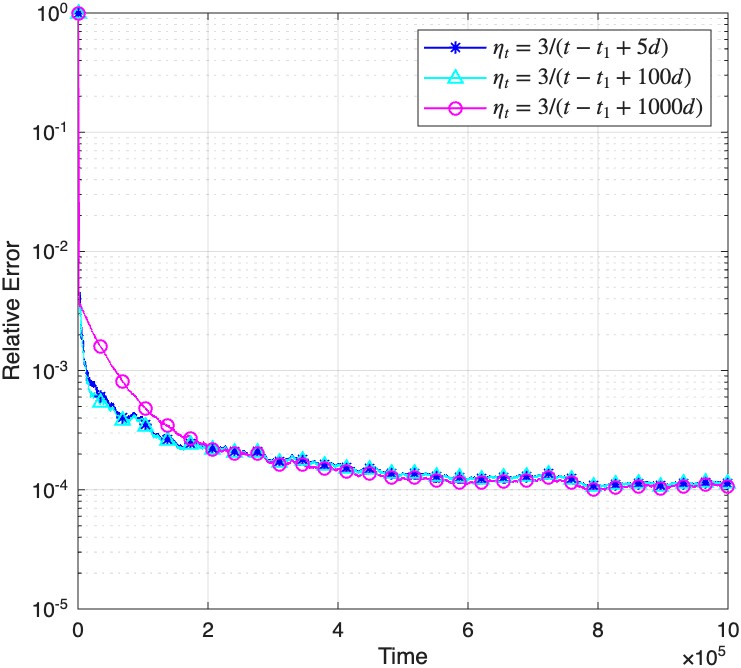}
    \caption{SGD for Dependent Observations}
    \label{fig:LR_dependent_Cb}
    \end{subfigure}
    \caption{Sensitivity with respect to $\Cb$: SGD in online linear regression under both i.i.d. and dependent observations. Each curve shows the convergence dynamics for $\Cb \in \{ 5,100,1000\}$.}
    \label{fig:LR_Cb}
\end{figure}

\paragraph{Sparse Linear Regression.} This segment presents the numerical experiment for online sparse linear regression (Theorem~\ref{thm:sparse-est-supprecv}, Section~\ref{sec:sparse}). When the time horizon is sufficiently large, the support of $\Bbeta^*$ can be fully recovered. We set $T = 9{,}999{,}999$ and $|\operatorname{supp}(\Bbeta^*)| = 4$. Moreover, the covariates and noise are generated to be dependent following Equation~\ref{eq:numeric:covariate} and \ref{eq:numeric:noise}. The dense SGD algorithm with a constant stepsize scheme (Theorem~\ref{thm:nonsymLR}(1)) is used for initialization. After this initialization phase, we perform a one-step hard-thresholding, retaining the six largest entries of $\Bbeta_t$ in absolute value. We compare the performance of the sparse SGD method (Theorem~\ref{thm:sparse-est-supprecv}) with that of the dense-parameter SGD algorithm (Theorem~\ref{thm:nonsymLR}). For a fair comparison, both algorithms use the same tuning parameters, $\Ca = 3$ and $\Cb = 100$.

Figure~\ref{fig:LSR_LSNR} corresponds to a weak signal-to-noise ratio setting, $\|\Bbeta^*\| / \EE|\xi| = 5$, where the largest signal equals $4\EE|\xi|$ and the smallest is as low as $0.004\EE|\xi|$. After initialization, only two nonzero components are recovered. We set $\alpha = 7$ in Theorem~\ref{thm:sparse-est-supprecv}, meaning that the sparse algorithm updates the support of $\Bbeta_t$ at most seven times. Despite the extremely weak magnitude of the remaining two signals relative to the noise, the sparse algorithm successfully identifies all four support entries by the end of the time horizon. As shown in Figure~\ref{fig:LSR_LSNR}, the sparse estimator achieves substantially smaller estimation error than the dense SGD estimator. Figure~\ref{fig:LSR_HSNR} reports the results under a high signal-to-noise ratio. After the constant stepsize phase and the one-step hard-thresholding, three of the four signals are recovered. The sparse algorithm subsequently identifies the final missing signal, again producing markedly smaller estimation error than the dense SGD algorithm.

\begin{figure}
    \centering
    \begin{subfigure}[b]{0.45\textwidth}
        \centering
        \includegraphics[width=\linewidth]{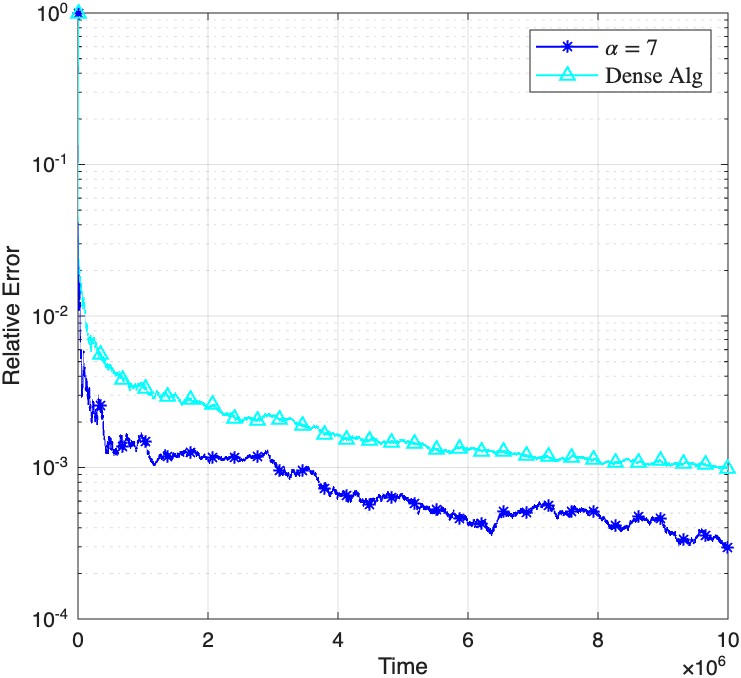}
        \caption{Low Signal-to-Noise Ratio: $\frac{\|\Bbeta^*\|}{\EE|\xi|}=5$}
        \label{fig:LSR_LSNR}
    \end{subfigure}
    \begin{subfigure}[b]{0.45\textwidth}
        \centering
        \includegraphics[width=\linewidth]{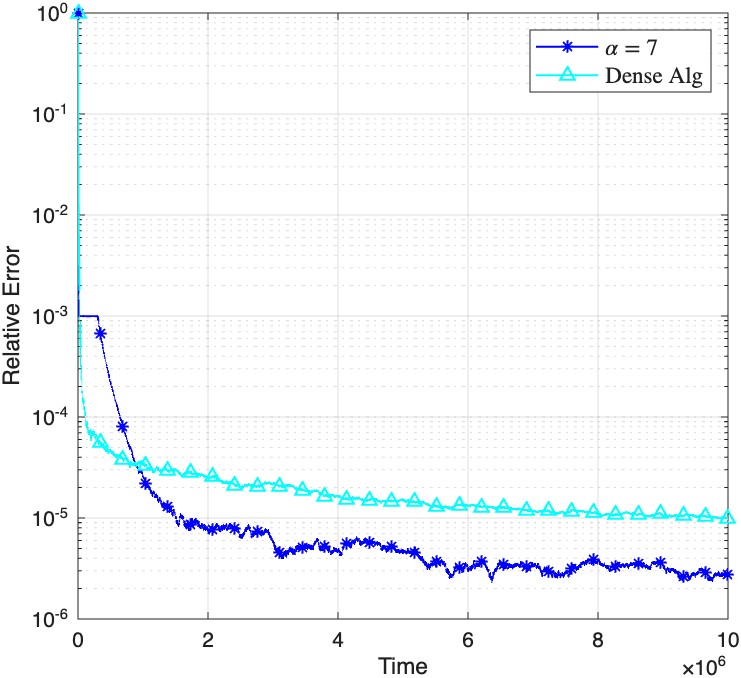}
        \caption{High Signal-to-Noise Ratio: $\frac{\|\Bbeta^*\|}{\EE|\xi|}=500$}
        \label{fig:LSR_HSNR}
    \end{subfigure}
    \caption{Sparse linear regression: the true parameter $\Bbeta^*\in\RR^{99}$ has support size $|\operatorname{supp}(\Bbeta^*)|=4$. The ``Dense Alg'' curve corresponds to the estimator in Theorem~\ref{thm:nonsymLR}$(2)$; while the ``$\alpha=7$'' curve reports the estimation error from the algorithm in Theorem~\ref{thm:sparse-est-supprecv}.}
    \label{fig:LSR}
\end{figure}

\paragraph{Linear Bandit.} We next study the contextual linear bandit setting (Theorems~\ref{thm:LBD-estimation} and \ref{thm:regret}, Section~\ref{sec:bandit}). Following the observations from the linear regression experiments, which show negligible accuracy loss under dependence, we focus here on the dependent data-generating processes \eqref{eq:numeric:covariate} and \eqref{eq:numeric:noise}. We set $K=5$ arms, of which exactly one is suboptimal, and choose parameters satisfying Assumption~\ref{assm:arm opt}. The stepsize schedule mirrors the regression experiment: a short warm-start phase with constant stepsize, followed by $\eta_t=\frac{20}{t-t_1+50d}$.

Figure~\ref{fig:LBD} compares estimation accuracy and regret across several exploration schemes. Figure~\ref{fig:LBD_Estimation} reports the estimation error of the optimal arm. Notably, the estimation trajectories under the decaying exploration rates $\pi_t=\frac{5}{\sqrt{t-t_1+50}}$ and $\pi_t=\frac{5}{t-t_1+50}$ are nearly indistinguishable from those obtained with a constant exploration rate. This demonstrates that the dependence structure induced by the decision-making process provides sufficient implicit exploration to sustain accurate estimation, in line with the guarantees of Theorem~\ref{thm:LBD-estimation}. Figure~\ref{fig:LBD_Regret} illustrates the cumulative regret under various exploration strategies. When the exploration rate is held constant, the regret grows linearly and is noticeably larger than under the decreasing exploration schedules $\pi_t=\frac{5}{\sqrt{t-t_1+50}}$ or $\pi_t=\frac{5}{t-t_1+50}$. The superior performance of the decaying exploration rules is consistent with the theoretical guarantees established in Theorem~\ref{thm:regret}.

\begin{figure}
    \centering
    \begin{subfigure}[b]{0.46\textwidth}
        \centering
        \includegraphics[width=\linewidth]{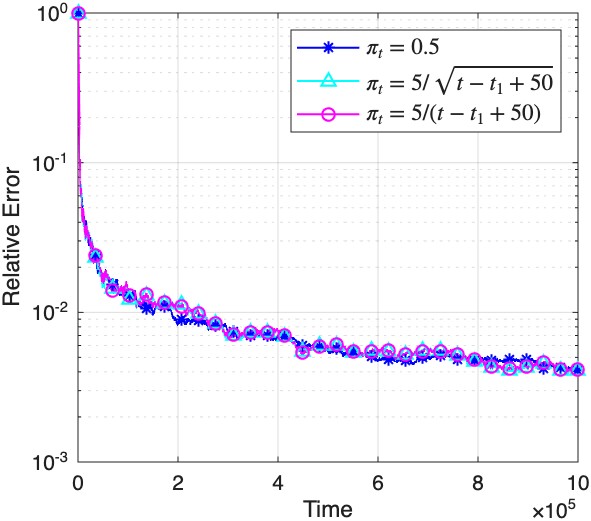}
    \caption{Estimation Error}
    \label{fig:LBD_Estimation}
    \end{subfigure}
    \hfill
    \begin{subfigure}[b]{0.45\textwidth}
        \centering
        \includegraphics[width=\linewidth]{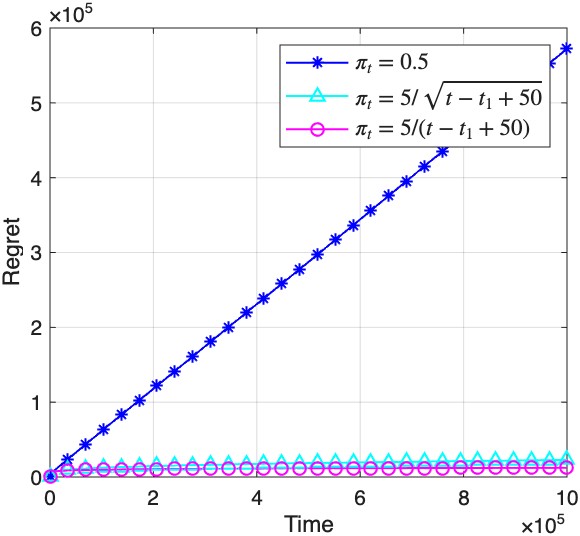}
    \caption{Regret}
    \label{fig:LBD_Regret}
    \end{subfigure}
    \caption{Exploration rate schemes: estimation error and regret under three exploration schemes, $\pi_t=\tfrac{1}{2},\tfrac{5}{\sqrt{t-t_1+50}}$, and $\tfrac{5}{t-t_1+50}$, all computed under a common fixed SGD stepsize scheme.}
    \label{fig:LBD}
\end{figure}

\section{Acknowledgment}
Y. Shen acknowledges support by the National Science Foundation grant DMS CAREER-2045068. 

\bibliographystyle{plainnat}
\bibliography{reference}

\newpage
\appendix
	\begin{center}
	\textbf{\Large{Supplementary Materials to ``SGD under Dependence:\\ Optimal Estimation, Regret and Inference"}}
\end{center}

\section{Proofs in Section~\ref{sec:LR}}
\label{sec:proofLR}
\subsection{Conditional Orlicz Norm}
In this section, we formally introduce the definition of conditional Orlicz norm. We need the following definition of essential supremum and we refer interested readers to \cite{barron2003conditional}, \cite{follmer2016stochastic} and \cite{lepinette2019conditional} for more detailed discussion.
\begin{definition}[Definition of Essential Supremum/Infimum]
    Let $\Phi$ be any set of random variables on $(\Omega,\calF,\PP)$.
    \begin{enumerate}
        \item There exists a random variable $\varphi^*$ satisfying the following two preperties.
        \begin{enumerate}
        \item  $\varphi^*\geq\varphi$ P-a.s. for all $\varphi\in\Phi$;
        \item for every $\psi\geq\varphi$ P-a.s. for all $\varphi\in\Phi$, it has $\psi\geq \varphi^*$;
    \end{enumerate}
    then $\varphi^*$ is the essential supremum of $\Phi$. The essential infimum of $\Phi$ is the essential supremum of $-\Phi$.
    \item Suppose in addition that $\Phi$ is directed upward, i.e., for every $\varphi,\widetilde{\varphi}\in\Phi$ there exists $\psi\in\Phi$ such that $\psi\geq \varphi\vee\widetilde{\varphi}$. Then there exists an increasing sequence $\varphi_1\leq\varphi_2\leq\cdots$ in $\Phi$ such that $\varphi^*=\lim_{n\to+\infty}\varphi_n$ P-almost surely.
    \end{enumerate}
    \label{def:esssup}
\end{definition}

\begin{theorem}[Conditional Orlicz Norm]
        Suppose $X$ is a real valued random variable of $(\Omega,\calF_X,\PP)$. Then for any $\calF\subseteq\calF_X$, let $\Phi:=\{\varphi\in\calF:\; \varphi\geq 0,\; \EE\{\exp(|X/\varphi|^{\alpha})|\calF\}\leq 2 \}$. Define the conditional Orlicz norm as follows,
        \begin{align*}
            \|X|\calF\|_{\Psi_\alpha}:=\operatorname{ess\, inf}\Phi.
        \end{align*}
        Then for $\alpha\geq1$, $\|\cdot|\calF\|_{\Psi_{\alpha}}$ is well defined as a norm and satisfies $\EE\{\exp(|X/\|X|\calF\|_{\Psi_\alpha}|^{\alpha})|\calF\}\leq 2 $.
        \label{thm:conditional orlicz norm}
    \end{theorem}
    The proof is adapted from the proof for Orlicz norm $\|\cdot\|_{\Psi_\alpha}$.
    \begin{proof}
We first prove $\|\cdot|\calF\|_{\Psi_\alpha}$ is a norm. It is obvious that for any scalar $a$, it has $\|aX|\calF\|_{\Psi_{\alpha}}=|a|\cdot\|X|\calF\|_{\Psi_{\alpha}}$. On the other hand, if $X=0$ a.s., then $\|X|\calF\|_{\Psi_{\alpha}}=0$ a.s.. Conversely, if $\|X|\calF\|_{\Psi_{\alpha}}=0$ a.s., then let $A:=\{w:X(w)\neq 0\}$. By Jensen's inequality, we have $\exp((\EE\{|X|/K_{\calF}|\calF\})^\alpha)\leq\EE\l\{\exp((|X|/K_{\calF})^\alpha)\big|\calF\r\}\leq 2$, and it implies $\EE|X|=0$, which shows $X=0$ almost surely. We only need to verify the triangle inequality. For any $\eps>0$, let $u=\|X|\calF\|_{\alpha}+\eps$ and $v=\|Y|\calF\|_{\Psi_{\alpha}}+\eps$, and it has \begin{align*}
        \EE\l\{\exp\l(\l(\frac{|X+Y|}{u+v}\r)^{\alpha}\r)\bigg|\calF\r\}&\leq\EE\l\{\exp\l(\l(\frac{|X|+|Y|}{u+v}\r)^{\alpha}\r)\bigg|\calF\r\}\\
        &=\EE\l\{\exp\l(\l(\frac{u}{u+v}\frac{|X|}{u}+\frac{v}{u+v}\frac{|X|}{v}\r)^{\alpha}\r)\bigg|\calF\r\}\\
        &\leq\frac{u}{u+v}\EE\l\{\exp\l(\l(\frac{|X|}{u}\r)^{\alpha}\r)\bigg|\calF\r\}+\frac{v}{u+v}\EE\l\{\exp\l(\l(\frac{|X|}{v}\r)^{\alpha}\r)\bigg|\calF\r\}\\
        &\leq 2,
        \end{align*}
        which proves $\|\cdot|\calF\|_{\Psi_\alpha}$ is a norm. Moreover, $-\Phi$ is directed upward and then by Definition~\ref{def:esssup}, we have the existence of a decreasing positive random variables $\varphi_1\geq\varphi_2\geq\cdots$ in $\Phi$ such that $\lim_{n\to+\infty} \varphi_n=\varphi^*:=\|X|\calF\|_{\Psi_\alpha}$. Hence, by Fatou's Lemma, we have
        \begin{align*}
            \EE\l\{\exp\l|\frac{X}{\|X|\calF\|_{\Psi_\alpha}}\r|^\alpha|\calF\r\}=\lim_{n\to+\infty} \EE\l\{\exp\l|\frac{X}{\varphi_n}\r|^\alpha|\calF\r\}\leq 2.
        \end{align*}
    \end{proof}
\begin{lemma}[Preliminary] Suppose $\X,\Y$ are random vectors and $\calF$ is a sigma field, then we have
		\begin{align*}
			\|\X\Y^{\top}|\calF\|_{\Psi_1}\leq \|\X|\calF\|_{\Psi_2}\cdot\|\Y|\calF\|_{\Psi_2}.
		\end{align*}
	If $X\perp Y|\calF$, then $\EE[XY|\calF]=\EE[X|\calF]\cdot\EE[Y|\calF]$ and $\EE[X|Y,\calF]=\EE[X|\calF]$, $a.s.$. If $\|X|\calF\|_{\Psi_\alpha}\leq K$ holds $a.s.$, and $K$ is a constant, then $\|X\|_{\Psi_\alpha}\leq K$. If $X$ is independent of $\calF$, then $\|X|\calF\|_{\Psi_1}=\|X\|_{\Psi_1}$, $a.s.$. If $X,Y$ are nonnegative nonnegative random variables and $|Y|\leq K$ a.s., where $K$ is a constant, then $\|XY|\calF\|_{\Psi_\alpha}\leq K\|X|\calF\|$.
	\end{lemma}
    \begin{lemma}
        If $\calF_1\subseteq\calF_2$ are sigma fields and $\|\X|\calF_2\|_{\Psi_2}\leq K$ for some constant $K$, then $\|\X|\calF_1\|_{\Psi_2}\leq K$. Additionally, for $\X\in\calF_2$ with $\|\X|\calF_1\|_{\Psi_{2}}\leq\lambda$ and for $\|\xi|\calF_2\|_{\Psi_2}\leq K$, then we have $\|\xi\X|\calF_1\|_{\Psi_1}\leq K\lambda$.
    \end{lemma}
   \noindent Afterwards, the properties of conditional Orlicz norm are used as knowledge.
\subsection{Proof of Theorem~\ref{thm:nonsymLR}}
In this section, we only prove the constant stepsize scheme $\eta_{t}:=\eta$ of Theorem~\ref{thm:nonsymLR}. We defer the remaining proof to Section~\ref{sec:proof-bandit}, where a more general convergence is proved and the second part of Theorem~\ref{thm:nonsymLR} is included as a special case.
\begin{proof}[Proof of Theorem~\ref{thm:nonsymLR} (i)]
	We prove by induction. Denote the event $$\calE_t:=\l\{\|\Bbeta_{t}-\Bbeta^*\|^2\leq C^*(1-c)^{2t}\|\Bbeta_{0}-\Bbeta^*\|^2+2\frac{\lambda_{\max}}{\lambda_{\min}}\eta\sigma^2d\r\},$$ where values of $c,C^*$ will be specified later. Obviously, the event holds with $t=0$. We proceed to prove the convergence dynamics of $\Bbeta_{t+1}-\Bbeta^*$ under $\cup_{l=0}^t \calE_l$. Let $\calE_t^X:=\l\{\|\X_t\|^2\leq 2\lambda_{\max}d\r\}$. Assumption~\ref{assm:cov} implies $\PP(\calE_t^X|\calF_t)\geq1-\exp(-Cd)$. According to the update Equation~\eqref{alg:regression}, we can write the estimation error as follows,
	\begin{align*}
		\l\|\Bbeta_{t+1}-\Bbeta^*\r\|^2=\l\|\Bbeta_{t}-\Bbeta^*\r\|^2-2\eta\l(\Bbeta_{t}-\Bbeta^*\r)^{\top}\X_t\l(\X_t^{\top}\Bbeta_t-Y_t\r)+\eta^2 \l(\X_t^{\top}\Bbeta_t-Y_t\r)^2 \l\|\X_t\r\|^2.
	\end{align*}
	We first consider its expectation conditional on $\calF_t^+$. Assumption~\ref{assm:noise} implies that $\EE\l\{\xi_t|\calF_t^+\r\}=0$. Hence, we have
	\begin{multline*}
		\EE\l\{\l\|\Bbeta_{t+1}-\Bbeta^*\r\|^2\big|\calF_t^+ \r\}\\
		\leq\l\|\Bbeta_{t}-\Bbeta^*\r\|^2-2\eta\l(\Bbeta_{t}-\Bbeta^*\r)^{\top}\X_t\X_t^{\top}\l(\Bbeta_{t}-\Bbeta^*\r)+\eta^2\l(\X_t^{\top}\Bbeta_{t}-\X_t^{\top}\Bbeta^*\r)^2\|\X_t\|^2+\eta^2\sigma^2\|\X_t\|^2.
	\end{multline*}
	We further consider its expectation conditional on $\calF_t$. By Assumption~\ref{assm:cov}, we have
	\begin{align*}
		\EE\l\{\l\|\Bbeta_{t+1}-\Bbeta^*\r\|^2\big|\calF_t \r\}
		\leq \l( 1-2\eta\lambda_{\min}+\eta^2\lambda_{\max}^2d\r)\l\|\Bbeta_{t}-\Bbeta^*\r\|^2+\eta^2\sigma^2\lambda_{\max}d.
	\end{align*}
	We then consider the difference between $\l\|\Bbeta_{t+1}-\Bbeta^*\r\|^2$ and $\EE\l\{\l\|\Bbeta_{t+1}-\Bbeta^*\r\|^2\big|\calF_t \r\}$,
	\begin{align*}
		\l\|\Bbeta_{t+1}-\Bbeta^*\r\|^2&-\EE\l\{\l\|\Bbeta_{t+1}-\Bbeta^*\r\|^2\big|\calF_t \r\}=\underbrace{2\eta\cdot\xi_t\l(\Bbeta_{t}-\Bbeta^*\r)^{\top}\X_t}_{I_1^{(t)}}\\
		&\underbrace{-2\eta \l(\Bbeta_{t}-\Bbeta^*\r)^{\top}\l\{\X_t\X_t^{\top}-\EE\l\{\X_t\X_t^{\top} \big|\calF_t\r\}\r\}\l(\Bbeta_{t}-\Bbeta^*\r)}_{I_2^{(t)}}+\underbrace{\eta^2\l\{\xi_t^2\|\X_t\|^2-\EE\l\{\xi_t^2\|\X_t\|^2\big|\calF_t\r\}\r\}}_{I_3^{(t)}}\\
		&+\underbrace{\eta^2\l\{\l(\X_t^{\top}\Bbeta_{t}-\X_t^{\top}\Bbeta^*\r)^2\|\X_t\|^2-\EE\l\{\l(\X_t^{\top}\Bbeta_{t}-\X_t^{\top}\Bbeta^*\r)^2\|\X_t\|^2\big|\calF_t\r\} \r\}}_{I_4^{(t)}}.
	\end{align*}
	Then, we decompose the estimation error as follows,
	\begin{align*}
		\|\Bbeta_{t+1}-\Bbeta^*\|^2&=	\EE\l\{\l\|\Bbeta_{t+1}-\Bbeta^*\r\|^2\big|\calF_t \r\}+\l\{\l\|\Bbeta_{t+1}-\Bbeta^*\r\|^2 -	\EE\l\{\l\|\Bbeta_{t+1}-\Bbeta^*\r\|^2\big|\calF_t \r\}\r\}\\
		&\leq \l( 1-2\eta\lambda_{\min}+\eta^2\lambda_{\max}^2d\r)\l\|\Bbeta_{t}-\Bbeta^*\r\|^2+\eta^2\sigma^2\lambda_{\max}d+I_1^{(t)}+I_2^{(t)}+I_3^{(t)}+I_4^{(t)}.
	\end{align*}
	For $\eta\leq \lambda_{\min}/(d\lambda_{\max}^2)$, we have
    \begin{align*}
        \|\Bbeta_{t+1}-\Bbeta^*\|^2\leq \l( 1-\eta\lambda_{\min}\r)\l\|\Bbeta_{t}-\Bbeta^*\r\|^2+\eta^2\sigma^2\lambda_{\max}d+I_1^{(t)}+I_2^{(t)}+I_3^{(t)}+I_4^{(t)}.
    \end{align*}
    We then accumulate the above equation until $t=0$,
	\begin{equation}
		\begin{split}
			&\|\Bbeta_{t+1}-\Bbeta^*\|^2\leq
			 \l( 1-\eta\lambda_{\min}\r)^{t+1}\l\|\Bbeta_{0}-\Bbeta^*\r\|^2+\sum_{l=0}^{t}\l( 1-\eta\lambda_{\min}\r)^{t-l}I_{1}^{(l)}+\sum_{l=0}^{t}\l( 1-\eta\lambda_{\min}\r)^{t-l}I_{2}^{(l)}\\
			&~~~~~~~~~~~~~~+\sum_{l=0}^{t}\l( 1-\eta\lambda_{\min}\r)^{t-l}I_{3}^{(l)}+\sum_{l=0}^{t}\l( 1-\eta\lambda_{\min}\r)^{t-l}I_{4}^{(l)}+\eta^2\sigma^2\lambda_{\max}d\sum_{l=0}^{t}\l( 1-\eta\lambda_{\min}\r)^l.
		\end{split}
		\label{eq1:regression}
	\end{equation}
	 We then bound each term of the RHS for Equation~\ref{eq1:regression}. It's worth noting that $$ \l( 1-\eta\lambda_{\min}\r)^{t+1}\leq \l( 1-\eta\lambda_{\min}/2\r)^{2(t+1)},\quad \eta^2\sigma^2\lambda_{\max}d\sum_{l=0}^{t}\l( 1-\eta\lambda_{\min}\r)^l \leq\frac{\lambda_{\max}}{\lambda_{\min}}\eta\sigma^2d.$$We then only need to bound the cumulated term for $I_1,I_2,I_3,I_4$. First consider the $I_1$ term, as $\Bbeta_t\in\calF_t$, and we have
	\begin{align*}
		\EE\l\{\xi_l\l(\Bbeta_{l}-\Bbeta^*\r)^{\top}\X_l\big|\calF_l\r\}=0,\quad \l\| \xi_l\l(\Bbeta_{l}-\Bbeta^*\r)^{\top}\X_l\big|\calF_l\r\|_{\Psi_1}\leq \sqrt{\lambda_{\max}}\sigma\|\Bbeta_{l}-\Bbeta^*\|.
	\end{align*}
Additionally, under the event $\calE_l$, $\|\Bbeta_{l}-\Bbeta^*\|^2\leq C^* (1-c)^{2l}\|\Bbeta_l-\Bbeta^*\|^2+2\frac{\lambda_{\max}}{\lambda_{\min}}\eta \sigma^2 d$. We then use technical Lemma~\ref{teclem:azuma}. Thus, under the event $\cup_{l=0}^t\calE_l$, we have 
\begin{align*}
	\sum_{l=0}^t \l\| (1-\eta\lambda_{\min})^{t-l}I_{1}^{(l)}\big|\calF_l\r\|_{\Psi_1}^2&\leq C\eta^2\sigma^2\lambda_{\max}\sum_{l=0}^t (1-\eta\lambda_{\min})^{2t-2l} ((1-c)^{2l}\|\Bbeta_{0}-\Bbeta^*\|^2+\frac{\lambda_{\max}}{\lambda_{\min}}d\eta\sigma^2)\\
	&\leq C\eta^2(1-c)^{2t}\frac{1}{1-\frac{1-\eta\lambda_{\min}}{1-c}}\|\Bbeta_{0}-\Bbeta^*\|^2+d\eta^3 \sigma^4\frac{\lambda_{\max}^2}{\lambda_{\min}(1-(1-\eta\lambda_{\min})^2)}\\
	&\leq C\frac{\eta\lambda_{\max}\sigma^2}{\lambda_{\min}}(1-c)^{2t}\|\Bbeta_{0}-\Bbeta^*\|^2+d\frac{\lambda_{\max}^2\eta^2\sigma^4}{\lambda_{\min}^2},
\end{align*}
where the last line is due to $4c\leq \eta\lambda_{\min}$. Additionally, we have
\begin{align*}
	\max_{l}\l\| (1-\eta\lambda_{\min})^{t-l}I_{1}^{(l)}\big|\calF_l\r\|_{\Psi_1}\leq C\eta\sigma\sqrt{\lambda_{\max}}\l((1-c)^{t}\|\Bbeta_{0}-\Bbeta^*\|+\sqrt{\frac{\lambda_{\max}}{\lambda_{\min}}}\sqrt{\eta}\sqrt{d}\sigma \r).
\end{align*}
 Thus, by Lemma~\ref{teclem:azuma}, we have
\begin{multline*}
	\PP\l(\l|\sum_{l=0}^{t}\l( 1-\eta\lambda_{\min}\r)^{t-l}I_{1}^{(l)}\r|\geq s \r)
	\leq\exp\l(-\min\l\{ \frac{s^2}{C\l(\frac{\sigma^2\eta\lambda_{\max}}{\lambda_{\min}}(1-c)^{2t}\|\Bbeta_{0}-\Bbeta^*\|^2+d\frac{\eta^2\sigma^4\lambda_{\max}^2}{\lambda_{\min}^2} \r)},\r.\r. \\ \l.\l.\frac{s}{\eta\sigma\sqrt{\lambda_{\max}}\l((1-c)^t\|\Bbeta_{0}-\Bbeta^*\|+\sqrt{\frac{\lambda_{\max}}{\lambda_{\min}}}\sqrt{\eta}\sqrt{d}\sigma \r)}\r\}\r).
\end{multline*}
We insert $s=(1-c)^{2t+2}\|\Bbeta_{0}-\Bbeta^*\|^2+\frac{\lambda_{\max}}{\lambda_{\min}}\eta\sigma^2d\geq 2(1-c)^{t+1}\sqrt{\frac{\lambda_{\max}}{\lambda_{\min}}}\sqrt{\eta d}\|\Bbeta_{0}-\Bbeta^*\|\sigma$ into the above equation and it yields
\begin{multline*}
	\PP\l(\l|\sum_{l=0}^{t}\l( 1-\eta\lambda_{\min}\r)^{t-l}I_{1}^{(l)}\r|\geq (1-c)^{2t+2}\|\Bbeta_{0}-\Bbeta^*\|^2+\frac{\lambda_{\max}}{\lambda_{\min}}\eta\sigma^2d \r)\\ \leq\exp\l(-c_1\min\l\{d,\sqrt{\frac{1}{\lambda_{\min}}\frac{d}{\eta}}\r\}\r)
	%\l\{\begin{array}{lr}
	%	\exp\l(-c_1\min\l\{d,\sqrt{\frac{1}{\lambda_{\min}}\frac{d}{\eta}}\r\}\r)&  \text{ if }(1-c)^{2t-2}\|\Bbeta_{0}-\Bbeta^*\|^2\geq \eta\sigma^2(d+\delta_{t+1})\\
	%	\exp\l(-c_1\min\l\{d,\sqrt{\frac{1}{\lambda_{\min}}\frac{d}{\eta}}\r\}\r)&  \text{ if } (1-c)^{2t-2}\|\Bbeta_{0}-\Bbeta^*\|^2\geq \eta\sigma^2(d+\delta_{t+1})
%	\end{array}\r.
\end{multline*}
We then consider the term $I_2$. Assumption~\ref{assm:cov} guarantees,
\begin{align*}
	\l\|\l(\Bbeta_{l}-\Bbeta^*\r)^{\top}\l\{\X_l\X_l^{\top}-\EE\l\{\X_l\X_l^{\top} \big|\calF_l\r\}\r\}\l(\Bbeta_{l}-\Bbeta^*\r)\big|\calF_l\r\|_{\Psi_1}\leq \lambda_{\max}\|\Bbeta_{l}-\Bbeta^*\|^2.
\end{align*} Under the event $\cup_{l=0}^t\calE_l$, we have that 
\begin{align*}
	\sum_{l=0}^t \l\| (1-\eta\lambda_{\min})^{t-l}I_{2}^{(l)}\big|\calF_l\r\|_{\Psi_1}^2&\leq C\eta^2\lambda_{\max}^2\sum_{l=0}^t (1-\eta\lambda_{\min})^{2t-2l}\l((1-c)^{4l}\|\Bbeta_{0}-\Bbeta^*\|^4+\frac{\lambda_{\max}^2}{\lambda_{\min}^2}\eta^2\sigma^4d^2\r)\\
	&\leq C\eta\frac{\lambda_{\max}^2}{\lambda_{\min}}(1-c)^{4t}\|\Bbeta_{0}-\Bbeta^*\|^4+C\frac{\lambda_{\max}^4}{\lambda_{\min}^3}\eta^3 \sigma^4d^2,
\end{align*}
and
\begin{align*}
	\max_{l}\l\| (1-\eta\lambda_{\min})^{t-l}I_{2}^{(l)}\big|\calF_l\r\|_{\Psi_1}\leq C\eta\lambda_{\max}(1-c)^{2t}\|\Bbeta_{0}-\Bbeta^*\|^2+\eta^2\frac{\lambda_{\max}^2}{\lambda_{\min}}\sigma^2d.
\end{align*}
Then follow Lemma~\ref{teclem:azuma} and select $s=(1-c)^{2t+2}\|\Bbeta_{0}-\Bbeta^*\|^2+\frac{\lambda_{\max}}{\lambda_{\min}}\eta\sigma^2d$, we have
\begin{align*}
	\PP\l(\l|\sum_{l=0}^{t}\l( 1-\eta\lambda_{\min}\r)^{t-l}I_{1}^{(l)}\r|\geq (1-c)^{2t+2}\|\Bbeta_{0}-\Bbeta^*\|^2+\frac{\lambda_{\max}}{\lambda_{\min}}\eta\sigma^2d\r) \leq \exp\l(-c_1\frac{\lambda_{\min}}{\lambda_{\max}^2}\frac{1}{\eta}\r).
\end{align*}
We then consider $I_3$.  Under $\calE_l^X$, by Assumption~\ref{assm:noise}, we have \begin{align*}
	\l\|\l\{\xi_l^2\|\X_l\|^2-\EE\l\{\xi_l^2\|\X_l\|^2\big|\calF_l\r\}\r\}\big|\calF_l  \r\|_{\Psi_1}\leq C\lambda_{\max}\sigma^2d.
\end{align*}
Under the event $\cup_{l=0}^t \calE_l^X$, we have
\begin{align*}
	\sum_{l=0}^t \l\| (1-\eta\lambda_{\min})^{t-l}I_{3}^{(l)}\big|\calF_l\r\|_{\Psi_1}^2&\leq \sum_{l=0}^t(1-\eta\lambda_{\min})^{2t-2l}\eta^4\lambda_{\max}^2\sigma^4d^2\leq\frac{\lambda_{\max}^2}{\lambda_{\min}}\eta^3d^2\sigma^4,
\end{align*}
and
\begin{align*}
	\max_{l}\l\| (1-\eta\lambda_{\min})^{t-l}I_{3}^{(l)}\big|\calF_l\r\|_{\Psi_1}\leq \eta^2\lambda_{\max}\sigma^2d.
\end{align*}
Follow Lemma~\ref{teclem:azuma} and then we insert $s=(1-c)^{2t+2}\|\Bbeta_{0}-\Bbeta^*\|^2+\frac{\lambda_{\max}}{\lambda_{\min}}\eta\sigma^2d$,
\begin{align*}
	\PP\l( \l|\sum_{l=0}^{t}\l( 1-\eta\lambda_{\min}\r)^{t-l}I_{3}^{(l)}\r|\geq (1-c)^{2t+2}\|\Bbeta_{0}-\Bbeta^*\|^2+\frac{\lambda_{\max}}{\lambda_{\min}}\eta\sigma^2d\r)
	\leq \exp\l(-c_1\frac{1}{\lambda_{\max}}\frac{1}{\eta}\r).
\end{align*}
Finally, we consider $I_4$. Under the event $\calE_t^X$, we have
\begin{align*}
	\l\| \l\{\l(\X_l^{\top}\Bbeta_{l}-\X_t^{\top}\Bbeta^*\r)^2\|\X_l\|^2-\EE\l\{\l(\X_l^{\top}\Bbeta_{l}-\X_l^{\top}\Bbeta^*\r)^2\|\X_l\|^2\big|\calF_l\r\} \r\}\big|\calF_l\r\|_{\Psi_1}\leq C\lambda_{\max}^2d\|\Bbeta_{l}-\Bbeta^*\|^2.
\end{align*}
Under the event $\cup_{l=0}^t \calE_l^X$, we have
\begin{align*}
	\sum_{l=0}^t \l\| (1-\eta\lambda_{\min})^{t-l}I_{4}^{(l)}\big|\calF_l\r\|_{\Psi_1}^2&\leq C\eta^4\lambda_{\max}^4d^2\sum_{l=0}^t (1-\eta\lambda_{\min})^{2t-2l}\l((1-c)^{4l}\|\Bbeta_{0}-\Bbeta^*\|^4+\frac{\lambda_{\max}^2}{\lambda_{\min}^2} \eta^2\sigma^4d^2\r)\\
	&\leq C\frac{\lambda_{\max}^4}{\lambda_{\min}}\eta^3d^2\|\Bbeta_{0}-\Bbeta^*\|^4+C\frac{\lambda_{\max}^6}{\lambda_{\min}^3}\eta^5\sigma^4d^4,
\end{align*}
and
\begin{align*}
	\max_{l} \l\| (1-\eta\lambda_{\min})^{t-l}I_{4}^{(l)}\big|\calF_l\r\|_{\Psi_1}\leq C\eta^2\lambda_{\max}^2d(1-c)^{2t}\|\Bbeta_{0}-\Bbeta^*\|^2+C\eta^3(1-c)^{2t}\frac{\lambda_{\max}^3}{\lambda_{\min}}\sigma^2d^2.
\end{align*}
Insert $s=(1-c)^{2t+2}\|\Bbeta_{0}-\Bbeta^*\|^2+\frac{\lambda_{\max}}{\lambda_{\min}}\eta\sigma^2d$ into Lemma~\ref{teclem:azuma}, which leads to
\begin{multline*}
	\PP\l( \l|\sum_{l=0}^{t}\l( 1-\eta\lambda_{\min}\r)^{t-l}I_{4}^{(l)}\r|\geq (1-c)^{2t+2}\|\Bbeta_{0}-\Bbeta^*\|^2+\frac{\lambda_{\max}}{\lambda_{\min}}\eta\sigma^2d \r)\\
	\leq\exp\l(-\min\l\{\frac{\lambda_{\min}}{\lambda_{\max}^4\eta^3d^2},\frac{1}{\lambda_{\max}^2\eta^2d}\r\}\r).
\end{multline*}
Thus, Equation~\eqref{eq1:regression} arrives at,
\begin{align*}
	\|\Bbeta_{t+1}-\Bbeta^*\|^2\leq 5(1-c)^{2t+2}\|\Bbeta_{0}-\Bbeta^*\|^2+5\frac{\lambda_{\max}}{\lambda_{\min}}\eta\sigma^2d,
\end{align*}
with probability over $1-5\exp(-c_1d)$. And $c$ is any constant satisfying $c\leq\eta\lambda_{\min}/2$.
\end{proof}

\section{Proofs in Section~\ref{sec:sparse}}
\label{sec:proofSparse}
This section presents the proofs of sparse regression, specifically, consisting of Lemma~\ref{lem:FixedSupport}, Lemma~\ref{lem:sparse-gradient} and Proposition~\ref{prop:sparse-select}.
\subsection{Proof of Lemma~\ref{lem:FixedSupport}}
\label{sec:proof-FixedSupport}
First of all, the SGD update with a fixed support $\calS_0$ can be written as,
\begin{align*}
	\Bbeta_{t+1}-\Bbeta^*=\calH_{\calS_0}\l(\Bbeta_{t}-\eta_{t}\cdot\l(\X_t^{\top}\Bbeta_{t}-Y_t\r)\X_t\r)-\Bbeta^*=\Bbeta_{t}-\Bbeta^*-\eta_{t}\cdot\l(\X_t^{\top}\Bbeta_{t}-Y_t\r)\calH_{\calS_0}(\X_t).
\end{align*}
It's worth noting that $\Bbeta_{t}-\Bbeta^*$ is supported on $\calS^*\cup\calS_0$, where $\calS^*\cup\calS_0=(\calS^*\setminus\calS_0)\cup\calS_0$. For one thing, the iterates restricted on $\calS^*\setminus\calS_0$ satisfy
\begin{align*}
	\l[\Bbeta_{t+1}-\Bbeta^* \r]_{\calS^*\setminus \calS_0} = \l[\Bbeta_{t}-\Bbeta^* \r]_{\calS^*\setminus \calS_0}=[-\Bbeta^*]_{\calS^*\setminus\calS_0},
\end{align*}
while for entries on $\calS_0$, it is
\begin{align*}
	\l[\Bbeta_{t+1}-\Bbeta^* \r]_{ \calS_0}= [\Bbeta_{t}-\Bbeta^*]_{\calS_0}-\eta_{t}[\X_t]_{\calS_0}(\X_t^{\top}\Bbeta_{t}-Y_t).
\end{align*}
We are going to prove Lemma~\ref{lem:FixedSupport} by induction. Define the event $$\calE_t:=\l\{\|[\Bbeta_{t}-\Bbeta^*]_{\calS_0}\|^2\leq C^*\frac{s\log(2d/s)+\delta_{t}}{t+\Cb s\log(2d/s)}\frac{\sigma^2}{\lambda_{\min}}+C\l(\frac{\OffDiagS}{\lambda_{\min}}\r)^2\cdot\|[\Bbeta^*]_{\calS^*\setminus\calS_0}\|^2\r\},$$
where $C^*$ is some constant that does not depend on $t,d,s$, having its close form in Lemma~\ref{lem:FixedSupport}. It is obvious that $\calE_0$ holds. We will prove $\calE_{t+1}$ holds with high probability under $\cup_{l=0}^t\calE_{l}$. First of all, the squared estimation error supported on $\calS_0$ has
\begin{multline*}
	\|[\Bbeta_{t+1}-\Bbeta^* ]_{ \calS_0} \|^2=\|[\Bbeta_{t}-\Bbeta^*]_{\calS_0} \|^2-2\eta_{t}(\X_t^{\top}\Bbeta_{t}-Y_t)[\Bbeta_{t}-\Bbeta^*]_{\calS_0}^{\top}[\X_t]_{\calS_0}\\
	+\eta_{t}^2(\X_t^{\top}\Bbeta_{t}-Y_t)^2\| [\X_t]_{\calS_0}\|^2.
\end{multline*}
We remark that $\X_t$ might depend on $\calS_{0}$. Based on Assumption~\ref{assm:cov-sparse} and Lemma~\ref{teclem:maxsum tail}, we have $\EE\{\|[\X_t]_{\calS_{0}}\|^4|\calF_t\}\leq s^2\log^2(2d/s)\lambda_{\max}^2$ and, moreover, we have
\begin{align*}
	&\EE\l\{(\X_t^{\top}\Bbeta_{t}-Y_t)^2\| [\X_t]_{\calS_0}\|^2\big|\calF_t \r\}=\EE\l\{(\X_t^{\top}\Bbeta_{t}-\X_t^{\top}\Bbeta^*)^2\| [\X_t]_{\calS_0}\|^2\big|\calF_t \r\}+\EE\l\{\xi_t^2\| [\X_t]_{\calS_0}\|^2\big|\calF_t \r\}\\
	&~~~~~\leq \frac{s\log(2d/s)}{2\|\Bbeta_{t}-\Bbeta^*\|^2}\EE\l\{(\X_t\Bbeta_t-\X_t^{\top}\Bbeta^*)^4\big|\calF_t\r\}+\frac{\|\Bbeta_{t}-\Bbeta^*\|^2}{2s\log(2d/s)}\EE\l\{\|[\X_t]_{\calS_0}\|^4\big|\calF_t\r\}+\sigma^2s\log(2d/s)\lambda_{\max}\\
	&~~~~~\leq \lambda_{\max}^2s\log(2d/s)\|\Bbeta_{t}-\Bbeta^*\|^2+2\sigma^2s\log(2d/s)\lambda_{\max}\leq 3\sigma^2s\log(2d/s)\lambda_{\max},
\end{align*}
where the second line uses the inequality $a^2+b^2\geq 2ab$. Thus, the conditional expectation of $\|\Bbeta_{t+1}-\Bbeta^*\|^2$ satisfies the following equation,
\begin{multline*}
	\EE\l\{\l\|[\Bbeta_{t+1}-\Bbeta^* ]_{ \calS_0}\r\|^2\big|\calF_t\r\}\leq \l(1-2\eta_{t}\lambda_{\min}\r) \l\|\l[\Bbeta_{t}-\Bbeta^* \r]_{ \calS_0}\r\|^2+3\eta_{t}^2\lambda_{\max}\cdot s\log(2d/s)\sigma^2\\
	+2\eta_{t}\OffDiagS \cdot \|\l[\Bbeta_{t}-\Bbeta^* \r]_{\calS_0}\|\cdot\|[\Bbeta^*]_{\calS^*\setminus\calS_0}\|.
\end{multline*}
We bound the third term of the RHS with $2\OffDiagS \cdot \|\l[\Bbeta_{t}-\Bbeta^* \r]_{\calS_0}\|\cdot\|[\Bbeta^*]_{\calS^*\setminus\calS_0}\|\leq \lambda_{\min}\|[\Bbeta_{t}-\Bbeta^*]_{\calS_0}\|^2+(\OffDiagS)^2\|[\Bbeta^*]_{\calS^*\setminus\calS_0}\|^2/\lambda_{\min}$. After simplifications, it arrives at
\begin{multline*}
	\EE\l\{\l\|[\Bbeta_{t+1}-\Bbeta^* ]_{ \calS_0}\r\|^2\big|\calF_t\r\}
	\leq \l(1-\eta_{t}\lambda_{\min}\r)\l\|\l[\Bbeta_{t}-\Bbeta^* \r]_{ \calS_0}\r\|^2\\
	+3\eta_{t}^2\lambda_{\max}s\log(2d/s)\sigma^2+\eta_{t}\|[\Bbeta^*]_{\calS^*\setminus\calS_0}\|^2\cdot\frac{(\OffDiagS)^2}{\lambda_{\min}}.
\end{multline*}
We then decompose the difference between the squared estimation error and the conditional expectation,
\begin{align*}
	&\|[\Bbeta_{t+1}-\Bbeta^* ]_{ \calS_0}\|^2-\EE\l\{\l\|[\Bbeta_{t+1}-\Bbeta^* ]_{ \calS_0}\r\|^2\big|\calF_t\r\}\\
	&~~~~~~=-2\eta_{t}\underbrace{\l\{ [\Bbeta_{t}-\Bbeta^*]_{\calS_0}^{\top}[\X_t]_{\calS_0}\X_t^{\top} (\Bbeta_{t}-\Bbeta^*)-\EE\l\{[\Bbeta_{t}-\Bbeta^*]_{\calS_0}^{\top}[\X_t]_{\calS_0}\X_t^{\top} (\Bbeta_{t}-\Bbeta^*)\big| \calF_t \r\}\r\}}_{\boldsymbol{\varXi}_{1}^{(t)}}\\
	&~~~~~~~+2\eta_{t}\underbrace{\l\{ \xi_t[\Bbeta_{t}-\Bbeta^*]_{\calS_0}^{\top}[\X_t]_{\calS_0}-\EE\l\{\xi_t[\Bbeta_{t}-\Bbeta^*]_{\calS_0}^{\top}[\X_t]_{\calS_0}\big| \calF_t \r\}\r\}}_{\boldsymbol{\varXi}_2^{(t)}}\\
	&~~~~~~~+\eta_{t}^2\underbrace{\l\{(\X_t^{\top}\Bbeta_{t}-Y_t)^2\|[\X_t]_{\calS_0}\|^2-\EE\l\{(\X_t^{\top}\Bbeta_{t}-Y_t)^2\|[\X_t]_{\calS_0}\|^2\big|\calF_t \r\}\r\}}_{\boldsymbol{\varXi}_3^{(t)}}.
\end{align*}
It is obvious that the three terms $\boldsymbol{\varXi}_i^{(t)}$ are all martingale differences, with conditional Orlicz norm bounded with
\begin{align*}
	\|\boldsymbol{\varXi}_1^{(t)}|\calF_t\|_{\Psi_1}\leq \lambda_{\max}\|[\Bbeta_{t}-\Bbeta^*]_{\calS_0}\|^2+\OffDiagS\cdot\|[\Bbeta_{t}-\Bbeta^*]_{\calS_0}\|\cdot\|[\Bbeta^*]_{\calS^*\setminus\calS_0}\|,
\end{align*}
and
\begin{align*}
	\|\boldsymbol{\varXi}_2^{(t)}|\calF_t\|_{\Psi_1}\leq\sqrt{\lambda_{\max}}\sigma\|[\Bbeta_{t}-\Bbeta^*]_{\calS_0}\|.
\end{align*}
Denote $\calE_t^X:=\{\|[\X_t]_{\calS_0}\|^2\leq C\lambda_{\max}(s\log(2d/s)+\delta_{t})\}$ and by Lemma~\ref{teclem:maxsum tail}, we have $\PP(\calE_t^X)\geq 1-\exp(-c_2(s\log(2d/s)+\delta_{t}))$. Thus, under $\calE_t^X$, we have
\begin{align*}
	\|\boldsymbol{\varXi}_3^{(t)}|\calF_t,\calE_t^X\|_{\Psi_1}\leq C(s\log(2d/s)+\delta_{t})\lambda_{\max}\sigma^2.
\end{align*}
Thus, putting the conditional expectation together with the difference, we have
\begin{multline*}
	\|[\Bbeta_{t+1}-\Bbeta^* ]_{ \calS_0}\|^2=\EE\l\{ \|[\Bbeta_{t+1}-\Bbeta^* ]_{ \calS_0}\|^2\big|\calF_t\r\}+\l\{ \|[\Bbeta_{t+1}-\Bbeta^* ]_{ \calS_0}\|^2-\EE\l\{\|[\Bbeta_{t+1}-\Bbeta^* ]_{ \calS_0}\|^2\big|\calF_t\r\}\r\}\\
	\leq \l(1-\eta_{t}\lambda_{\min}\r)\l\|\l[\Bbeta_{t}-\Bbeta^* \r]_{ \calS_0}\r\|^2+3\eta_{t}^2\lambda_{\max}s\log(2d/s)\sigma^2
	\\
	+\eta_{t}\|[\Bbeta^*]_{\calS^*\setminus\calS_0}\|^2\cdot\frac{(\OffDiagS)^2}{\lambda_{\min}}-2\eta_{t}\boldsymbol{\varXi}_1^{(t)}+2\eta_{t}\boldsymbol{\varXi}_2^{(t)}+\eta_{t}^2\boldsymbol{\varXi}_3^{(t)}.
\end{multline*}
Insert the stepsize $\eta_{t}=\frac{\Ca}{\lambda_{\min}}\frac{1}{t+\Cb s\log(2d/s)}$ into the above equation and then we have
\begin{multline*}
	\|[\Bbeta_{t+1}-\Bbeta^* ]_{ \calS_0}\|^2\leq \underbrace{\prod_{j=0}^{t}\l(1-\frac{\Ca}{j+\Cb s\log(2d/s)}\r)\l\|\l[\Bbeta_{0}-\Bbeta^* \r]_{ \calS_0}\r\|^2}_{\boldsymbol{\mathit{A}}_1}\\
	+\underbrace{3s\log(2d/s)\Ca^2\frac{\lambda_{\max}}{\lambda_{\min}^2}\sum_{l=0}^{t}\prod_{j=l}^{t}\l(1-\frac{\Ca}{j+\Cb s\log(2d/s)}\r)\l(\frac{1}{l+\Cb s\log(2d/s)}\r)^2\sigma^2}_{\boldsymbol{\mathit{A}}_2}\\
	+\underbrace{\Ca\sum_{l=0}^{t}\prod_{j=l}^{t}\l(1-\frac{\Ca}{j+\Cb s\log(2d/s)}\r)\frac{1}{l+\Cb s\log(2d/s)}\|[\Bbeta^*]_{\calS^*\setminus\calS_0}\|^2\cdot\frac{(\OffDiagS)^2}{\lambda_{\min}^2}}_{\boldsymbol{\mathit{A}}_3}\\
	-\underbrace{2\frac{\Ca}{\lambda_{\min}}\sum_{l=0}^{t}\prod_{j=l}^{t}\l(1-\frac{\Ca}{j+\Cb s\log(2d/s)}\r)\frac{1}{l+\Cb s\log(2d/s)}\boldsymbol{\varXi}_1^{(l)}}_{\boldsymbol{\mathit{A}}_4}\\
	+\underbrace{2\frac{\Ca}{\lambda_{\min}}\sum_{l=0}^{t}\prod_{j=l}^{t}\l(1-\frac{\Ca}{j+\Cb s\log(2d/s)}\r)\frac{1}{l+\Cb s\log(2d/s)}\boldsymbol{\varXi}_2^{(l)}}_{\boldsymbol{\mathit{A}}_5}\\
	+\underbrace{\frac{\Ca^2}{\lambda_{\min}^2}\sum_{l=0}^{t}\prod_{j=l}^{t}\l(1-\frac{\Ca}{j+\Cb s\log(2d/s)}\r)\frac{1}{(l+\Cb s\log(2d/s))^2}\boldsymbol{\varXi}_3^{(l)}}_{\boldsymbol{\mathit{A}}_6}.
\end{multline*}
We then bound each of the terms respectively. First, consider $\A_1$. By Lemma~\ref{teclem:prod upper bound}, we have
\begin{align*}
	\prod_{j=0}^{t}\l(1-\frac{\Ca}{j+\Cb s\log(2d/s)}\r)\leq \l(\frac{\Cb s\log(2d/s)}{t+1+\Cb s\log(2d/s)}\r)^{\Ca},
\end{align*}
which implies $\boldsymbol{\mathit{A}}_1\leq\frac{\Cb s\log(2d/s)}{t+\Cb s\log(2d/s)}\|[\Bbeta_{0}-\Bbeta^*]_{\calS_0}\|^2$. Then consider term $\boldsymbol{\mathit{A}}_2$. By Lemma~\ref{teclem:prod upper bound}, we have $\prod_{j=l}^{t}\l(1-\frac{\Ca}{j+\Cb s\log(2d/s)}\r)\l(\frac{1}{l+\Cb s\log(2d/s)}\r)^2\leq\frac{(l+\Cb s\log(2d/s))^{\Ca -2}}{(t+1+\Cb s\log(2d/s))^{\Ca}}$ and Lemma~\ref{teclem:sum upper bound} further implies $\sum_{l=0}^{t}\prod_{j=l}^{t}\l(1-\frac{\Ca}{j+\Cb s\log(2d/s)}\r)\frac{1}{(l+\Cb s\log(2d/s))^2}\leq\frac{2/\Ca}{t+1+\Cb s\log(2d/s)}$. Thus, we have
\begin{align*}
	\boldsymbol{\mathit{A}}_2\leq C\frac{\lambda_{\max}}{\lambda_{\min}^2}\cdot\frac{\Ca s\log(2d/s)}{t+1+\Cb s\log(2d/s)}\sigma^2.
\end{align*}
As for term $\boldsymbol{\mathit{A}}_3$, similarly, we have
\begin{align*}
	\boldsymbol{\mathit{A}}_3\leq C\frac{(\OffDiagS)^2}{\lambda_{\min}^2}\cdot\|[\Bbeta^*]_{\calS^*\setminus\calS_0}\|^2.
\end{align*}
We then consider $\boldsymbol{\mathit{A}}_4$. We notice that
\begin{multline*}
	\l\|\prod_{j=l}^{t}\l(1-\frac{\Ca}{j+\Cb s\log(2d/s)}\r)\frac{1}{l+\Cb s\log(2d/s)}\boldsymbol{\varXi}_1^{(l)}\bigg|\calF_l \r\|_{\Psi_1}\\
	\leq \frac{(l+\Cb s\log(2d/s))^{\Ca-1}}{(t+1+\Cb s\log(2d/s))^{\Ca}}\l(\lambda_{\max}\|[\Bbeta_{l}-\Bbeta^*]_{\calS_0}\|^2+\frac{(\OffDiagS)^2}{\lambda_{\max}}\|[\Bbeta^*]_{\calS^*\setminus\calS_0}\|^2 \r).
\end{multline*}
Then with Lemma~\ref{teclem:azuma}, under the event $\cup_{l=0}^t \calE_{l}$, we have
\begin{multline*}
	\PP\l(\l| \sum_{l=0}^t \prod_{j=l}^{t}\l(1-\frac{\Ca}{j+\Cb s\log(2d/s)}\r)\frac{1}{l+\Cb s\log(2d/s)}\boldsymbol{\varXi}_1^{(l)}\r|  \geq c_1\frac{C^*}{\Ca}\frac{s\log(2d/s)+\delta_{t+1}}{t+1+\Cb s}\sigma^2\r.\\\l.+c_1\frac{(\OffDiagS)^2}{\Ca\lambda_{\max}}\|[\Bbeta^*]_{\calS^*\setminus\calS_0}\|^2\r)
	\leq\exp\l(-c_2\frac{\Cb s\log(2d/s)+t}{\Ca}\frac{\lambda_{\min}^2}{\lambda_{\max}^2}\r),
\end{multline*}
where $c_1<0.1$ is some sufficiently small constant and it further implies
\begin{align*}
	|\boldsymbol{\mathit{A}}_4 |\leq c_1C^*\frac{s\log(2d/s)+\delta_{t}}{t+1+\Cb s\log(2d/s)}\frac{\sigma^2}{\lambda_{\min}}+c_1\l(\frac{\OffDiagS}{\lambda_{\min}}\r)^2\|[\Bbeta^*]_{\calS\setminus\calS_0}\|^2.
\end{align*}
Regarding term $\boldsymbol{\mathit{A}}_5$, we have
\begin{multline*}
	\l\|\prod_{j=l}^{t}\l(1-\frac{\Ca}{j+\Cb s\log(2d/s)}\r)\frac{1}{l+\Cb s\log(2d/s)}\boldsymbol{\varXi}_2^{(l)}  \big|\calF_{l}\r\|_{\Psi_1}\\
	\leq C\sqrt{\lambda_{\max}}\sigma\frac{(l+\Cb s\log(2d/s))^{\Ca-1}}{(t+1+\Cb s\log(2d/s))^{\Ca}}\cdot\|[\Bbeta_l-\Bbeta^*]_{\calS_0}\|.
\end{multline*}
Then, under event $\cup_{l=0}^t\{\calE_{l}\}$ and by Lemma~\ref{teclem:azuma}, we have
\begin{multline*}
	\PP\l(\l|\sum_{l=0}^{t}\prod_{j=l}^{t}\l(1-\frac{\Ca}{j+\Cb s\log(2d/s)}\r)\frac{1}{l+\Cb s\log(2d/s)}\boldsymbol{\varXi}_2^{(l)}\r| \r.\\ \l.\geq c_1\frac{C^*}{\Ca}\frac{(s\log(2d/s)+\delta_{t+1})\sigma^2}{t+1+\Cb s\log(2d/s)}+c_1\frac{1}{\Ca}\frac{(\OffDiagS)^2}{\lambda_{\min}}\|[\Bbeta^*]_{\calS^*\setminus\calS_0}\|^2\r)\\
	 \leq\exp\l(-c_2\min\l\{ (s\log(2d/s)+\delta_{t+1})\frac{C^*}{\Ca}\frac{\lambda_{\min}}{\lambda_{\max}},\sqrt{s\log(2d/s)+\delta_{t+1}}\sqrt{\Cb s\log(2d/s)+t}\frac{\sqrt{C^*}}{\Ca}\frac{\sqrt{\lambda_{\min}}}{\sqrt{\lambda_{\max}}}\r\}\r),
\end{multline*}
under which, we have,
\begin{align*}
	|\boldsymbol{\mathit{A}}_5|\leq c_1\frac{C^*}{\lambda_{\min}}\frac{s\log(2d/s)+\delta_{t+1}}{t+1+\Cb s\log(2d/s)}\sigma^2+c_1\l(\frac{\OffDiagS}{\lambda_{\min}}\r)^2\|[\Bbeta^*]_{\calS^*\setminus\calS_0}\|^2.
\end{align*}
We then only need to consider the term $\boldsymbol{\mathit{A}}_6$, and under $\calE_l^X$, it has
\begin{multline*}
	\l\|\prod_{j=l}^{t}\l(1-\frac{\Ca}{j+\Cb s\log(2d/s)}\r)\frac{1}{(l+\Cb s\log(2d/s))^2}\boldsymbol{\varXi}_3^{(l)}\bigg|\calF_{l},\calE_{l}^X\r\|_{\Psi_1}\\
	\leq C\frac{(l+\Cb s\log(2d/s))^{\Ca-2}}{(t+1+\Cb s\log(2d/s))^{\Ca}}(s\log(2d/s)+\delta_{l})\lambda_{\max}\sigma^2.
\end{multline*}
Thus, by Lemma~\ref{teclem:azuma}, under $\cup_{l=0}^t\calE_{l}^X$, it has
\begin{multline*}
	\PP\l(\l|\prod_{j=l}^{t}\l(1-\frac{\Ca}{j+\Cb s\log(2d/s)}\r)\frac{1}{(l+\Cb s\log(2d/s))^2}\boldsymbol{\varXi}_3^{(l)}  \r|\geq c_1C^*\frac{\lambda_{\min}}{\Ca^2}\frac{(s\log(2d/s)+\delta_{t+1})\sigma^2}{t+1+\Cb s\log(2d/s)}\r)\\
	\leq\exp\l(-c_2\l(s\log(2d/s)+\frac{t}{\Cb}\r)\min\l\{ \frac{\lambda_{\min}^2}{\lambda_{\max}^2}\frac{\Cb}{\Ca}\l(\frac{C^*}{\Ca}\r)^2,\frac{\lambda_{\min}}{\lambda_{\max}}\frac{C^*\Cb}{\Ca^2}\r\}\r),
\end{multline*}
which implies
\begin{align*}
	|\boldsymbol{\mathit{A}}_6|\leq c_1C^*\frac{1}{\lambda_{\min}}\frac{s\log(2d/s)+\delta_{t+1}}{t+1+\Cb s\log(2d/s)}\sigma^2.
\end{align*}
Thus, overall, under the events $\cup_{l=0}^t\{\calE_l,\calE_{l}^X\}$ and by the value of $\Cb\geq\Ca$ and $C^*$, with probability over $1-\exp(-c(\Cb s\log(2d/s)+ t)/\Ca)-\exp(-c(s\log(2d/s)+\delta_{t+1}))$, we have
\begin{align*}
	\|[\Bbeta_{t+1}-\Bbeta^*]_{\calS_0}\|^2\leq C^*\frac{1}{\lambda_{\min}}\frac{s\log(2d/s)+\delta_{t+1}}{t+1+\Cb s\log(2d/s)}\sigma^2+c\frac{(\OffDiagS)^2}{\lambda_{\min}^2}\cdot\|[\Bbeta^*]_{\calS^*\setminus\calS_0}\|^2,
\end{align*}
which completes the proof.
\subsection{Proof of Lemma~\ref{lem:sparse-gradient}}
The vector $\G$ can be written as the sum of two terms
\begin{align}
	\G(\{\Bbeta_t,\X_t,Y_t\}_{t\in\calT})= \sum_{t\in\calT} \X_t\X_t^{\top}\l(\Bbeta_t-\Bbeta^*\r)-\sum_{t\in\calT}\xi_t\cdot\X_t.
	\label{eq1:sparse-gradient}
\end{align}
We first analyze the second term, which contains the noise. It's worth noting that $\xi_t\X_t$ is martingale difference with respect to $\calF_t$, and it has $\|\xi_t\X_t|\calF_t\|_{\Psi_1}\leq \sigma\sqrt{\lambda_{\max}}$. Lemma~\ref{teclem:vec max norm} proves that for any $\delta>0$, with probability over $1-d\exp(-c\min\{\delta,\sqrt{\delta|\calT|}\})$, we have
\begin{align}
	\l\|\sum_{t\in\calT}\xi_t\X_t \r\|_{\infty}\leq C\sigma\sqrt{\lambda_{\max}}\sqrt{\delta|\calT|}.
	\label{eq2}
\end{align}
It is also noteworthy that there possibly exists dependence between $\X_t$ and $\calS_{0}$. By Lemma~\ref{teclem:Bernstein partial length}, for any $\calS\subseteq[d]$ with size $|\calS|\leq s$, we have
\begin{equation}
	\begin{split}
		&\PP\l(\l\|\calH_{\calS}\l(\sum_{t\in\calT}\xi_t\X_t\r) \r\|\geq C\sigma\sqrt{\lambda_{\max}}\sqrt{|\calT|}\sqrt{s\log(2d/s)+\delta}\r)\\
		&~~~~~~~~~~~~~~~~~~~~~~~~~~~~~~~~~~~~~~~~~~~~~~~~~~~~~~~\leq\exp\l(-c\min\l\{\delta, \sqrt{|\calT|/(s\log(2d/s))}\sqrt{\delta}\r\}\r).
	\end{split}
	\label{eq3}
\end{equation}
~\\
~\\
Then, it only remains to analyze the first term of \eqref{eq1:sparse-gradient}, namely, $ \sum_{t\in\calT} \X_t\X_t^{\top}\l(\Bbeta_t-\Bbeta^*\r)$. Specifically, we aim to provide an upper bound for entries located on $\calS_{\calT}^c\cap\calS^{*c}$, which represent entries that are not included by $\calS^*$ and to provide a lower bound for entries supported on $\calS^*\setminus\calS_{\calT}$. Firstly, for $i\in \calS_{\calT}^c\cap\calS^{*c}$, we decompose it as follows
\begin{multline*}
	[\sum_{t\in\calT} \X_t\X_t^{\top}\l(\Bbeta_t-\Bbeta^*\r)]_{i}=\sum_{t\in\calT} [\X_t]_{i}\X_t^{\top}\l(\Bbeta_t-\Bbeta^*\r)\\
	=\underbrace{\sum_{t\in\calT}\EE\l\{ [\X_t]_{i}\X_t^{\top}\l(\Bbeta_t-\Bbeta^*\r)\big|\calF_t\r\}}_{\D_1}+\underbrace{\sum_{t\in\calT}\l\{ [\X_t]_{i}\X_t^{\top}\l(\Bbeta_t-\Bbeta^*\r)-\EE\l\{ [\X_t]_{i}\X_t^{\top}\l(\Bbeta_t-\Bbeta^*\r)\big|\calF_t\r\}\r\}}_{\D_2}.
\end{multline*}
We first bound term $\D_1$. Recall the definition of $\OffDiagO$. Under the event $\calE_{\{\calS_{\calT},V_{\calT},W_{\calT}\}}$, we have
\begin{multline*}
	\l|\sum_{t\in\calT}\EE\l\{ [\X_t]_{i}\X_t^{\top}\l(\Bbeta_t-\Bbeta^*\r)\big|\calF_t\r\}\r|\\%=\l| \sum_{t\in\calT} \bSigma_{i,\calS\cup\calS_{\calT}}\calH_{\calS\cup\calS_{\calT}}\l(\Bbeta_t-\Bbeta^*\r) \r|\\
	\leq \sum_{t\in\calT}\l\|\EE\l\{ [\X_t]_{i}[\X_t]_{\calS^*\cup\calS_{\calT}}\r\}\big|\calF_t\r\| \cdot\| \Bbeta_t-\Bbeta^*\|
	\leq \OffDiagO \sum_{t\in\calT}\| \Bbeta_t-\Bbeta^*\|\leq \OffDiagO\sum_{t\in\calT}V_{t}.
\end{multline*}
We then consider term $\D_2$. It is worth noting that $\D_2$ is sum of $|\calT|$ martingale differences with respect to $\{\calF_t\}$ and under event $\calE_{\{\calS_{\calT},D_\calT,U_\calT\}}$, each single term of $\D_2$ satisfies
\begin{align*}
	\l\|\l\{ [\X_t]_{i}\X_t^{\top}\l(\Bbeta_t-\Bbeta^*\r)-\EE\l\{ [\X_t]_{i}\X_t^{\top}\l(\Bbeta_t-\Bbeta^*\r)\big|\calF_t\r\}\r\}\big|\calF_t \r\|_{\Psi_1}\leq\OffDiagO V_{t}.
\end{align*}
Thus, by Lemma~\ref{teclem:azuma}, we have
\begin{multline*}
	\PP\l( \l| \sum_{t\in\calT}\l\{ [\X_t]_{i}\X_t^{\top}\l(\Bbeta_t-\Bbeta^*\r)-\EE\l\{ [\X_t]_{i}\X_t^{\top}\l(\Bbeta_t-\Bbeta^*\r)\big|\calF_t\r\}\r\}\r|\geq u\r)\\
	\leq\exp\l(-\min\l\{\frac{u^2}{(\OffDiagO)^2\sum_{t\in\calT}V_t^2},\frac{u}{2\OffDiagO\max_t V_t}\r\}\r).
\end{multline*}
Putting the above two equations together and taking the uniform for $\calS_{\calT}^c\cap\calS^{*c}$, we have
\begin{multline*}
	\PP\l(\l\|[\sum_{t\in\calT} \X_t\X_t^{\top}\l(\Bbeta_t-\Bbeta^*\r)]_{\calS^{*c}\cap\calS_{\calT}^c}\r\|_{\infty} \geq \OffDiagO\sum_{t\in\calT} V_t+u\r)\\
	\leq d\exp\l(-\min\l\{\frac{u^2}{(\OffDiagO)^2\sum_{t\in\calT}V_t^2},\frac{u}{2\OffDiagO\max_t V_t}\r\}\r).
\end{multline*}
Thus, in all, together with Equation~\eqref{eq2}, we have the entrywise norm for the vector $	\G(\{\Bbeta_t,\X_t,Y_t\}_{t\in\calT})$,
\begin{multline*}
	\PP\l( \l\|[\G(\{\Bbeta_t,\X_t,Y_t\}_{t\in\calT})]_{\calS^{*c}\cap \calS_{\calT}^c} \r\|_{\infty}\geq \OffDiagO\sum_{t\in\calT} V_t+u+C\sigma\sqrt{\lambda_{\max}} \sqrt{\delta|\calT|}\r)\\
	 \leq d\exp\l(-\min\l\{\frac{u^2}{(\OffDiagO)^2\sum_{t\in\calT}V_t^2},\frac{u}{2\OffDiagO\max_t V_t}\r\}\r)+d\exp(-c\min\{\delta,\sqrt{\delta|\calT|}\}).
\end{multline*}
We then consider the lower bound of the vector supported on $\calS^*\setminus\calS_{\calT}$. We still decompose it into the conditional expectation term and the martingale difference term,
\begin{multline*}
	[\sum_{t\in\calT} \X_t\X_t^{\top}\l(\Bbeta_t-\Bbeta^*\r)]_{\calS^*\setminus\calS_{\calT}}%=\sum_{t\in\calT} [\X_t]_{\calS_{\calT}}\X_t^{\top}\l(\Bbeta_t-\Bbeta^*\r)\\
	=\underbrace{\sum_{t\in\calT}\EE\l\{ [\X_t]_{\calS^*\setminus\calS_{\calT}}\X_t^{\top}\l(\Bbeta_t-\Bbeta^*\r)\big|\calF_t\r\}}_{\E_1} \\+ \underbrace{\sum_{t\in\calT} \l\{[\X_t]_{\calS^*\setminus\calS_{\calT}}\X_t^{\top}\l(\Bbeta_t-\Bbeta^*\r) - \EE\l\{ [\X_t]_{\calS^*\setminus\calS_{\calT}}\X_t^{\top}\l(\Bbeta_t-\Bbeta^*\r)\big|\calF_t\r\}\r\}}_{\E_2}.
\end{multline*}
We first consider the term $\E_1$. Notice that $\Bbeta_{t}-\Bbeta^*$ is supported on $\calS_{\calT}\cup\calS^*$ and we further decompose it as follows,
\begin{multline*}
	\l\| \sum_{t\in\calT}\EE\l\{ [\X_t]_{\calS^*\setminus\calS_{\calT}}\X_t^{\top}\l(\Bbeta_t-\Bbeta^*\r)\big|\calF_t\r\} \r\|=\l\| \sum_{t\in\calT}\EE\l\{ [\X_t]_{\calS^*\setminus\calS_{\calT}} [\X_t]_{\calS^*\cup\calS_{\calT}}^{\top}\big|\calF_t\r\}\l(\Bbeta_t-\Bbeta^*\r) \r\|\\
	\geq \l\| \sum_{t\in\calT}\EE\l\{ [\X_t]_{\calS^*\setminus\calS_{\calT}} [\X_t]_{ \calS^*\setminus\calS_{\calT}}^{\top}\big|\calF_t\r\} [\Bbeta^*]_{ \calS^*\setminus\calS_{\calT}}\r\|-\l\| \sum_{t\in\calT}\EE\l\{ [\X_t]_{\calS^*\setminus\calS_{\calT}} [\X_t]_{ \calS_{\calT}}^{\top}\big|\calF_t\r\}[\Bbeta_{t}-\Bbeta^*]_{ \calS_{\calT}} \r\|.
\end{multline*}
By Assumption~\ref{assm:cov-sparse} and Lemma~\ref{teclem:min eigenvalue}, we have
\begin{align*}
	 \l\| \sum_{t\in\calT}\EE\l\{ [\X_t]_{\calS^*\setminus\calS_{\calT}} [\X_t]_{ \calS^*\setminus\calS_{\calT}}^{\top}\big|\calF_t\r\} [\Bbeta^*]_{ \calS^*\setminus\calS_{\calT}}\r\|\geq |\calT|\cdot\lambda_{\min}\| [\Bbeta^*]_{\calS^*\setminus\calS_{\calT}}\|,
\end{align*}
and on the other hand, under the event $\calE_{\{\calS_{\calT},D_\calT,U_\calT\}}$, we have
\begin{align*}
\l\| \sum_{t\in\calT}\EE\l\{ [\X_t]_{\calS^*\setminus\calS_{\calT}} [\X_t]_{ \calS_{\calT}}^{\top}\big|\calF_t\r\}[\Bbeta_{t}-\Bbeta^*]_{ \calS_{\calT}} \r\|\leq \OffDiagS\sum_{t\in\calT} W_t.
\end{align*}
Thus in all, we have the lower bound for $\|\E_1\|$,
\begin{align*}
	\|\E_1\|\geq |\calT|\cdot\lambda_{\min}\| [\Bbeta^*]_{\calS^*\setminus\calS_{\calT}}\|-\OffDiagS\sum_{t\in\calT} W_t.
\end{align*}
Then, we are going to find the upper bound for $\E_2$'s norm. We notice that
\begin{multline*}
	\l\|\l\{[\X_t]_{\calS^*\setminus\calS_{\calT}}\X_t^{\top}\l(\Bbeta_t-\Bbeta^*\r) - \EE\l\{ [\X_t]_{\calS^*\setminus\calS_{\calT}}\X_t^{\top}\l(\Bbeta_t-\Bbeta^*\r)\big|\calF_t\r\}\r\}\big|\calF_t \r\|_{\Psi_1}\\
	\leq \OffDiagS W_t+\lambda_{\max}\|[\Bbeta^*]_{\calS^*\setminus\calS_{\calT}}\|.
\end{multline*}
Hence, by Lemma~\ref{teclem:Bernstein Martingale}, we have
\begin{multline*}
	\PP\l( \l\|\sum_{t\in\calT} \l\{[\X_t]_{\calS^*\setminus\calS_{\calT}}\X_t^{\top}\l(\Bbeta_t-\Bbeta^*\r) - \EE\l\{ [\X_t]_{\calS^*\setminus\calS_{\calT}}\X_t^{\top}\l(\Bbeta_t-\Bbeta^*\r)\big|\calF_t\r\}\r\}\r\|\geq 2u  \r)\leq\exp\l(Cs\r)\\
	\times\exp\l(-\min\l\{\frac{u^2}{(\OffDiagS)^2\sum_{t\in\calT} W_t^2+|\calT|\lambda_{\max}^2\|[\Bbeta^*]_{\calS^*\setminus\calS_{\calT}}\|^2},\frac{u}{\OffDiagS\max_{t\in\calT} W_t+\lambda_{\max}\|[\Bbeta^*]_{\calS^*\setminus\calS_{\calT}}\|}\r\}\r).
\end{multline*}
Thus, combining these inequalities, we have
\begin{multline*}
	\PP\l( \l\|[\sum_{t\in\calT} \X_t\X_t^{\top}\l(\Bbeta_t-\Bbeta^*\r)]_{\calS^*\setminus\calS_{\calT}} \r\| \geq |\calT|\cdot\lambda_{\min}\| [\Bbeta^*]_{\calS^*\setminus\calS_{\calT}}\|-\OffDiagS\sum_{t\in\calT} W_t -u \r)\geq 1-\\
    \exp\l(Cs-\min\l\{\frac{u^2}{(\OffDiagS)^2\sum_{t\in\calT} W_t^2+|\calT|\lambda_{\max}^2\|[\Bbeta^*]_{\calS^*\setminus\calS_\calT}\|^2},\frac{u}{\OffDiagS\max_{t\in\calT} W_t+\lambda_{\max}\|[\Bbeta^*]_{\calS^*\setminus\calS_\calT}\|}\r\}\r).
\end{multline*}
Together with Equation~\ref{eq3}, we have
\begin{multline*}
	\PP\l( \l\|[\G(\{\Bbeta_t,\X_t,Y_t\}_{t\in\calT})]_{\calS^*\setminus\calS_{\calT}} \r\| \geq |\calT|\lambda_{\min}\| [\Bbeta^*]_{\calS^*\setminus\calS_{\calT}}\|-\OffDiagS\sum_{t\in\calT} W_t\r. \\ \l.-u-C\sigma\sqrt{\lambda_{\max}}\sqrt{|\calT|}\sqrt{|\calS^*\setminus\calS_{\calT}|\log(2d/|\calS^*\setminus\calS_{\calT}|)+\delta} \r)
	\geq 1-\exp(Cs)\\
    \times\exp\l(-\min\l\{\frac{u^2}{(\OffDiagS)^2\sum_{t\in\calT} W_t^2+|\calT|\lambda_{\max}^2\|[\Bbeta^*]_{\calS^*\setminus\calS_\calT}\|^2},\frac{u}{\OffDiagS\max_{t\in\calT} W_t+\lambda_{\max}\|[\Bbeta^*]_{\calS^*\setminus\calS_\calT}\|}\r\}\r)\\-\exp\l(-c\min\l\{\delta, \sqrt{\delta}\sqrt{|\calT|/ (|\calS^*\setminus\calS_{\calT}|\log(2d/|\calS^*\setminus\calS_{\calT}|))}\r\}\r),
\end{multline*}
which completes the proof.

\subsection{Proof of Proposition~\ref{prop:sparse-select}}
We prove Proposition~\ref{prop:sparse-select} by induction. Lemma~\ref{lem:FixedSupport} proves the iterates' dynamics of $t\in[0,\tau_1-1]$. Here, we only prove the iterates in the interval $[\tau_1,\tau_2-1]$ and for general $[\tau_l,\tau_l-1]$, it can be proved in the same way only with more notations. During  $t\in[\tau_1,\tau_2-1]$, the update is
$$\Bbeta_{t+1}=\calH_{\calS_1}(\Bbeta_{t}-\eta_t\cdot\g_t),\quad t\in[\tau_1,\tau_2-1].$$
During this procedure, we have
$$\|[\Bbeta_{t}-\Bbeta^*]_{\calS^*\setminus\calS_1} \|=\|[\Bbeta_{0}-\Bbeta^*]_{\calS^*\setminus\calS_1} \|.$$
For entries on $\calS_1\setminus\calS^*$, using a similar analysis to the proof of Lemma~\ref{lem:FixedSupport}, we have
\begin{multline*}
	\|[\Bbeta_{t+1}-\Bbeta^* ]_{ \calS_1}\|^2
	\leq \l(1-\eta_{t}\lambda_{\min}\r)\l\|\l[\Bbeta_{t}-\Bbeta^* \r]_{ \calS_1}\r\|^2+2\eta_{t}^2\lambda_{\max}s\sigma^2
	\\
	+\eta_{t}\|[\Bbeta^*]_{\calS^*\setminus\calS_1}\|^2\cdot\frac{(\OffDiagS)^2}{\lambda_{\min}}-2\eta_{t}\boldsymbol{\varXi}_1^{(t)}+2\eta_{t}\boldsymbol{\varXi}_2^{(t)}+\eta_{t}^2\boldsymbol{\varXi}_3^{(t)},
\end{multline*}
where $\boldsymbol{\varXi}_1^{(t)},\boldsymbol{\varXi}_2^{(t)},\boldsymbol{\varXi}_3^{(t)}$ are inherited from the the proof of Lemma~\ref{lem:FixedSupport} with the substitution of $\calS_1$ in $\calS_0$ accordingly. We accumulate the above equation from $t=\tau_1$ and it arrives at
\begin{equation}
    \begin{split}
        &\|[\Bbeta_{t+1}-\Bbeta^* ]_{ \calS_1}\|^2
	\leq  \prod_{j=\tau_1}^{t}\l(1-\frac{\Ca}{j+\Cb s\log(2d/s)}\r)\l\|\l[\Bbeta_{\tau_1}-\Bbeta^* \r]_{ \calS_1}\r\|^2\\
	&~~~~~+2s\sigma^2\Ca^2\frac{\lambda_{\max}}{\lambda_{\min}^2}\sum_{l=\tau_1}^{t}\prod_{j=l}^{t}\l(1-\frac{\Ca}{j+\Cb s\log(2d/s)}\r)\l(\frac{1}{l+\Cb s\log(2d/s)}\r)^2\\
	&~~~~~+\Ca\sum_{l=\tau_1}^{t}\prod_{j=l}^{t}\l(1-\frac{\Ca}{j+\Cb s\log(2d/s)}\r)\frac{1}{l+\Cb s\log(2d/s)}\|[\Bbeta^*]_{\calS^*\setminus\calS_1}\|^2\cdot\frac{(\OffDiagS)^2}{\lambda_{\min}^2}\\
	&~~~~~-2\frac{\Ca}{\lambda_{\min}}\sum_{l=\tau_1}^{t}\prod_{j=l}^{t}\l(1-\frac{\Ca}{j+\Cb s\log(2d/s)}\r)\frac{1}{l+\Cb s\log(2d/s)}\boldsymbol{\varXi}_1^{(l)}\\
	&~~~~~+2\frac{\Ca}{\lambda_{\min}}\sum_{l=\tau_1}^{t}\prod_{j=l}^{t}\l(1-\frac{\Ca}{j+\Cb s\log(2d/s)}\r)\frac{1}{l+\Cb s\log(2d/s)}\boldsymbol{\varXi}_2^{(l)}\\
	&~~~~~+\frac{\Ca^2}{\lambda_{\min}^2}\sum_{l=\tau_1}^{t}\prod_{j=l}^{t}\l(1-\frac{\Ca}{j+\Cb s\log(2d/s)}\r)\frac{1}{(l+\Cb s\log(2d/s))^2}\boldsymbol{\varXi}_3^{(l)}.
    \end{split}
    \label{eq1:prop:sparse-select}
\end{equation}
It is worth noting that at $\tau_1$, it has
$$\l\|\l[\Bbeta_{\tau_1}-\Bbeta^* \r]_{ \calS_1}\r\|^2=\l\|\l[\Bbeta_{\tau_1}-\Bbeta^* \r]_{ \calS_0}\r\|^2+\big\|[\Bbeta^*]_{\calS_1\setminus\calS_0} \big\|^2.$$
We then insert the above equation of $\l\|\l[\Bbeta_{\tau_1}-\Bbeta^* \r]_{ \calS_0}\r\|^2$ into the Equation~\ref{eq1:prop:sparse-select} and accumulate the update until $t=0$, which yields
\begin{multline*}
	\|[\Bbeta_{t+1}-\Bbeta^* ]_{ \calS_1}\|^2
	\leq  \underbrace{\prod_{j=0}^{t}\l(1-\frac{\Ca}{j+\Cb s\log(2d/s)}\r)\l\|\l[\Bbeta_{0}-\Bbeta^* \r]_{ \calS_0}\r\|^2}_{\boldsymbol{\mathit{B}}_1}\\
	+\underbrace{\prod_{j=\tau_1}^{t}\l(1-\frac{\Ca}{j+\Cb s\log(2d/s)}\r)\l\|\l[\Bbeta^* \r]_{ \calS_1\setminus\calS_{0}}\r\|^2}_{\boldsymbol{\mathit{B}}_2}\\
	+\underbrace{2s\sigma^2\Ca^2\frac{\lambda_{\max}}{\lambda_{\min}^2}\sum_{l=0}^{t}\prod_{j=l}^{t}\l(1-\frac{\Ca}{j+\Cb s\log(2d/s)}\r)\l(\frac{1}{l+\Cb s\log(2d/s)}\r)^2}_{\boldsymbol{\mathit{B}}_3}\\
	+\underbrace{\Ca\sum_{l=0}^{t}\prod_{j=l}^{t}\l(1-\frac{\Ca}{j+\Cb s\log(2d/s)}\r)\frac{1}{l+\Cb s\log(2d/s)}\|[\Bbeta^*]_{\calS^*\setminus\calS_l}\|^2\cdot\frac{(\OffDiagS)^2}{\lambda_{\min}^2}}_{\boldsymbol{\mathit{B}}_4}\\
	-\underbrace{2\frac{\Ca}{\lambda_{\min}}\sum_{l=0}^{t}\prod_{j=l}^{t}\l(1-\frac{\Ca}{j+\Cb s\log(2d/s)}\r)\frac{1}{l+\Cb s\log(2d/s)}\boldsymbol{\varXi}_1^{(l)}}_{\boldsymbol{\mathit{B}}_5}\\
	+\underbrace{2\frac{\Ca}{\lambda_{\min}}\sum_{l=0}^{t}\prod_{j=l}^{t}\l(1-\frac{\Ca}{j+\Cb s\log(2d/s)}\r)\frac{1}{l+\Cb s\log(2d/s)}\boldsymbol{\varXi}_2^{(l)}}_{\boldsymbol{\mathit{B}}_6}\\
	+\underbrace{\frac{\Ca^2}{\lambda_{\min}^2}\sum_{l=0}^{t}\prod_{j=l}^{t}\l(1-\frac{\Ca}{j+\Cb s\log(2d/s)}\r)\frac{1}{(l+\Cb s\log(2d/s))^2}\boldsymbol{\varXi}_3^{(l)}}_{\boldsymbol{\mathit{B}}_7}.
\end{multline*}
Firstly, we bound the term $\boldsymbol{\mathit{B}}_1$ via Lemma~\ref{teclem:prod upper bound},
\begin{align*}
	\boldsymbol{\mathit{B}}_1\leq \frac{\Cb s\log(2d/s)}{t+1+\Cb s\log(2d/s)}\|[\Bbeta_{0}-\Bbeta^*]_{\calS_{0}}\|^2\leq\frac{\Cb s\log(2d/s)}{t+1+\Cb s\log(2d/s)}\frac{\sigma^2}{\lambda_{\max}}.
\end{align*}
Regarding term $\boldsymbol{\mathit{B}}_2$, we still apply Lemma~\ref{teclem:prod upper bound} and it has
\begin{align*}
	\boldsymbol{\mathit{B}}_2\leq \l(\frac{\tau_1+\Cb s\log(2d/s)}{t+\Cb s\log(2d/s)}\r)^{\Ca}\|[\Bbeta^*]_{\calS_1\setminus\calS_{0}}\|^2.
\end{align*}
As for term $\B_3$, it has
\begin{multline*}
	\prod_{j=l}^{t}\l(1-\frac{\Ca}{j+\Cb s\log(2d/s)}\r)\l(\frac{1}{l+\Cb s\log(2d/s)}\r)^2\sigma^2\\
	\leq \l(\frac{1}{t+1+\Cb s\log(2d/s)}\r)^{\Ca}\l(l+\Cb s\log(2d/s)\r)^{\Ca-2}\sigma^2.
\end{multline*}
Then, by Lemma~\ref{teclem:sum upper bound}, it can be bounded with
\begin{align*}
	\boldsymbol{\mathit{B}}_3\leq \Ca\frac{\lambda_{\max}}{\lambda_{\min}^2}\frac{2s}{t+1+\Cb s\log(2d/s)}\sigma^2.
\end{align*}
In a similar fashion, we have the following bound for $\boldsymbol{\mathit{B}}_4$,
\begin{align*}
	\boldsymbol{\mathit{B}}_4\leq\l(\frac{\OffDiagS}{\lambda_{\min}}\r)^2\cdot\|[\Bbeta^*]_{\calS^*\setminus\calS_1} \|^2 + \l(\frac{\tau_1+\Cb s\log(2d/s)}{t+\Cb s\log(2d/s)}\r)^{\Ca}\l(\frac{\OffDiagS}{\lambda_{\min}}\r)^2\cdot\|[\Bbeta^*]_{\calS_1\setminus\calS_0} \|^2.
\end{align*}
We then consider the term $\boldsymbol{\mathit{B}}_5$. Specifically, the conditional Orlicz norm can be bounded with
\begin{align*}
	\|\boldsymbol{\varXi}_1^{(l)}|\calF_l\|_{\Psi_1}&\leq \lambda_{\max}\|[\Bbeta_{l}-\Bbeta^*]_{\calS_l}\|^2+\OffDiagS\cdot\|[\Bbeta_{l}-\Bbeta^*]_{\calS_l}\|\cdot\|[\Bbeta^*]_{\calS^*\setminus\calS_l}\|\\
	&\leq 2\lambda_{\max}\|[\Bbeta_{l}-\Bbeta^*]_{\calS_l}\|^2+\frac{(\OffDiagS)^2}{\lambda_{\max}}\|[\Bbeta^*]_{\calS^*\setminus\calS_l}\|^2.
\end{align*}
By Lemma~\ref{teclem:azuma}, under event $\cup_{l=0}^t\{\calE_{l}\}$, we have
\begin{multline*}
		\PP\l(\l| \sum_{l=0}^t \prod_{j=l}^{t}\l(1-\frac{\Ca}{j+\Cb s\log(2d/s)}\r)\frac{1}{l+\Cb s\log(2d/s)}\boldsymbol{\varXi}_1^{(l)}\r|\geq c_1\frac{C^*}{\Ca}\frac{(s\log(2d/s)+\delta_{t})\sigma^2}{t+1+\Cb s\log(2d/s)} \r.\\\l. +\frac{c_1}{\Ca}\frac{(\OffDiagS)^2}{\lambda_{\min}}\|[\Bbeta^*]_{\calS^*\setminus\calS_1}\|^2+\lambda_{\min}\frac{c_1}{\Ca}\l(\frac{\tau_1+\Cb s\log(2d/s)}{t+1+\Cb s\log(2d/s)}\r)^{\Ca-2} \|[\Bbeta^*]_{\calS_1\setminus\calS_{0}}\|^2\r)\\
	\leq\exp\l(-c_2\frac{\Cb s\log(2d/s)+t}{\Ca}\frac{\lambda_{\min}^2}{\lambda_{\max}^2}\r),
\end{multline*}
under which we have \begin{multline*}
	|\boldsymbol{\mathit{B}}_5|\leq c_1\frac{C^*}{\lambda_{\min}}\frac{(s\log(2d/s)+\delta_{t})\sigma^2}{t+1+\Cb s\log(2d/s)}+c_1\l(\frac{\OffDiagS}{\lambda_{\min}}\r)^2\|[\Bbeta^*]_{\calS^*\setminus\calS_1}\|^2\\+c_1\l(\frac{\tau_1+\Cb s\log(2d/s)}{t+1+\Cb s\log(2d/s)}\r)^{\Ca-2}\cdot\|[\Bbeta^*]_{\calS_1\setminus\calS_{0}}\|^2.
\end{multline*}
As for the term $\boldsymbol{\mathit{B}}_6$, we have
\begin{align*}
	\|\boldsymbol{\varXi}_2^{(t)}|\calF_t\|_{\Psi_1}\leq\sqrt{\lambda_{\max}}\sigma\|[\Bbeta_{t}-\Bbeta^*]_{\calS_t}\|,
\end{align*}
which leads to
\begin{multline*}
	\l\|\prod_{j=l}^{t}\l(1-\frac{\Ca}{j+\Cb s\log(2d/s)}\r)\frac{1}{l+\Cb s\log(2d/s)}\boldsymbol{\varXi}_2^{(l)}  \big|\calF_{l}\r\|_{\Psi_1}\\
	\leq C\sqrt{\lambda_{\max}}\sigma\frac{(l+\Cb s\log(2d/s))^{\Ca-1}}{(t+1+\Cb s\log(2d/s))^{\Ca}}\cdot\|[\Bbeta_l-\Bbeta^*]_{\calS_l}\|.
\end{multline*}
Then under event $\cup_{l=0}^t\{\calE_{l}\}$, by Lemma~\ref{teclem:azuma}, we have
\begin{multline*}
	\PP\l(\l|\sum_{l=0}^{t}\prod_{j=l}^{t}\l(1-\frac{\Ca}{j+\Cb s\log(2d/s)}\r)\frac{1}{l+\Cb s\log(2d/s)}\boldsymbol{\varXi}_2^{(l)}\r| \geq c_1\frac{C^*}{\Ca}\frac{(s\log(2d/s)+\delta_{t+1})\sigma^2}{t+1+\Cb s\log(2d/s)}\r.\\
	 \l.+\frac{c_1}{\Ca}\frac{(\OffDiagS)^2}{\lambda_{\min}}\|[\Bbeta^*]_{\calS^*\setminus\calS_1}\|^2+\lambda_{\min}\frac{c_1}{\Ca}\l(\frac{\tau_1+\Cb s\log(2d/s)}{t+1+\Cb s\log(2d/s)}\r)^{\Ca-2}\|[\Bbeta^*]_{\calS_1\setminus\calS_0}\|^2\r)\\
	\leq\exp\l(-c_2\min\l\{ (s\log(2d/s)+\delta_{t+1})\frac{C^*}{\Ca}\frac{\lambda_{\min}}{\lambda_{\max}},\sqrt{s\log(2d/s)+\delta_{t+1}}\sqrt{\Cb s\log(2d/s)+t}\frac{\sqrt{C^*}}{\Ca}\frac{\sqrt{\lambda_{\min}}}{\sqrt{\lambda_{\max}}}\r\}\r),
\end{multline*}
which implies the following bound
\begin{multline*}
	|\boldsymbol{\mathit{B}}_6|\leq c_1\frac{C^*}{\lambda_{\min}}\frac{(s\log(2d/s)+\delta_{t+1})\sigma^2}{t+1+\Cb s\log(2d/s)}\\
	+c_1\frac{(\OffDiagS)^2}{\lambda_{\min}^2}\|[\Bbeta^*]_{\calS^*\setminus\calS_1}\|^2+c_1\l(\frac{\tau_1+\Cb s\log(2d/s)}{t+1+\Cb s\log(2d/s)}\r)^{\Ca-2}\cdot\|[\Bbeta^*]_{\calS_1\setminus\calS_0}\|^2.
\end{multline*}  Recall the definition $\calE_t^X:=\{\max_{|\calS|\leq s}\|[\X_t]_{\calS}\|^2\leq C\lambda_{\max}(s\log(2d/s)+\delta_{t})\}$ (defined in Section~\ref{sec:proof-FixedSupport}). Hence, under $\calE_t^X$, we have
\begin{align*}
	\|\boldsymbol{\varXi}_3^{(t)}|\calF_t,\calE_t^X\|_{\Psi_1}\leq C(s\log(2d/s)+\delta_{t})\lambda_{\max}\sigma^2.
\end{align*}
With a similar analysis to Section~\ref{sec:proof-FixedSupport}, we have with probability over $1-\exp(-c_2(s\log(2d/s)+t/\Cb))$, $$|\boldsymbol{\mathit{B}}_7|\leq c_1C^*\frac{1}{\lambda_{\min}}\frac{s\log(2d/s)+\delta_{t+1}}{t+1+\Cb s\log(2d/s)}\sigma^2.$$
Thus, altogether, we finished proving the convergence dynamics for $t\in[\tau_1,\tau_2-1]$,
\begin{multline*}
	\|[\Bbeta_{t+1}-\Bbeta^*]_{\calS_1}\|^2\leq C^*\frac{s\log(2d/s)}{t+1+\Cb s\log(2d/s)}\sigma^2+C\l(\frac{\OffDiagS}{\lambda_{\min}}\r)^2\|[\Bbeta^*]_{\calS^*\setminus\calS_1}\|^2\\+C\l(\frac{\tau_1+\Cb s\log(2d/s)}{t+1+\Cb s\log(2d/s)}\r)^{\Ca-2}\|[\Bbeta^*]_{\calS_1\setminus\calS_{0}}\|^2.
\end{multline*}
\section{Proofs in Section~\ref{sec:bandit}}
\label{sec:proof-bandit}
\subsection{Proof of Lemma~\ref{lem:decision region}}
\begin{proof}
	Recall that $\max_{i\in[K]}\|\Bbeta_{i}-\Bbeta_{i}^*\|\leq h_0$. Then by triangle inequality, we have
	\begin{align*}
		\l|\frac{\X^\top\Bbeta_{i}^*}{\|\X\|}- \frac{\X^\top\Bbeta_{i}}{\|\X\|}\r|\leq\|\Bbeta_{i}-\Bbeta_{i}^*\|\leq h_0,\quad \l|\max_{j\neq i}\frac{\X^\top\Bbeta_{j}^*}{\|\X\|}-\max_{j\neq i}\frac{\X^\top\Bbeta_{j}}{\|\X\|}\r|\leq\max_{j\neq i}\|\Bbeta_{j}-\Bbeta_{j}^*\|\leq h_0.
	\end{align*}
	 Hence, for $\X\in\calX\l(i,\l\{\Bbeta_{j}\r\}\r)$, it has
	\begin{align*}
		\frac{\X^\top\Bbeta_i^*}{\|\X\|}&= \frac{\X^\top\Bbeta_i}{\|\X\|}+\l(\frac{\X^\top\Bbeta_i^*}{\|\X\|}-\frac{\X^\top\Bbeta_i}{\|\X\|}\r)\geq \max_{j\neq i}\frac{\X^\top\Bbeta_j}{\|\X\|}-h_0\geq\max_{j\neq i}\frac{\X^\top\Bbeta_j^*}{\|\X\|}-2h_0,
	\end{align*}
which implies $\X\in\calU_i(-2h_0)$. Hence, we finish proving $\calX\l(i,\l\{\Bbeta_{j}\r\}\r)\subseteq \calU_i(-2h_0)$. On the other hand, if $\X\in\calU_i(2h_0)$, then we have
	\begin{align*}
		\frac{\X^\top\Bbeta_{i}}{\|\X\|}= \frac{\X^\top\Bbeta_i^*}{\|\X\|}+\l(\frac{\X^\top\Bbeta_i}{\|\X\|}-\frac{\X^\top\Bbeta_i^*}{\|\X\|}\r)\geq\max_{j\neq i}\frac{\X^\top\Bbeta_{j}^*}{\|\X\|}+2h_0-h_0\geq\max_{j\neq i}\frac{\X^\top\Bbeta_{j}}{\|\X\|},
	\end{align*}
	which shows $\X\in\calX\l(i,\l\{\Bbeta_{j}\r\}\r)$. Thus, we complete the proof.
\end{proof}
\subsection{Proof of Lemma~\ref{lem:regularity}}
\begin{proof}[Proof of Lemma~\ref{lem:regularity}]
We first prove claim $(1)$ of Lemma~\ref{lem:regularity}. For $j\in\calA^c$, we have
	\begin{align*}
		\frac{\X_t^{\top}\Bbeta_{j}}{\|\X_t\|}=\frac{\X_t^{\top}\Bbeta_{j}^*}{\|\X_t\|}+\frac{\X_t^{\top}\l(\Bbeta_{j}-\Bbeta_{j}^*\r)}{\|\X_t\|}= \frac{\max_{i\in[K]}\X_t^{\top}\Bbeta_{i}^*}{\|\X_t\|}-\frac{\max_{i\in[K]}\X_t^{\top}\Bbeta_{i}^*-\X_t^\top\Bbeta_{j}^*}{\|\X_t\|}+\frac{\X_t^{\top}\l(\Bbeta_{j}-\Bbeta_{j}^*\r)}{\|\X_t\|}.
	\end{align*}
	Assumption~\ref{assm:arm opt} guarantees $\frac{\max_{i\in[K]}\X_t^{\top}\Bbeta_{i}^*-\X_t^\top\Bbeta_{j}^*}{\|\X_t\|}\geq h$, a.s.. Thus, together with $\X_t^{\top}\l(\Bbeta_{j}-\Bbeta_{j}^*\r)\leq\|\X_t\|\cdot\|\Bbeta_{i}-\Bbeta_{i}^*\| $ and under the event $\{\|\Bbeta_{i}-\Bbeta_{i}^*\|\leq \frac{h}{2},\ i\in[K]\}$, it arrives at
	\begin{align*}
		\frac{\X_t^{\top}\Bbeta_{j}}{\|\X_t\|}\leq \frac{\max_{i\in[K]}\X_t^{\top}\Bbeta_{i}^*}{\|\X_t\|} -\frac{h}{2}< \frac{\max_{i\in[K]}\X_t^{\top}\Bbeta_{i}^*}{\|\X_t\|},
	\end{align*}
	which implies $\calX(j,\{\Bbeta_i\})=\emptyset$ for $j\in\calA^c$. 

Denote $h_{\calA}:=\max_{i\in\calA}\|\Bbeta_i-\Bbeta_i^*\|$. Further, under the event $\{\|\Bbeta_{i}-\Bbeta_{i}^*\|\leq \frac{h}{2},\ i\in[K]\}$, for any $i\in\calA$, if $\X\in\calX\l(i,\l\{\Bbeta_{j}\r\}\r)$, it has 
	\begin{align*}
		\frac{\X^\top\Bbeta_i^*}{\|\X\|}&= \frac{\X^\top\Bbeta_i}{\|\X\|}+\l(\frac{\X^\top\Bbeta_i^*}{\|\X\|}-\frac{\X^\top\Bbeta_i}{\|\X\|}\r)\geq \max_{j\neq i,j\in\calA}\frac{\X^\top\Bbeta_j}{\|\X\|}-h_{\calA}\geq\max_{j\neq i,j\in\calA}\frac{\X^\top\Bbeta_j^*}{\|\X\|}-2h_{\calA},
	\end{align*}
    and by Assumption~\ref{assm:arm opt}, it has $$ \max_{j\in\calA}\frac{\X^\top\Bbeta_j^*}{\|\X\|}\geq \max_{j\in\calA^c}\frac{\X^\top\Bbeta_j^*}{\|\X\|}+h.$$
The above two equation together imply $\X\in\calU_i(-2h_{\calA})$. Hence, we finish proving $\calX\l(i,\l\{\Bbeta_{j}\r\}\r)\subseteq \calU_i(-2h_{\calA})$. On the other hand, for any $\X\in\calU_i(2h_{\calA})$, we have
	\begin{align*}
		\frac{\X^\top\Bbeta_{i}}{\|\X\|}= \frac{\X^\top\Bbeta_i^*}{\|\X\|}+\l(\frac{\X^\top\Bbeta_i}{\|\X\|}-\frac{\X^\top\Bbeta_i^*}{\|\X\|}\r)\geq\max_{j\neq i,j\in\calA}\frac{\X^\top\Bbeta_{j}^*}{\|\X\|}+2h_{\calA}-h_{\calA}.
	\end{align*}
    Then, by triangle inequality $\max_{j\neq i,j\in\calA}\frac{\X^\top\Bbeta_{j}^*}{\|\X\|}\geq \max_{j\neq i,j\in\calA}\frac{\X^\top\Bbeta_{j}}{\|\X\|}-h_{\calA}$, we have $\frac{\X^\top\Bbeta_{i}}{\|\X\|}\geq \max_{j\neq i,j\in\calA}\frac{\X^\top\Bbeta_{j}}{\|\X\|}$. Moreover, under the event $\{\|\Bbeta_{i}-\Bbeta_{i}^*\|\leq \frac{h}{2},\ i\in[K]\}$, we have $\X\notin \calX\l(k,\l\{\Bbeta_{j}\r\}\r)$ for all $k\in\calA^c$. Hence, we finish proving $\X\in\calX\l(i,\l\{\Bbeta_{j}\r\}\r)$.
    
	We then only need to prove claim $(2)$ of Lemma~\ref{lem:regularity}. By triangle inequality, for all $\X\in\RR^d$, it has
	\begin{align*}
		\X^{\top}\Bbeta_{i}&\geq \X^{\top}\Bbeta_{i}^{*} - \left| \X^{\top}\Bbeta_{i} - \X^{\top}\Bbeta_{i}^{*}\right|\\
		%&= \max_{j\neq i} \X^{\top}\Bbeta_{j}^{*}- \left| \X^{\top}\Bbeta_{i} - \X^{\top}\Bbeta_{i}^{*}\right| + \left(\X^{\top}\Bbeta_{i}^{*} -  \max_{j\neq i} \X^{\top}\Bbeta_{j}^{*} \right)\\
		&\geq\max_{j\neq i} \X^{\top}\Bbeta_{j}- 2\max_{j}\left| \X^{\top}\Bbeta_{j} - \X^{\top}\Bbeta_{j}^{*}\right| + \left(\X^{\top}\Bbeta_{i}^{*} -  \max_{j\neq i} \X^{\top}\Bbeta_{j}^{*} \right).
	\end{align*}
	It further implies that
	\begin{multline*}
		\II\l\{\X^\top\Bbeta_{i}\geq \max_{j\neq i} \X^{\top}\Bbeta_{j} \r\}\\ 
		\geq \II\l\{2\max_{j}\left| \X^{\top}\Bbeta_{j} - \X^{\top}\Bbeta_{j}^{*}\right|\cdot \|\X\|^{-1} \leq \left(\X^{\top}\Bbeta_{i}^{*} -  \max_{j\neq i} \X^{\top}\Bbeta_{j}^{*} \right)\cdot\|\X\|^{-1} \r\}.
	\end{multline*}
	Thus, using the above equation, we have the following bound
	\begin{align*}
		&~~~~\EE\left\{\II{\{\X_t^{\top}\Bbeta_{i} \geq \max_j \X_t^{\top}\Bbeta_j\}}\cdot \X_t\X_t^{\top}\bigg| \calF_t \right\}\\
		&\succeq \EE\left\{\II{\left\{ 2\max_{j}\left| \X_t^{\top}\Bbeta_{j} - \X_t^{\top}\Bbeta_{j}^{*}\right|\cdot\|\X_t\|^{-1} \leq \left(\X_t^{\top}\Bbeta_{i}^{*} -  \max_{j\neq i} \X_t^{\top}\Bbeta_{j}^{*} \right)\cdot\|\X_t\|^{-1}\right\}}\cdot \X_t\X_t^{\top}\bigg| \calF_t \right\}.
	\end{align*}
	Recall that under Assumption~\ref{assm:arm opt}, $\X^{\top}\Bbeta_{i}^*-\max_{j\neq i}\X^{\top}\Bbeta_{j}^*\geq h\|\X\|$ holds for all $\X\in \calU_i(h)$. Thus, we restrict the right hand side to region $\calU_i(h)$ and then it arrives at
	\begin{multline*}
		\EE\left\{\II{\{\X_t^{\top}\Bbeta_{i} \geq \max_j \X_t^{\top}\Bbeta_j\}}\cdot \X_t\X_t^{\top}\bigg| \calF_t \right\}\\
		\succeq\EE\left\{\II{\left\{ 2\max_{j}\left| \X_t^{\top}\Bbeta_{j} - \X_t^{\top}\Bbeta_{j}^{*}\right| \cdot \|\X_t\|^{-1}\leq h\right\}}\cdot \X_t\X_t^{\top} \cdot \II\l\{ \X_t\in \calU_i(h)\r\}\bigg|\calF_t \right\}.
	\end{multline*}
	On the other hand, it's obvious that $\max_{j}| \X_t^{\top}\Bbeta_{j} - \X_t^{\top}\Bbeta_{j}^{*}| \cdot \|\X_t\|^{-1}\leq\max_{j}\|\Bbeta_{j}-\Bbeta_{j}^*\|$. Then we have
	\begin{align*}
		\EE\left\{\II{\{\X_t^{\top}\Bbeta_{i} \geq \max_j \X_t^{\top}\Bbeta_j\}}\cdot \X_t\X_t^{\top}\bigg| \calF_t \right\}%&\succeq \II\l\{\max_{j}\|\Bbeta_{j}^{(t)}-\Bbeta_j^*\|\leq \frac{h}{2}\r\}\EE\left\{\X_t\X_t^{\top} \cdot \II\l\{ \X_t\in U_i\r\}\bigg|\calF_t \right\}\\
		&\succeq \II\l\{\max_{j}\|\Bbeta_{j}-\Bbeta_j^*\|\leq\frac{h}{2}\r\}\cdot\EE\left\{\X_t\X_t^{\top} \cdot \II\l\{ \X_t\in \calU_i\r\}\big|\calF_t \right\},
	\end{align*}
	which proves the first statement of Lemma~\ref{lem:regularity}. It remains to discuss $\calA^c$.  Thus, we complete the proof.
\end{proof}

\subsection{Proof of Theorem~\ref{thm:LBD-estimation}}
\label{sec:proof:thm:LBD-estimation}
Recall $\calF_t=\sigma(Y_{t-1},\X_{t-1},\alpha_{t-1},\ldots)$. The following proposition justifies the conditional expectation of the estimation error.
\begin{proposition}
	Under the same conditions as Theorem~\ref{thm:LBD-estimation}, for $i\in\calA$, it has
	\begin{multline*}
		\EE\l\{\|\Bbeta_{i}^{(t+1)} - \Bbeta_{i}^*\|^2 \big| \calF_t\r\}\leq \|\Bbeta_{i}^{(t)}-\Bbeta_{i}^*\|^2-2\eta_{t}\frac{\pi_t}{K}\lambda_{\min}\|\Bbeta_{i}^{(t)}-\Bbeta_{i}^*\|^2+2\eta_{t}^2d\lambda_{\max}^2\|\Bbeta_{i}^{(t)}-\Bbeta_{i}^*\|^2\\-2\eta_{t}(1-\pi_t)\lambda_{\min}\|\Bbeta_{i}^{(t)}-\Bbeta_{i}^*\|^2\cdot\II\l\{\max_{j}\|\Bbeta_{j}^{(t)}-\Bbeta_{i}^*\|\leq\frac{h}{2}\r\}+2\eta_{t}^2d\lambda_{\max}\sigma^2,
	\end{multline*}
	and for $i\in\calA^c$, we have
	\begin{multline*}
		\EE\left\{\|\Bbeta_{i}^{(t+1)} - \Bbeta_{i}^*\|^2 \big| \calF_t\right\}
		\leq \|\Bbeta_{i}^{(t)}-\Bbeta_{i}^*\|^2-2\eta_{t}\frac{\pi_t}{K}\lambda_{\min}\|\Bbeta_{i}^{(t)}-\Bbeta_{i}^*\|^2\\
		+2\eta_{t}^2d\lambda_{\max}\sigma^2\cdot\l(\frac{\pi_t}{K}+(1-\pi_t) \cdot  \II\l\{ \max_{i\in[K]}\|\Bbeta_{i}^{(t)}-\Bbeta_{i}^*\|\geq \frac{h}{2}\r\} \r)\\
		+2\eta_{t}^2d\lambda_{\max}^2\|\Bbeta_{i}^{(t)}-\Bbeta_{i}^*\|^2\cdot\l(\frac{\pi_t}{K}+(1-\pi_t) \cdot  \II\l\{ \max_{i\in[K]}\|\Bbeta_{i}^{(t)}-\Bbeta_{i}^*\|\geq \frac{h}{2}\r\} \r),
	\end{multline*}
	where $C>0$ is some constant.
	\label{prop: conditional expectation}
\end{proposition}
\begin{proof}[Proof of Proposition~\ref{prop: conditional expectation}]
	Indeed, the update can be characterized as follows,
	\begin{align*}
		\Bbeta_{i}^{(t+1)} - \Bbeta_{i}^*= \Bbeta_{i}^{(t)}-\Bbeta_{i}^* - \eta_t\ \II{\{a_t=i\}}\cdot (\X_t^{\top}\Bbeta_{i}^{(t)} - Y_t)\ \X_t.
	\end{align*}
	Then the $L_2$ normed estimation error rate has,
	\begin{multline*}
		\|\Bbeta_{i}^{(t+1)} - \Bbeta_{i}^*\|^2= \|\Bbeta_{i}^{(t)}-\Bbeta_{i}^*\|^2 - 2\eta_t\ \II{\{a_t=i\}}\cdot (\X_t^{\top}\Bbeta_{i}^{(t)} - \X_t^\top \Bbeta_{i}^*)^2\\
		-2\eta_{t}\ \II{\{a_t=i\}}\ \xi_t \ (\X_t^{\top}\Bbeta_{i}^{(t)} - \X_t^\top \Bbeta_{i}^*)+\eta_{t}^2\ \II{\{a_t=i\}}\cdot (\X_t^{\top}\Bbeta_{i}^{(t)} - Y_t)^2\ \|\X_t\|^2.
	\end{multline*}
	Recall $\calF_t^+:=\sigma\left(\X_t,Y_{t-1},\X_{t-1},\alpha_{t-1},\dots\right)$ and we first consider the conditional expectation of the above equation. First of all, the indicator function has
	\begin{align*}
		\EE\left\{ \II{\{a_t=i\}}\big| \calF_t^+\right\} &=\EE\left\{ \II{\{a_t=i\}}\cdot \II\left\{\alpha_t=1\right\}\big| \calF_t^+\right\}+\EE\left\{ \II{\{a_t=i\}} \cdot \II\left\{\alpha_t=0\right\}\big| \calF_t^+\right\}\\
		&= \frac{\pi_t}{K}+(1-\pi_t)\cdot \II\l\{i\in\arg\max_j \X_t^{\top}\Bbeta_j^{(t)}\r\}.
	\end{align*}
	Secondly, the conditional expectation of $(\X_t^{\top}\Bbeta_{i}^{(t)} - Y_t)^2$ is given by
	\[\EE \left\{\l(\X_t^{\top}\Bbeta_{i}^{(t)} - Y_t\r)^2\bigg| \calF_t^+ \right\} \leq \l(\X_t^{\top}\Bbeta_{i}^{(t)} - \X_t^{\top}\Bbeta_{i}^*\r)^2+\sigma^2.\]
	Thus, altogether, we have
	\begin{multline*}
		\EE\left\{\left\|\Bbeta_{i}^{(t+1)} - \Bbeta_{i}^*\right\|^2 \bigg| \calF_t^+\right\}\\
		= \left\|\Bbeta_{i}^{(t)}-\Bbeta_{i}^*\right\|^2 - 2\eta_t\ \left(\frac{\pi_t}{K} + (1-\pi_t)\cdot \II\l\{i \in \arg\max_j \X_t^{\top}\Bbeta_j^{(t)}\r\}\right)\cdot \l(\X_t^{\top}\Bbeta_{i}^{(t)} - \X_t^\top \Bbeta_{i}^*\r)^2\\
		+\eta_{t}^2\ \left(\frac{\pi_t}{K}+(1-\pi_t)\cdot \II\l\{i= \arg\max_j \X_t^{\top}\Bbeta_j^{(t)}\r\}\right)\cdot \left((\X_t^{\top}\Bbeta_{i}^{(t)} - \X_t^{\top}\Bbeta_{i}^*)^2+\sigma^2\right)\ \left\|\X_t\right\|^2.
	\end{multline*}
	Then, the expectation conditional on $\calF_t$ has,
	\begin{multline*}
		\EE\left\{\left\|\Bbeta_{i}^{(t+1)} - \Bbeta_{i}^*\right\|^2 \bigg| \calF_t\right\}
		\leq \|\Bbeta_{i}^{(t)}-\Bbeta_{i}^*\|^2 - 2\eta_t\frac{\pi_t}{K}\lambda_{\min} \| \Bbeta_{i}^{(t)} - \Bbeta_{i}^*\|^2 \\
		-2\eta_{t} (1-\pi_t)\cdot \underbrace{\EE\left\{ \II{\{i = \arg\max_j \X_t^{\top}\Bbeta_j^{(t)}\}}\cdot (\X_t^{\top}\Bbeta_{i}^{(t)} - \X_t^\top \Bbeta_{i}^*)^2\big| \calF_t \right\}}_{\A_1}\\
		+\eta_{t}^2\cdot\frac{\pi_t}{K} \ \sigma^2\cdot \EE\left\|\X_t\right\|^2+\eta_{t}^2\cdot \frac{\pi_t}{K}\cdot \underbrace{\EE\left\{ (\X_t^{\top}\Bbeta_{i}^{(t)} - \X_t^\top \Bbeta_{i}^*)^2\cdot\left\|\X_t\right\|^2\big| \calF_t \right\}}_{\A_2}\\
		+\eta_{t}^2\cdot(1-\pi_t)\cdot\underbrace{\EE\left\{ \II{\{i=\arg\max_j \X_t^{\top}\Bbeta_j^{(t)}\}}\cdot (\X_t^{\top}\Bbeta_{i}^{(t)} - \X_t^\top \Bbeta_{i}^*)^2\cdot \|\X_t\|^2\big| \calF_t \right\}}_{\A_3}\\
		+\eta_{t}^2\cdot(1-\pi_t)\cdot\sigma^2\cdot\underbrace{\EE\left\{ \II{\{i=\arg \max_j \X_t^{\top}\Bbeta_j^{(t)}\}}\cdot \|\X_t\|^2\big| \calF_t \right\}}_{\A_4}.
	\end{multline*}
	We shall discuss the bound of $\A_1,\A_2,\A_3,\A_4$ separately depending on whether $i\in\calA$ or $i\in\calA^c$.
	\begin{enumerate}[1]
		\item For $i\in\calA$.
		
		\noindent\textbf{Bound of $\A_1$}
		Lemma~\ref{lem:regularity} gives
		$$\A_1\geq \lambda_{\min} \|\Bbeta_{i}^{(t)} - \Bbeta_{i}\|^2\cdot \II\l\{\max_{i}\|\Bbeta_{i}^{(t)}-\Bbeta_{i}^*\|\leq\frac{h}{2}\r\}. $$
		\noindent\textbf{Bound of $\A_2$} To bound $\A_2$, we know that $\Bbeta_t\in\calF_t$. Then under Assumption~\ref{assm:cov regression}, we have
		\begin{align*}
			\A_2\leq \| \Bbeta_{i}^{(t)}-\Bbeta_{i}^*\|^2\cdot \l\|\EE\left\{ \X_t\X_t^\top\cdot\left\|\X_t\right\|^2 \right\}\big|\calF_t\r\|\leq 2d\lambda_{\max}^2 \| \Bbeta_{i}^{(t)}-\Bbeta_{i}^*\|^2,
		\end{align*}
		where $C$ is some constant.
		
		\noindent\textbf{Bound of $\A_3$} Similar to bound of $\A_2$, we have
		\begin{align*}
			\A_3\leq 2d\lambda_{\max}^2\|\Bbeta_{i}^{(t)}-\Bbeta^*\|^2.
		\end{align*}
		\noindent\textbf{Bound of $\A_4$} In a similar fashion, we have
		\begin{align*}
			\A_4\leq 2d\lambda_{\max}.
		\end{align*}
		Thus, in all, the conditional expected estimation error has,
		\begin{multline*}
			\EE\left\{\|\Bbeta_{i}^{(t+1)} - \Bbeta_{i}^*\|^2 \big| \calF_t\right\}\leq \|\Bbeta_{i}^{(t)}-\Bbeta_{i}^*\|^2-2\eta_{t}\frac{\pi_t}{K}\lambda_{\min}\|\Bbeta_{i}^{(t)}-\Bbeta_{i}^*\|^2+2\eta_{t}^2d\lambda_{\max}\sigma^2\\
			-2\eta_{t}(1-\pi_t)\lambda_{\min}\|\Bbeta_{i}^{(t)}-\Bbeta_{i}^*\|^2\cdot\II\l\{\max_{i}\|\Bbeta_{i}^{(t)}-\Bbeta_{i}^*\|\leq\frac{h}{2}\r\}+2\eta_{t}^2d\lambda_{\max}^2\|\Bbeta_{i}^{(t)}-\Bbeta_{i}^*\|^2.
		\end{multline*}
		\item For $i\in\calA^c$.
		
		\noindent\textbf{Bound of $\A_1$} We use $\A_1\geq 0$.
		
		\noindent\textbf{Bound of $\A_2$} Its bound is same as the case when $i\in\calA$.
		
		\noindent\textbf{Bound of $\A_3$} Lemma~\ref{lem:regularity} shows that when $\max_{i\in[K]}\|\Bbeta_{i}^{(t)}-\Bbeta_{i}^*\|\leq h/2$, $$\arg\max_{i\in[K]} \X^{\top}\Bbeta_{i}^{(t)} \subseteq \calA.$$ Thus, we have $\A_3\leq Cd\lambda_{\max}^2\|\Bbeta_{i}^{(t)}-\Bbeta^*\|^2\cdot\II\l\{ \max_{i\in[K]}\|\Bbeta_{i}^{(t)}-\Bbeta_{i}^*\|\geq h/2\r\}$.
		
		\noindent\textbf{Bound of $\A_4$} In a similar fashion, we have
		\begin{align*}
			\A_4\leq 2d\lambda_{\max}\cdot\II\l\{ \max_{i\in[K]}\|\Bbeta_{i}^{(t)}-\Bbeta_{i}^*\|\geq h/2\r\}.
		\end{align*}
		Thus, in all, the estimation error rate has
		\begin{multline*}
			\EE\l\{\|\Bbeta_{i}^{(t+1)} - \Bbeta_{i}^*\|^2 \big| \calF_t\r\}
			\leq \|\Bbeta_{i}^{(t)}-\Bbeta_{i}^*\|^2-2\eta_{t}\frac{\pi_t}{K}\lambda_{\min}\|\Bbeta_{i}^{(t)}-\Bbeta_{i}^*\|^2\\
			+\eta_{t}^2d\lambda_{\max}\sigma^2\cdot\l(\frac{\pi_t}{K}+(1-\pi_t) \cdot  \II\l\{ \max_{i\in[K]}\|\Bbeta_{i}^{(t)}-\Bbeta_{i}^*\|\geq \frac{h}{2}\r\} \r)\\
			+2\eta_{t}^2d\lambda_{\max}^2\|\Bbeta_{i}^{(t)}-\Bbeta_{i}^*\|^2\cdot\l(\frac{\pi_t}{K}+(1-\pi_t) \cdot  \II\l\{ \max_{i\in[K]}\|\Bbeta_{i}^{(t)}-\Bbeta_{i}^*\|\geq \frac{h}{2}\r\} \r),
		\end{multline*}
		which completes the proof.
	\end{enumerate}
	
\end{proof}
Then we are ready to prove Theorem~\ref{thm:LBD-estimation}.
\begin{proof}[Proof of Theorem~\ref{thm:LBD-estimation}]
	Before discussing the convergence, we introduce
	\begin{align*}
		&\|\Bbeta_{i}^{(t+1)} - \Bbeta_{i}^*\|^2  - \EE\l\{\|\Bbeta_{i}^{(t+1)} - \Bbeta_{i}^*\|^2 \big| \calF_t\r\}\\
		&~~~~~=-2\eta_{t}\underbrace{\l\{ \II{\{a_t=i\}}\cdot (\X_t^{\top}\Bbeta_{i}^{(t)} - \X_t^\top \Bbeta_{i}^*)^2 -\EE\l\{ \II{\{a_t=i\}}\cdot (\X_t^{\top}\Bbeta_{i}^{(t)} - \X_t^\top \Bbeta_{i}^*)^2\big|\calF_t\r\} \r\}}_{=:\boldsymbol{\Xi}_1^{(t)}} \\
		&~~~~~~~-2\eta_{t}\underbrace{\l\{\II{\{a_t=i\}}\ \xi_t \ (\X_t^{\top}\Bbeta_{i}^{(t)} - \X_t^\top \Bbeta_{i}^*)- \EE\l\{\II{\{a_t=i\}}\ \xi_t \ (\X_t^{\top}\Bbeta_{i}^{(t)} - \X_t^\top \Bbeta_{i}^*) \big|\calF_t\r\}\r\}}_{=:\boldsymbol{\Xi}_2^{(t)}}\\
		&~~~~~~~+\eta_{t}^2\underbrace{\l\{\II{\{a_t=i\}}\cdot (\X_t^{\top}\Bbeta_{i}^{(t)} - Y_t)^2\ \|\X_t\|^2-\EE\l\{\II{\{a_t=i\}}\cdot (\X_t^{\top}\Bbeta_{i}^{(t)} - Y_t)^2\ \|\X_t\|^2\r\} \r\}}_{=:\boldsymbol{\Xi}_3^{(t)}}.
	\end{align*}
	Conditional on $\calF_t$, it has
	\begin{align*}
		\|\boldsymbol{\Xi}_1^{(t)}|\calF_t\|_{\Psi_1}\leq \lambda_{\max}\|\Bbeta_{i}^{(t)}-\Bbeta_{i}^*\|^2,\quad	\|\boldsymbol{\Xi}_2^{(t)}|\calF_t \|_{\Psi_1}\leq \sigma\sqrt{\lambda_{\max}}\|\Bbeta_{i}^{(t)}-\Bbeta_{i}^*\|,
	\end{align*}
	while under $\calE_t^X:=\{\|\X_t\|^2\leq C\lambda_{\max}(d+\delta_t)\}$,
	\begin{align*}
		\|\boldsymbol{\Xi}_3^{(t)}|\calF_t,\ \calE_t \|_{\Psi_1}\leq C(d+\delta_t)\lambda_{\max}\sigma^2+(d+\delta_t)\lambda_{\max}^2\|\Bbeta_{i}^{(t)}-\Bbeta_{i}^*\|^2.
	\end{align*}
	Specifically, $\PP(\calE_t^X|\calF_t)\geq 1-\exp(-c_2(d+\delta_t))$.
	
	~\\\noindent\textbf{THE FIRST PHASE} Define event $$\calE_t:=\l\{\|\Bbeta_{i}^{(t)}-\Bbeta_{i}^*\|^2\leq \frac{C^*\; (d+\delta_t)}{t+\Cb d}\frac{\sigma^2}{\lambda_{\min}},\quad i\in[K]\r\},$$ where $C^*$ is some constant uncorrelated with $t,d,n$ and will be specified later. We are going to prove by induction. It is obvious that $\calE_0$ holds. Then we shall show event $\calE_{t+1}$ holds under $\cup_{s=0}^t\calE_s$. Substitue the exploration rate $\pi_t\geq \pi$ into Proposition~\ref{prop: conditional expectation} and then it arrives at
	$$\EE\left\{\|\Bbeta_{i}^{(t+1)} - \Bbeta_{i}^*\|^2 \big| \calF_t\right\}\leq \|\Bbeta_{i}^{(t)}-\Bbeta_{i}^*\|^2-2\eta_{t}\frac{\pi}{K}\lambda_{\min}\|\Bbeta_{i}^{(t)}-\Bbeta_{i}^*\|^2
	+C\eta_{t}^2d\lambda_{\max}\sigma^2+C\eta_{t}^2d\lambda_{\max}^2 \|\Bbeta_{i}^{(t)}-\Bbeta_{i}^*\|^2.$$
	%Note that under event $\calE_t$, $\|\Bbeta_{i}^{(t)}-\Bbeta_{i}^*\|^2\leq C^*\sigma^2/\Cb \lambda_{\min}\leq\sigma^2/\lambda_{\max}$. 
	Then the estimation error dynamics have
	\begin{align*}
		\|\Bbeta_{i}^{(t+1)} - \Bbeta_{i}^*\|^2 &= \EE\l\{\|\Bbeta_{i}^{(t+1)} - \Bbeta_{i}^*\|^2 \big| \calF_t\r\} + \|\Bbeta_{i}^{(t+1)} - \Bbeta_{i}^*\|^2  - \EE\l\{\|\Bbeta_{i}^{(t+1)} - \Bbeta_{i}^*\|^2 \big| \calF_t\r\}\\
		&\leq \|\Bbeta_{i}^{(t)}-\Bbeta_{i}^*\|^2-2\eta_{t}\frac{\pi}{K}\lambda_{\min}\|\Bbeta_{i}^{(t)}-\Bbeta_{i}^*\|^2
		+2\eta_{t}^2d\lambda_{\max}\sigma^2\\
		&~~~~~~~~~~~~~~~~~~~~~~+2\eta_{t}^2d\lambda_{\max}^2\|\Bbeta_{i}^{(t)}-\Bbeta_{i}\|^2+2\eta_{t}\boldsymbol{\Xi}_1^{(t)}+2\eta_{t}\boldsymbol{\Xi}_2^{(t)}+\eta_{t}^2\boldsymbol{\Xi}_3^{(t)}.
		%&~~~~~~~~~~~~~~~~~~~~~~~~~~~~~~~~~~~~~~~~+\|\Bbeta_{i}^{(t+1)} - \Bbeta_{i}^*\|^2  - \EE\l\{\|\Bbeta_{i}^{(t+1)} - \Bbeta_{i}^*\|^2 \big| \calF_t\r\}.
	\end{align*}
	Substitute the stepsize $\eta_{t}=\frac{1}{\lambda_{\min}}\frac{\Ca}{t+\Cb d}$ into the above equation,
	\begin{multline*}
		\|\Bbeta_{i}^{(t+1)} - \Bbeta_{i}^*\|^2
		%&\leq \|\Bbeta_{i}^{(t)}-\Bbeta_{i}^*\|^2-2\eta_{t}\frac{\pi}{K}\lambda_{\min}\|\Bbeta_{i}^{(t)}-\Bbeta_{i}^*\|^2+2\eta_{t}^2d\lambda_{\max}\sigma^2+2\eta_{t}\boldsymbol{\Xi}_1^{(t)}+2\eta_{t}\boldsymbol{\Xi}_2^{(t)}+\eta_{t}^2\boldsymbol{\Xi}_3^{(t)}\\
		\leq \l(1-\frac{\pi}{K}\frac{\Ca}{t+\Cb d}\r)\|\Bbeta_{i}^{(t)}-\Bbeta_i^*\|^2+\frac{2\,\Ca^2}{(t+\Cb d)^2}\,\frac{\lambda_{\max}}{\lambda_{\min}^2}\,d\,\sigma^2\\
		- \frac{2\,\Ca}{\lambda_{\min}}\frac{1}{t+\Cb d}\,\boldsymbol{\Xi}_1^{(t)} - \frac{2\, \Ca}{\lambda_{\min}}\frac{1}{t+\Cb d}\,\boldsymbol{\Xi}_2^{(t)}+\frac{\Ca^2}{\lambda_{\min}^2}\frac{1}{(t+\Cb d)^2}\,\boldsymbol{\Xi}_3^{(t)},
	\end{multline*}
	where $2\eta_td \leq\lambda_{\min}/\lambda_{\max}^2$ is used. We accumulate the upper bound until $t=0$,
	\begin{align*}
		\|\Bbeta_{i}^{(t+1)} - \Bbeta_{i}^*\|^2&\leq\underbrace{\prod_{s=0}^{t}  \l(1-\frac{\pi}{K}\frac{\Ca}{s+\Cb d}\r)\|\Bbeta_{i}^{(0)}-\Bbeta_i^*\|^2}_{\boldsymbol{\Delta}_1}+\underbrace{2\frac{\lambda_{\max}}{\lambda_{\min}^2}\sigma^2\sum_{s=0}^{t} \prod_{l=s+1}^{t}  \l(1-\frac{\pi}{K}\frac{\Ca}{l+\Cb d}\r)\frac{\Ca^2 d}{(s+\Cb d)^2}}_{\boldsymbol{\Delta}_2}\\
		%&~~~~~~~~~~+ \underbrace{C\frac{\lambda_{\max}^2}{\lambda_{\min}^2}\sum_{s=0}^{t} \prod_{l=s+1}^{t}  \l(1-\frac{\pi}{K}\frac{\Ca}{l+\Cb d}\r)\frac{\Ca^2 d}{(s+\Cb d)^2}\|\Bbeta_{i}^{(s)}-\Bbeta_{i}^*\|^2}_{\boldsymbol{\Delta}_3}\\
		&~~~~~~~~~~-\underbrace{\frac{2\,\Ca}{\lambda_{\min}}\sum_{s=0}^{t}\prod_{l=s+1}^{t}  \l(1-\frac{\pi}{K}\frac{\Ca}{l+\Cb d}\r)\frac{1}{s+\Cb d}\boldsymbol{\Xi}_1^{(s)}}_{\boldsymbol{\Delta}_3}\\
		&~~~~~~~~~~-\underbrace{\frac{2\, \Ca}{\lambda_{\min}}\sum_{s=0}^{t}\prod_{l=s+1}^{t}  \l(1-\frac{\pi}{K}\frac{\Ca}{l+\Cb d}\r)\frac{1}{s+\Cb d}\boldsymbol{\Xi}_2^{(s)}}_{\boldsymbol{\Delta}_4}\\
		&~~~~~~~~~~-\underbrace{\frac{\Ca^2}{\lambda_{\min}^2}\sum_{s=0}^{t}\prod_{l=s+1}^{t}  \l(1-\frac{\pi}{K}\frac{\Ca}{l+\Cb d}\r)\l(\frac{1}{s+\Cb d}\r)^2\boldsymbol{\Xi}_3^{(s)}}_{\boldsymbol{\Delta}_5}.
	\end{align*}
	Recall that we require $\Ca\geq 2K/\pi$. Then, by Lemma~\ref{teclem:prod upper bound}, we have
	\begin{align*}
		\prod_{s=0}^{t}  \l(1-\frac{\pi}{K}\frac{\Ca}{s+\Cb d}\r)\leq \l(\frac{\Cb\, d}{t+1+\Cb d}\r)^{\frac{\Ca \pi}{K}}\leq \frac{\Cb\, d}{t+1+\Cb d},
	\end{align*}
	and 
	\begin{align*}
		\l(1-\frac{\pi}{K}\frac{\Ca}{s+1+\Cb d}\r)\cdots\l(1-\frac{\pi}{K}\frac{\Ca}{t+\Cb d}\r)\l(\frac{1}{s+\Cb d}\r)^2\leq \frac{\l(s+\Cb d\r)^{\pi\cdot\Ca/K-2}}{\l(t+1+\Cb d\r)^{\pi\cdot\Ca/K}}.
	\end{align*}
	Then, sum up the above equation over $s$ and follow Lemma~\ref{teclem:sum upper bound} calculations, we have,
	\begin{align*}
		\sum_{s=0}^{t} \l(1-\frac{\pi}{K}\frac{\Ca}{s+1+\Cb d}\r)\cdots\l(1-\frac{\pi}{K}\frac{\Ca}{t+\Cb d}\r)\l(\frac{1}{s+\Cb d}\r)^2\leq \frac{2K}{\pi\cdot \Ca}\frac{1}{t+1+\Cb d}.
	\end{align*}
	Thus together with the initialization, we have the bound for $\boldsymbol{\Delta}_1$ and $\boldsymbol{\Delta}_2$,
	$$\boldsymbol{\Delta}_1\leq \frac{\Cb\, d}{t+1+\Cb d}\frac{\sigma^2}{\lambda_{\max}},\quad \boldsymbol{\Delta}_2\leq 4\Ca\frac{K}{\pi}\frac{\lambda_{\max}}{\lambda_{\min}^2}\frac{d\, \sigma^2}{t+1+\Cb d}.$$
	Then consider $\boldsymbol{\Delta}_3$. The conditional Orlicz norm can be bounded with
	\begin{align*}
		\l\|\prod_{l=s+1}^{t}  \l(1-\frac{\pi}{K}\frac{\Ca}{l+\Cb d}\r)\frac{1}{s+\Cb d}\boldsymbol{\Xi}_1^{(s)}\bigg|\calF_s\r\|_{\Psi_1}\leq C\lambda_{\max}\frac{\l(s+\Cb d\r)^{\Ca\cdot\pi/K-1}}{\l(t+1+\Cb d\r)^{\Ca\cdot\pi/K}}\cdot\|\Bbeta_{i}^{(s)}-\Bbeta_{i}^*\|^2,
	\end{align*}
	and under event $\cup_{s=0}^t\calE_s$, Lemma~\ref{teclem:azuma} leads to
	\begin{multline*}
		\PP\l(\l|\sum_{s=0}^{t}\prod_{l=s+1}^{t}  \l(1-\frac{\pi}{K}\frac{\Ca}{l+\Cb d}\r)\frac{1}{s+\Cb d}\boldsymbol{\Xi}_1^{(s)}\r|\geq c_1\frac{C^*}{\Ca}\frac{(d+\delta_{t+1})\ \sigma^2}{t+1+\Cb d}\r)\\
		\leq\exp\l(-c_2  \frac{\pi}{K}\frac{\Cb d+t}{\Ca}\frac{\lambda_{\min}^2}{\lambda_{\max}^2}\r),
	\end{multline*}
	where $c_1<0.1$ is some constant and under which we have
	\begin{align*}
		\big|\boldsymbol{\Delta}_3\big|\leq c_1C^*\frac{1}{\lambda_{\min}}\frac{(d+\delta_{t+1})\; \sigma^2}{t+1+\Cb d}.
	\end{align*}
	As for the term $\boldsymbol{\Delta}_4$, we have
	\begin{align*}
		\l\|\prod_{l=s+1}^{t}  \l(1-\frac{\pi}{K}\frac{\Ca}{l+\Cb d}\r)\frac{1}{s+\Cb d}\boldsymbol{\Xi}_2^{(s)}\bigg|\calF_s \r\|_{\Psi_1}\leq C\sqrt{\lambda_{\max}} \sigma\frac{(s+\Cb d)^{\Ca\cdot\pi/K-1}}{(t+1+\Cb d)^{\Ca\cdot\pi/K}}\cdot\|\Bbeta_{i}^{(s)}-\Bbeta_i^*\|,
	\end{align*}
	and then under event $\cup_{s=0}^t\calE_s$, Lemma~\ref{teclem:azuma} guarantees
	\begin{multline*}
		\PP\l(\l|\sum_{s=0}^{t}\prod_{l=s+1}^{t}  \l(1-\frac{\pi}{K}\frac{\Ca}{l+\Cb d}\r)\frac{1}{s+\Cb d}\boldsymbol{\Xi}_2^{(s)}\r|\geq c_1\frac{C^*}{\Ca}\frac{(d+\delta_{t+1})\ \sigma^2}{t+1+\Cb d}\r)\\
		\leq\exp\l(-c_2 \min\l\{(d+\delta_{t+1})\frac{C^*}{\Ca}\frac{\pi}{ K}\frac{\lambda_{\min}}{\lambda_{\max}},\sqrt{d+\delta_{t+1}}\sqrt{\Cb d+t}\frac{\sqrt{C^*}}{\Ca}\frac{\sqrt{\lambda_{\min}}}{\sqrt{\lambda_{\max}}}\r\}\r).
	\end{multline*}
	Hence, we have the bound for $\boldsymbol{\Delta}_4$,
	\begin{align*}
		\big| \boldsymbol{\Delta}_4\big|\leq c_1\frac{C^*}{\lambda_{\min}}\frac{(d+\delta_{t+1})\; \sigma^2}{t+1+\Cb d}.
	\end{align*}
	Finally, under $\calE_s^X$ and $\|\Bbeta_{t}-\Bbeta^*\|\leq 3\sigma^2/\lambda_{\max}$, the $\boldsymbol{\Delta}_5$ term has
	\begin{multline*}
		\l\|\prod_{l=s+1}^{t}  \l(1-\frac{\pi}{K}\frac{\Ca}{l+\Cb d}\r)\l(\frac{1}{s+\Cb d}\r)^2\boldsymbol{\Xi}_3^{(s)}\bigg|\calF_s,\calE_s^X\r\|_{\Psi_1}\\
		\leq C\lambda_{\max}(d+\delta_s)\sigma^2\frac{\l(s+\Cb d\r)^{\Ca\cdot\pi/K-2}}{\l(t+1+\Cb d\r)^{\Ca\cdot\pi/K}}+C(d+\delta_s)^2\lambda_{\max}^2\frac{(s+\Cb s)^{\Ca\cdot\pi/K-3}}{(t+1+\Cb d)^{\Ca\cdot\pi/K}}.
	\end{multline*}
	It implies that under event $\cup\calE_s^X$, we have
	\begin{multline*}
		\PP\l(\l|\sum_{s=0}^{t}\prod_{l=s+1}^{t}  \l(1-\frac{\pi}{K}\frac{\Ca}{l+\Cb d}\r)\l(\frac{1}{s+\Cb d}\r)^2\boldsymbol{\Xi}_3^{(s)}\r|\geq c_1C^*\frac{\lambda_{\min}}{\Ca^2}\frac{(d+\delta_{t+1})\ \sigma^2}{t+1+\Cb d}\r)\\ \leq\exp\l(-c_2\l(d+\frac{t}{\Cb}\r)\min\l\{ \frac{\lambda_{\min}^2}{\lambda_{\max}^2}\frac{\pi}{K}\frac{\Cb}{\Ca}\l(\frac{C^*}{\Ca}\r)^2,\frac{\lambda_{\min}}{\lambda_{\max}}\frac{C^*\Cb}{\Ca^2}\r\}\r),
	\end{multline*}
	under which we have
	\begin{align*}
		\big|\boldsymbol{\Delta}_5\big|\leq c_1\frac{C^*}{\lambda_{\min}}\frac{(d+\delta_{t+1})\; \sigma^2}{t+1+\Cb d}.
	\end{align*}
	Thus, under the aboved mentioned events, we have
	\begin{align*}
		\|\Bbeta_{i}^{(t+1)} - \Bbeta_{i}^*\|^2&\leq \frac{\Cb \; d}{t+1+\Cb d}\frac{\sigma^2}{\lambda_{\max}}+4\Ca\frac{K}{\pi}\frac{\lambda_{\max}}{\lambda_{\min}^2}\frac{{\rm d} \sigma^2}{t+1+\Cb d}+0.3\frac{C^*}{\lambda_{\min}}\frac{(d+\delta_{t+1})\; \sigma^2}{t+1+\Cb d}.
	\end{align*}
	With $C^*=8\Ca \frac{K}{\pi}\frac{\lambda_{\max}}{\lambda_{\min}}+2\Cb\frac{\lambda_{\min}}{\lambda_{\max}}$ and $\Cb\geq 3\Ca\frac{K}{\pi}\frac{\lambda_{\max}^2}{\lambda_{\min}^2}$, we have $$\|\Bbeta_{i}^{(t+1)} - \Bbeta_{i}^*\|^2\leq \frac{C^* \; (d+\delta_{t+1})}{t+1+\Cb d}\frac{\sigma^2}{\lambda_{\min}},\quad \text{for all } i\in[K],$$with probability excceeding $1-2K\exp\l(-cd-c\frac{t}{\Ca}\r)-K\exp\l(-c(d+\delta_{t+1})\r)$. %Furthermore, with $\Cb\geq 3\Ca\frac{K}{\pi}\frac{\lambda_{\max}^2}{\lambda_{\min}^2}$ and $\delta_t\leq Ct/\Cb$, we have $\|\Bbeta_{i}^{(t+1)} - \Bbeta_{i}^*\|^2\leq C\sigma^2/\lambda_{\max}$.
	~\\
	~\\
	\noindent\textbf{THE SECOND PHASE} We will first consider the convergence dynamics of the set $\calA$ and then discuss $i\in\calA^c$. In this phase, $\pi_t\in[0,1]$ is arbitrary. We denote the event
	\begin{align*}
		\calE_t:=\l\{\|\Bbeta_{i}^{(t)}-\Bbeta_{i}^*\|^2\leq\frac{C^{**}\, (d+\delta_t)}{t+\Cb d}\frac{\sigma^2}{\lambda_{\min}},\; i\in\calA,\quad \|\Bbeta_{i}^{(t)}-\Bbeta_{i}^*\|\leq \frac{h}{2},\; i\in\calA^c\r\},
	\end{align*} 
	which is different from the event in the first phase. We prove by inducton. It is worth noting that $\calE_{t_1}$ holds. Then we shall prove under $\cup_{s=t_1}^{t}\calE_s$, $\calE_{t+1}$ holds. Then given $\{\max_{i\in[K]}\|\Bbeta_{i}^{(t)}-\Bbeta_{i}^*\|\leq\frac{h}{2}\}$, Proposition~\ref{prop: conditional expectation} proves, for all $i\in\calA$,
	\begin{align*}
		\EE\l\{\|\Bbeta_{i}^{(t+1)} - \Bbeta_{i}^*\|^2 \big| \calF_t\r\}\leq \|\Bbeta_{i}^{(t)}-\Bbeta_{i}^*\|^2-\eta_{t}\lambda_{\min}\frac{1}{K}\|\Bbeta_{i}^{(t)}-\Bbeta_{i}^*\|^2+C\eta_{t}^2d\lambda_{\max}\sigma^2,
	\end{align*}
	where $\frac{\pi_t}{K}+(1-\pi_t)\geq \frac{1}{K}$. Thus, the estimation error rates can be bounded with
	\begin{align*}
		\|\Bbeta_{i}^{(t+1)} - \Bbeta_{i}^*\|^2&=\EE\l\{\|\Bbeta_{i}^{(t+1)} - \Bbeta_{i}^*\|^2 \big| \calF_t\r\}+\l\{\|\Bbeta_{i}^{(t+1)} - \Bbeta_{i}^*\|^2 -\EE\l\{\|\Bbeta_{i}^{(t+1)} - \Bbeta_{i}^*\|^2 \big| \calF_t\r\}\r\}\\
		&\leq \|\Bbeta_{i}^{(t)}-\Bbeta_{i}^*\|^2-\eta_{t}\lambda_{\min}\frac{1}{K}\|\Bbeta_{i}^{(t)}-\Bbeta_{i}^*\|^2+C\eta_{t}^2d\lambda_{\max}\sigma^2 - 2\eta_{t}\boldsymbol{\Xi}_1^{(t)}-2\eta_{t}\boldsymbol{\Xi}_2^{(t)}+\eta_{t}^2\boldsymbol{\Xi}_3^{(t)}.
	\end{align*}
	Insert the stepsize $\eta_{t}=\frac{\Ca}{\lambda_{\min}}\ \frac{1}{t+\Cb d}$ into the above equation and then we have
	\begin{align*}
		\|\Bbeta_{i}^{(t+1)} - \Bbeta_{i}^*\|^2
		&\leq \l(1-\frac{\Ca/K}{t+\Cb d}\r)\|\Bbeta_{i}^{(t)}-\Bbeta_i^*\|^2+\frac{C\,\Ca^2}{(t+\Cb d)^2}\,\frac{\lambda_{\max}}{\lambda_{\min}^2}\,d\,\sigma^2\\
		&~~~~~~~~~~~~~~~~~~~~~~- \frac{2\,\Ca}{\lambda_{\min}}\frac{1}{t+\Cb d}\,\boldsymbol{\Xi}_1^{(t)} - \frac{2\, \Ca}{\lambda_{\min}}\frac{1}{t+\Cb d}\,\boldsymbol{\Xi}_2^{(t)}+\frac{\Ca^2}{\lambda_{\min}^2}\l(\frac{1}{t+\Cb d}\r)^2\,\boldsymbol{\Xi}_3^{(t)}.
	\end{align*}
	It leads to
	\begin{align*}
		\|\Bbeta_{i}^{(t+1)} - \Bbeta_{i}^*\|^2&\leq\underbrace{\prod_{s=t_1}^{t}  \l(1-\frac{\Ca/K}{s+\Cb d}\r)\|\Bbeta_{i}^{(t_1)}-\Bbeta_i^*\|^2}_{\boldsymbol{\Theta}_1}+\underbrace{C\frac{\lambda_{\max}}{\lambda_{\min}^2}\sigma^2\sum_{s=t_1}^{t} \prod_{l=s+1}^{t}  \l(1-\frac{\Ca/K}{l+\Cb d}\r)\frac{\Ca^2 d}{(s+\Cb d)^2}}_{\boldsymbol{\Theta}_2}\\
		&~~~~~~~~~~-\underbrace{\frac{2\,\Ca}{\lambda_{\min}}\sum_{s=t_1}^{t}\prod_{l=s+1}^{t}  \l(1-\frac{\Ca/K}{l+\Cb d}\r)\frac{1}{s+\Cb d}\boldsymbol{\Xi}_1^{(s)}}_{\boldsymbol{\Theta}_3}\\
		&~~~~~~~~~~-\underbrace{\frac{2\, \Ca}{\lambda_{\min}}\sum_{s=t_1}^{t}\prod_{l=s+1}^{t}  \l(1-\frac{\Ca/K}{l+\Cb d}\r)\frac{1}{s+\Cb d}\boldsymbol{\Xi}_2^{(s)}}_{\boldsymbol{\Theta}_4} \\
		&~~~~~~~~~~-\underbrace{\frac{\Ca^2}{\lambda_{\min}^2}\sum_{s=t_1}^{t}\prod_{l=s+1}^{t}  \l(1-\frac{\Ca/K}{l+\Cb d}\r)\l(\frac{1}{s+\Cb d}\r)^2\boldsymbol{\Xi}_3^{(s)}}_{\boldsymbol{\Theta}_5}.
	\end{align*}
	By Lemma~\ref{teclem:prod upper bound} and $\Ca>K$, we have
	\begin{align*}
		\prod_{s=t_1}^{t}  \l(1-\frac{\Ca/K}{s+\Cb d}\r)\leq \frac{t_1+\Cb d}{t+1+\Cb d},
	\end{align*}
	which proves $\boldsymbol{\Theta}_1 \leq \frac{t_1+\Cb d}{t+1+\Cb d}\|\Bbeta_{i}^{(t_1)}-\Bbeta_{i}^*\|^2\leq \frac{Cd}{t+1+\Cb d}\frac{\sigma^2}{\lambda_{\min}}$. Moreover, Lemma~\ref{teclem:prod upper bound} proves
	\begin{align*}
		\l(1-\frac{\Ca/K}{s+1+\Cb d}\r)\cdots\l(1-\frac{\Ca/K}{t+\Cb d}\r)\l(\frac{1}{s+\Cb d}\r)^2\leq \l(\frac{s+\Cb d}{t+1+\Cb d}\r)^{\Ca/K}\l(\frac{1}{s+\Cb d}\r)^2.
	\end{align*}
	Then, sum over $s$ and Lemma~\ref{teclem:sum upper bound} shows
	\begin{align*}
		\sum_{s=t_1}^{t} \l(1-\frac{\Ca/K}{s+1+\Cb d}\r)\cdots\l(1-\frac{\Ca/K}{t+\Cb d}\r)\l(\frac{1}{s+\Cb d}\r)^2\leq \frac{2K}{\Ca}\frac{1}{t+1+\Cb d},
	\end{align*}
	which proves.
	\begin{align*}
		\boldsymbol{\Theta}_2\leq C\frac{\Ca}{\lambda_{\min}}\frac{\lambda_{\max}}{\lambda_{\min}}\frac{{\rm d}\sigma^2}{t+1+\Cb d}.
	\end{align*}
	Then consider the term $\boldsymbol{\Theta}_3$. Its single component has
	\begin{align*}
		\l\|\prod_{l=s+1}^{t}  \l(1-\frac{\Ca/K}{l+\Cb d}\r)\frac{1}{s+\Cb d}\boldsymbol{\Xi}_1^{(s)}\bigg|\calF_s\r\|_{\Psi_1}\leq C\lambda_{\max}\frac{\l(s+\Cb d\r)^{\Ca/K-1}}{\l(t+1+\Cb d\r)^{\Ca/K}}\|\Bbeta_{i}^{(s)}-\Bbeta_{i}^*\|^2.
	\end{align*}
	Under event $\cup_{s=t_1}^{t}\calE_s$, Lemma~\ref{lem:regularity} infers
	\begin{align*}
		\PP\l(\l|\sum_{s=t_1}^{t}\prod_{l=s+1}^{t}  \l(1-\frac{\Ca/K}{l+\Cb d}\r)\frac{1}{s+\Cb d}\boldsymbol{\Xi}_1^{(s)}\r|\geq c_1\frac{C^*}{\Ca}\frac{(d+\delta_{t+1})\ \sigma^2}{t+1+\Cb d}\r)\leq\exp\l(-c_2 (\Cb d+t)\frac{1}{\Ca}\frac{\lambda_{\min}^2}{\lambda_{\max}^2}\r),
	\end{align*}
	where $c_1<0.1$ is some small constant and then we have
	\begin{align*}
		\l|\boldsymbol{\Theta}_3\r|\leq c_1C^{**}\frac{1}{\lambda_{\min}}\frac{(d+\delta_{t+1})\; \sigma^2}{t+1+\Cb d}.
	\end{align*}
	As for the $\boldsymbol{\Theta}_4$ term, we have
	\begin{align*}
		\l\|\prod_{l=s+1}^{t}  \l(1-\frac{\Ca/K}{l+\Cb d}\r)\frac{1}{s+\Cb d}\boldsymbol{\Xi}_2^{(s)} \bigg|\calF_s\r\|_{\Psi_1}\leq C\sqrt{\lambda_{\max}} \sigma\frac{(s+\Cb d)^{\Ca/K-1}}{(t+1+\Cb d)^{\Ca/K}}\|\Bbeta_{i}^{(s)}-\Bbeta_i^*\|.
	\end{align*}
	Then under event $\cup_{s=t_1}^t \calE_s$, we have
	\begin{multline*}
		\PP\l(\l|\sum_{s=t_1}^{t}\prod_{l=s+1}^{t}  \l(1-\frac{\Ca/K}{l+\Cb d}\r)\frac{1}{s+\Cb d}\boldsymbol{\Xi}_2^{(s)}\r|\geq c_1\frac{C^*}{\Ca}\frac{(d+\delta_{t+1})\ \sigma^2}{t+1+\Cb d}\r)\\
		\leq\exp\l(-c_2 \min\l\{(d+\delta_{t+1})\frac{C^*}{\Ca}\frac{\lambda_{\min}}{\lambda_{\max}},\sqrt{d+\delta_{t+1}}\sqrt{d+t/\Cb}\frac{\sqrt{C^*\Cb}}{\Ca}\frac{\sqrt{\lambda_{\min}}}{\sqrt{\lambda_{\max}}}\r\}\r),
	\end{multline*}
	which leads to
	\begin{align*}
		\l|\boldsymbol{\Theta}_4\r|\leq c_1\frac{C^{**}}{\lambda_{\min}}\frac{(d+\delta_{t+1})\; \sigma^2}{t+1+\Cb d}.
	\end{align*}
	Finally, consider the term $\boldsymbol{\Theta}_5$. Under $\calF_s$ and under event $\calE_s,\calE_s^X$, we have
	\begin{align*}
		\l\|\prod_{l=s+1}^{t}  \l(1-\frac{\Ca/K}{l+\Cb d}\r)\l(\frac{1}{s+\Cb d}\r)^2\boldsymbol{\Xi}_3^{(s)}\bigg|\calF_s,\calE_s^X\r\|_{\Psi_1}\leq C\lambda_{\max}(d+\delta_{s})\sigma^2\frac{\l(s+\Cb d\r)^{\Ca-2}}{\l(t+1+\Cb d\r)^{\Ca}}.
	\end{align*}
	Then under event $\cup_{s=t_1}^t\l\{\calE_s\cup\calE_s^X\r\}$, we have
	\begin{multline*}
		\PP\l(\l|\sum_{s=t_1}^{t}\prod_{l=s+1}^{t}  \l(1-\frac{\Ca/K}{l+\Cb d}\r)\l(\frac{1}{s+\Cb d}\r)^2\boldsymbol{\Xi}_3^{(s)}\r|\geq c_1C^{**}\frac{\lambda_{\min}}{\Ca^2}\frac{(d+\delta_{t+1})\ \sigma^2}{t+1+\Cb d}\r)\\
		\leq\exp\l(-c_2\l(d+\frac{t}{\Cb}\r)\min\l\{\frac{\lambda_{\min}^2}{\lambda_{\max}^2}\frac{\Cb}{\Ca}\l(\frac{C^{**}}{\Ca}\r)^2,\frac{\lambda_{\min}}{\lambda_{\max}}\frac{C^{**}\Cb}{\Ca^2} \r\}\r),
	\end{multline*}
	which shows
	\begin{align*}
		\l|\boldsymbol{\Theta}_5\r|\leq c_1\frac{C^{**}}{\lambda_{\min}}\frac{(d+\delta_{t+1})\; \sigma^2}{t+1+\Cb d}.
	\end{align*}
	Thus, altogether, with $C^{**}=C^*+2C\Ca\frac{\lambda_{\max}}{\lambda_{\min}}$, we have
	\begin{align*}
		\|\Bbeta_{i}^{(t+1)} - \Bbeta_{i}^*\|^2&\leq \frac{C^*}{\lambda_{\min}}\frac{d\,\sigma^2}{t+1+\Cb d}+ C\frac{\Ca}{\lambda_{\min}}\frac{\lambda_{\max}}{\lambda_{\min}}\frac{{\rm d}\sigma^2}{t+1+\Cb d} + 3c_1\frac{C^{**}}{\lambda_{\min}}\frac{(d+\delta_{t+1})\; \sigma^2}{t+1+\Cb d}\\
		&\leq \frac{C^{**}}{\lambda_{\min}}\frac{(d+\delta_{t+1})\;\sigma^2}{t+1+\Cb d},
	\end{align*}
	holds with probability exceeding $1-|\calA|\exp(-c(d+t/\Ca))-|\calA|\exp(-c(d+\delta_{t+1}))$ for all $i\in\calA$. Moreover, with $\Cb\geq \Ca\frac{K}{\pi}\frac{\lambda_{\max}^2}{\lambda_{\min}^2}$, we have $\|\Bbeta_{i}^{(t+1)} - \Bbeta_{i}^*\|^2\leq C\frac{\sigma^2}{\lambda_{\max}}$; given $\frac{C^*}{\lambda_{\min}}\frac{(d+\delta_{t_1})\;\sigma^2}{t_1+\Cb d}\leq\frac{h^2}{16}$, we have $\|\Bbeta_{i}^{(t+1)} - \Bbeta_{i}^*\|^2\leq \frac{h^2}{4}$.
	~\\
	~\\
	~\\
	Then consider bandits in $i\in\calA^c$. Proposition~\ref{prop: conditional expectation} proves that under $\calE_t$, we have
	\begin{align*}
		\EE\l\{\|\Bbeta_{i}^{(t+1)} - \Bbeta_{i}^*\|^2\bigg| \calF_t \r\}\leq\|\Bbeta_{i}^{(t)}-\Bbeta_{i}^*\|^2-2\eta_{t}\pi_t\lambda_{\min}\|\Bbeta_{i}^{(t)}-\Bbeta_{i}^*\|^2+C\eta_{t}^2 \frac{\pi_t}{K}\lambda_{\max}\; d\sigma^2.
	\end{align*}
	Substitute $\eta_{t}=\frac{\Ca}{\lambda_{\min}}\frac{1}{t+\Cb d}$ into the above equation and then we have
	\begin{align*}
		\EE\l\{\|\Bbeta_{i}^{(t+1)} - \Bbeta_{i}^*\|^2\bigg|\calF_t  \r\}\leq\l(1-\frac{\Ca\,\pi_t}{t+\Cb d}\r)\|\Bbeta_{i}^{(t)}-\Bbeta_{i}^*\|^2+C\Ca^2\frac{\pi_t}{K}\frac{\lambda_{\max}}{\lambda_{\min}^2}\l(\frac{1}{t+\Cb d}\r)^2.
	\end{align*}
	Thus we have
	\begin{align*}
		\|\Bbeta_{i}^{(t+1)} - \Bbeta_{i}^*\|^2
		&\leq\underbrace{\prod_{s=t_1}^{t}  \l(1-\frac{\Ca\pi_s}{s+\Cb d}\r)\|\Bbeta_{i}^{(t_1)}-\Bbeta_i^*\|^2}_{\boldsymbol{\Gamma}_1}+\underbrace{C\Ca^2d\sigma^2\frac{1}{K}\frac{\lambda_{\max}}{\lambda_{\min}^2}\sum_{s=t_1}^{t} \prod_{l=s+1}^{t}  \l(1-\frac{\Ca\pi_l}{l+\Cb d}\r)\frac{\pi_s}{(s+\Cb d)^2}}_{\boldsymbol{\Gamma}_2}\\
		&~~~~~~~~~~-\underbrace{\frac{2\,\Ca}{\lambda_{\min}}\sum_{s=t_1}^{t}\prod_{l=s+1}^{t}  \l(1-\frac{\Ca\pi_l}{l+\Cb d}\r)\frac{1}{s+\Cb d}\boldsymbol{\Xi}_1^{(s)}}_{\boldsymbol{\Gamma}_3}\\
		&~~~~~~~~~~-\underbrace{\frac{2\, \Ca}{\lambda_{\min}}\sum_{s=t_1}^{t}\prod_{l=s+1}^{t}  \l(1-\frac{\Ca\pi_l}{l+\Cb d}\r)\frac{1}{s+\Cb d}\boldsymbol{\Xi}_2^{(s)}}_{\boldsymbol{\Gamma}_4} \\
		&~~~~~~~~~~+\underbrace{\frac{\Ca^2}{\lambda_{\min}^2}\sum_{s=t_1}^{t}\prod_{l=s+1}^{t}  \l(1-\frac{\Ca\pi_l}{l+\Cb d}\r)\l(\frac{1}{s+\Cb d}\r)^2\boldsymbol{\Xi}_3^{(s)}}_{\boldsymbol{\Gamma}_5}.
	\end{align*}
	For $\boldsymbol{\Gamma}_1$, it simply has $\boldsymbol{\Gamma}_1\leq \|\Bbeta_{i}^{(t_1)}-\Bbeta_{i}^*\|^2\leq\frac{h^2}{16}$. As for $\boldsymbol{\Gamma}_2$, we have
	\begin{align*}
		\boldsymbol{\Gamma}_2\leq C\Ca^2d\sigma^2\frac{1}{K}\frac{\lambda_{\max}}{\lambda_{\min}^2}\sum_{s=t_1}^{t}\frac{1}{(s+\Cb d)^2}\leq C\frac{\Ca^2}{K}\frac{\lambda_{\max}}{\lambda_{\min}^2}\frac{ d\sigma^2}{t_1+\Cb d}.
	\end{align*}
	It is worth noting that under event $\{\max_{j}\|\Bbeta_{j}^{(s)}-\Bbeta_{j}^*\|\leq \frac{h}{2}\}$, $a_t=i$ only when $\alpha_t=1$ and $\II\{a_t=i\}$ is independent of $\X_t$, $\Bbeta_i^{(t)}$, following $\II\{a_t=i\}\sim \operatorname*{Bernoulli}\l(\frac{\pi_t}{K}\r)$. Then a single component of $\boldsymbol{\Gamma}_3$ has
	\begin{align*}
		\l\|\prod_{l=s+1}^{t}  \l(1-\frac{\Ca\pi_l}{l+\Cb d}\r)\frac{1}{s+\Cb d}\boldsymbol{\Xi}_1^{(s)}\bigg|\calF_s\r\|_{\Psi_1}\leq \II\l\{\pi_s\neq0\r\}\cdot\lambda_{\max}\frac{h^2}{s+\Cb d}.
	\end{align*}
	By Lemma~\ref{teclem:azuma}, we have
	\begin{multline*}
		\PP\l(\l|\sum_{s=t_1}^{t}\prod_{l=s+1}^{t}  \l(1-\frac{\Ca\pi_l}{l+\Cb d}\r)\frac{1}{s+\Cb d}\boldsymbol{\Xi}_1^{(s)}\r|> c_1 \II\l\{\pi_s\neq0\r\} h^2 \frac{\lambda_{\min}}{\Ca}\r)\\
		\leq\II\l\{\pi_s\neq0\r\}\exp\l(-c_2 d\min\l\{ \frac{\Cb}{\Ca^2}\frac{\lambda_{\min}^2}{\lambda_{\max}^2}, \frac{\Cb}{\Ca}\frac{\lambda_{\min}}{\lambda_{\max}}\r\}\r),
	\end{multline*}
	under which it has
	\begin{align*}
		\big|\boldsymbol{\Gamma}_3\big|\leq c_1\II\l\{\pi_s\neq0\r\}\cdot h^2.
	\end{align*}
	Similarly, it has
	\begin{align*}
		\l\|\prod_{l=s+1}^{t}  \l(1-\frac{\Ca\pi_l}{l+\Cb d}\r)\frac{1}{s+\Cb d}\boldsymbol{\Xi}_2^{(s)}\bigg| \calF_s\r\|_{\Psi_1}\leq \II\l\{\pi_s\neq0\r\}\sqrt{\lambda_{\max}}\sigma\frac{h}{s+\Cb d}.
	\end{align*}
	Then Lemma~\ref{teclem:azuma} guarantees
	\begin{multline*}
		\PP\l(\l|\sum_{s=t_1}^{t}\prod_{l=s+1}^{t}  \l(1-\frac{\Ca\pi_l}{l+\Cb d}\r)\frac{1}{s+\Cb d}\boldsymbol{\Xi}_2^{(s)}\r|> c_1\II\l\{\pi_s\neq0\r\}h^2\frac{\lambda_{\min}}{\Ca}\r)\\
		\leq\II\l\{\pi_s\neq0\r\}\exp\l(-cd\min\l\{\frac{C^*}{\Ca^2}\frac{\lambda_{\min}}{\lambda_{\max}},\frac{\sqrt{\lambda_{\min}}}{\sqrt{\lambda_{\max}}}\frac{\sqrt{C^*\Cb}}{\Ca}\r\}\r),
	\end{multline*}
	which implies $\big|\boldsymbol{\Gamma}_4\big|\leq c_1h^2$. Then consider the term $\boldsymbol{\Gamma}_5$,
	\begin{align*}
		\l\|\prod_{l=s+1}^{t}  \l(1-\frac{\Ca\pi_l}{l+\Cb d}\r)\l(\frac{1}{s+\Cb d}\r)^2\boldsymbol{\Xi}_3^{(s)}\bigg|\calF_s\r\|_{\Psi_1}\leq \II\l\{\pi_s\neq0\r\} \l(\frac{1}{s+\Cb d}\r)^2d\lambda_{\max}\sigma^2.
	\end{align*}
	Follow Lemma~\ref{teclem:azuma},
	\begin{multline*}
		\PP\l(\l|\sum_{s=t_1}^{t}\prod_{l=s+1}^{t}  \l(1-\frac{\Ca\pi_l}{l+\Cb d}\r)\l(\frac{1}{s+\Cb d}\r)^2\boldsymbol{\Xi}_3^{(s)}\r| > c_1\II\l\{\pi_s\neq0\r\}\frac{\lambda_{\min}^2}{\Ca^2}h^2\r)\\
		\leq\II\l\{\pi_s\neq0\r\}\exp\l(-cd\min\l\{\frac{\lambda_{\min}^2}{\lambda_{\max}^2}\frac{C^{*2}\Cb}{\Ca^4},\frac{\lambda_{\min}}{\lambda_{\max}}\frac{C^*\Cb}{\Ca^2}\r\}\r),
	\end{multline*}
	by which we have $\big|\boldsymbol{\Gamma}_5\big|\leq c_1h^2$. Thus altogether, given $\Cb\geq C\Ca^2\lambda_{\max}^2/\lambda_{\min}^2$, the following holds for all $i\in\calA^c$ with probability over $1-\II\l\{\pi_s\neq0\r\}|\calA^c|\exp(-cd)$
	\begin{align*}
		\|\Bbeta_{i}^{(t+1)}-\Bbeta_{i}^*\|^2 \leq \frac{h^2}{4}.
	\end{align*}
\end{proof}
\subsection{Proof of Theorem~\ref{thm:regret}}
\begin{proof} First, consider the regret at time $t$, which can be decomposed into terms,
	\begin{multline*}
		\EE\l\{\max_{i\in[K]} \X_t^{\top}\Bbeta_{i}^* - \X_t^{\top}\Bbeta_{a_t}^*\r\}\\
		= \underbrace{\EE\l\{\II\{\alpha_t=1\}\cdot \l(\max_{i\in[K]} \X_t^{\top}\Bbeta_{i}^* - \X_t^{\top}\Bbeta_{a_t}^*\r)\r\}}_{\R_1} + \underbrace{\EE\l\{\II\{\alpha_t=0\}\cdot \l(\max_{i\in[K]} \X_t^{\top}\Bbeta_{i}^* - \X_t^{\top}\Bbeta_{a_t}^*\r)\r\}}_{\R_2}.
	\end{multline*}
	We shall analyze when $t\leq t_1$ and $t>t_1$ separately. When $t<t_1$,
	\begin{align*}
		\R_1&= \frac{\pi_t}{K} \cdot\sum_{j=1}^{K}\EE\l\{ \max_{i}\X_t^{\top}\Bbeta_{i}^*-\X_t^{\top}\Bbeta_{j}^*\r\}\\
		&\leq\frac{\pi_t}{K}\cdot\min\l\{\sum_{j=1}^{K}\sum_{i=1}^{K}\EE\l\{ \l|\X_t^{\top}\Bbeta_{i}^*-\X_t^{\top}\Bbeta_{j}^*\r|\r\}, \sum_{j=1}^{K}\EE\l\{ \|\X_t\|\cdot \max_{i}\|\Bbeta_{i}^*-\Bbeta_{j}^*\|\r\} \r\}\\
		&\leq \sqrt{\lambda_{\max}}\frac{\pi_t}{K} \cdot\min\l\{\sum_{j=1}^{K}\sum_{i=1}^{K} \|\Bbeta_{i}^*-\Bbeta_{j}^*\|, \sqrt{d}\sum_{j=1}^K \max_{i}\|\Bbeta_{i}^*-\Bbeta_{j}^*\| \r\}.
	\end{align*}
	Denote event $\calE_{t_1}:=\l\{ \|\Bbeta_{i}^{(t)}-\Bbeta_{i}^*\|^2\leq C^*\frac{1}{\lambda_{\min}}\frac{d+\delta_t}{t+\Cb d}\sigma^2,\quad i\in[K],\text{ for all } t\in[t_1]\r\}$. Theorem~\ref{thm:LBD-estimation} proves $$\PP(\calE_{t_1})\geq 1-c'K\Cb\exp(-cd)-K\exp(-cd)\sum_{s=0}^{t_1}\exp(-c\delta_{s}).$$
	Denote event $\calE^{(t)}:=\l\{\|\Bbeta_{i}^{(t)}-\Bbeta_{i}^*\|^2\leq C^*\frac{1}{\lambda_{\min}}\frac{d+\delta_t}{t+\Cb d}\sigma^2,\quad i\in[K]\r\}$. Then 
	\begin{multline*}
		\R_2= \underbrace{(1-\pi_t) \EE\l\{ \l(\max_{i\in[K]} \X_t^{\top}\Bbeta_{i}^* - \X_t^{\top}\Bbeta_{a_t}^*\r)\cdot \II\l\{ \calE^{(t)} \r\}\r\}}_{\R_{21}}\\
		+\underbrace{(1-\pi_t)\EE\l\{ \l(\max_{i\in[K]} \X_t^{\top}\Bbeta_{i}^* - \X_t^{\top}\Bbeta_{a_t}^*\r)\cdot \II\{\calE^{(t)C}\}\r\}}_{\R_{22}}.
	\end{multline*}
	Denote $a_t^*\in \arg\max_{i} \X^{\top}\Bbeta_{i}^*$ for simplicity. Observe that $\X_t^\top\Bbeta_{a_t}^{(t)}\geq \max_{i} \X_t^\top\Bbeta_{i}^{(t)}$ and then $\R_2$ has, 
	\begin{align*}
		\R_{21}&\leq (1-\pi_t)\cdot\EE\l\{ \l\{\X_t^{\top}\Bbeta_{a_t^*}^* - \X_t^{\top}\Bbeta_{a_t^*}^{(t)}+ \X_t^{\top}\Bbeta_{a_t}^{(t)}- \X_t^{\top}\Bbeta_{a_t}^* \r\} \cdot \II\{\calE^{(t)}\}\r\}\\
		&\leq (1-\pi_t)\cdot\EE\l\{ \l|\X_t^{\top}\Bbeta_{a_t^*}^* - \X_t^{\top}\Bbeta_{a_t^*}^{(t)} \r|\cdot \II\{\calE^{(t)}\} + \l| \X_t^{\top}\Bbeta_{a_t}^{(t)}- \X_t^{\top}\Bbeta_{a_t}^* \r|\cdot \II\{\calE^{(t)}\}\r\}\\
		&\leq 2(1-\pi_t)\cdot\min\l\{\sum_{i=1}^{K} \EE\l|\X_t^{\top}\l(\Bbeta_{i}^*-\Bbeta_{i}^{(t)}\r)\r|\cdot \II\{\calE^{(t)}\}, \EE\l|\|\X\|\cdot \max_{i}\|\Bbeta_{i}^{(t)}-\Bbeta_{i}^*\|\cdot \II\{\calE^{(t)}\}\r|\r\}.
	\end{align*}
	Thus taking event $\calE^{(t)}$ into consideration, we have
	\begin{align*}
		\R_{21}\leq2(1-\pi_t)\sqrt{C^*}\sqrt{\frac{\lambda_{\max}}{\lambda_{\min}}} \sqrt{\frac{d+\delta_{t}}{t+\Cb d}}\, \sigma\, \min\l\{K,\; \sqrt{d}\r\}.
	\end{align*}
	Similarly, we have
	\begin{align*}
		\R_{22}\leq(1-\pi_t) \min\l\{\sum_{j=1}^{K}\sum_{i=1}^{K}\EE\l| \X_t^{\top}\l(\Bbeta_{i}^*-\Bbeta_{j}^*\r)\r|, \EE\|\X_t\|\cdot \max_{i,j} \|\Bbeta_{i}^*-\Bbeta_{j}^*\|\r\}\cdot \PP(\calE^{(t)C}).
	\end{align*}
	Also note that $\sum_{t=0}^{t_1} \PP(\calE^{(t)C})=1-\PP(\calE_{t_1})\leq c'K\Cb\exp(-cd)+K\exp(-cd)\sum_{s=0}^{t_1}\exp(-c\delta_{s})$. Thus, we have
	\begin{multline*}
		\textsf{Regret}(t_1)\leq \sqrt{\lambda_{\max}}\frac{\sum_{s=0}^{t_1}\pi_s}{K} \cdot\min\l\{\sum_{j=1}^{K}\sum_{i=1}^{K} \|\Bbeta_{i}^*-\Bbeta_{j}^*\|, \sqrt{d}\sum_{j=1}^K \max_{i}\|\Bbeta_{i}^*-\Bbeta_{j}^*\| \r\}\\
		+C\sqrt{C^*}\sqrt{\frac{\lambda_{\max}}{\lambda_{\min}}} \sigma \min\l\{ K, \sqrt{d}\r\}\sum_{s=0}^{t_1}\sqrt{\frac{d+\delta_{s}}{s+\Cb d}}\\
		+K\sqrt{\lambda_{\max}}\exp(-cd)\min\l\{\sum_{j=1}^{K}\sum_{i=1}^{K}\EE\| \Bbeta_{i}^*-\Bbeta_{j}^*\|, \sqrt{d} \max_{i,j} \|\Bbeta_{i}^*-\Bbeta_{j}^*\|\r\}\cdot\l(c'\Cb+\sum_{s=0}^{t_1}\exp(-c\delta_{s}) \r)
	\end{multline*}
	Then consider the regret for $t>t_1$. For $t>t_1$, the bound of $\R_1$ is same as the case when $t\leq t_1$. For the analysis of $\R_2$, we denote 
	$$\calE^{(t)}:=\l\{\|\Bbeta_{i}^{(t)}-\Bbeta_{i}^*\|^2\leq C^*\frac{1}{\lambda_{\min}}\frac{d+\delta_t}{t+\Cb d}\sigma^2,\ i\in\calA;\, \|\Bbeta_{i}^{(t)}-\Bbeta_{i}^*\|^2\leq\frac{h^2}{4},\ i\in\calA^c\r\}.$$
	Then according to Lemma~\ref{lem:regularity}, we know under the above event, $a_t\in\calA$. Thus, we have
	\begin{align*}
		&\EE\l\{ \l(\max_{i\in[K]} \X_t^{\top}\Bbeta_{i}^* - \X_t^{\top}\Bbeta_{a_t}^*\r)\cdot \II\l\{ \calE^{(t)} \r\}\r\}\\
		&~~~~~\leq \min\l\{\sum_{i\in\calA} \EE\l|\X_t^{\top}\l(\Bbeta_{i}^*-\Bbeta_{i}^{(t)}\r)\r|\cdot \II\{\calE^{(t)}\}, \EE\l|\|\X\|\cdot \max_{i\in\calA}\|\Bbeta_{i}^{(t)}-\Bbeta_{i}^*\|\cdot \II\{\calE^{(t)}\}\r|\r\}\\
		&~~~~~\leq \sqrt{C^*}\sqrt{\frac{\lambda_{\max}}{\lambda_{\min}}} \sqrt{\frac{d+\delta_{t}}{t+\Cb d}}\, \sigma\, \min\l\{ K,\; \sqrt{d}\r\}.
	\end{align*}
	The analysis of $\EE\l\{ \l(\max_{i\in[K]} \X_t^{\top}\Bbeta_{i}^* - \X_t^{\top}\Bbeta_{a_t}^*\r)\cdot \II\l\{ \calE^{(t)C} \r\}\r\} $ is same as the $t<t_1$ case. Thus in all, we have
	\begin{multline*}
		\textsf{Regret}(T) \leq \sqrt{\lambda_{\max}}\frac{\sum_{s=0}^{T}\pi_s}{K} \cdot\min\l\{\sum_{j=1}^{K}\sum_{i=1}^{K} \|\Bbeta_{i}^*-\Bbeta_{j}^*\|, \sqrt{d}\sum_{j=1}^K \max_{i}\|\Bbeta_{i}^*-\Bbeta_{j}^*\| \r\}\\
		+C\sqrt{C^*}\sqrt{\frac{\lambda_{\max}}{\lambda_{\min}}} \sigma \min\l\{ K, \sqrt{d}\r\}\sum_{s=0}^{T}\sqrt{\frac{d+\delta_{s}}{s+\Cb d}}\\
		+\sqrt{\lambda_{\max}}\min\l\{\sum_{j=1}^{K}\sum_{i=1}^{K}\EE\| \Bbeta_{i}^*-\Bbeta_{j}^*\|, \sqrt{d} \max_{i,j} \|\Bbeta_{i}^*-\Bbeta_{j}^*\|\r\}\\ \times\l(c'\Cb\exp(-cd)+K\sum_{s=0}^{t_1}\exp(-c(\delta_{s}+d))+|\calA|\sum_{s=t_1+1}^{T}\exp(-c(\delta_{s}+d))+ |\calA^c|\sum_{s=t_1+1}^{T}\II\{\pi_s\neq 0\}\r),
	\end{multline*}
which completes the proof.
\end{proof}
\subsection{Proof of Proposition~\ref{prop:Bahadur}}
For brevity of writing, we introduce some notations. We denote
\begin{align*}
	\bSigma_{i,t}&:=\EE\l\{\II\{a_t=i\}\cdot\X_t\X_t^{\top}\bigg|\calF_t\r\}\\
	&=\frac{\pi_t}{K}\EE\l\{\X_t\X_t^{\top}\big|\calF_t\r\}+(1-\pi_t)\cdot\EE\l\{\X_t\X_t^{\top}\cdot\II\l\{\X_t\in\calX\l(i,\l\{\Bbeta_{j}^{(t)}\r\}_{i\in[K]}\r) \r\}\bigg|\calF_t\r\},
\end{align*}
and
\begin{align*}
	\bSigma_{i,t}^*:=\frac{\pi_t}{K}\EE\l\{\X_t\X_t^{\top}\big|\calF_t\r\}+(1-\pi_t)\cdot\EE\l\{\X_t\X_t^{\top}\cdot\II\l\{\X_t\in\calU_i^* \r\}\big|\calF_t\r\}.%=\frac{\pi_t}{K}\bSigma+(1-\pi_t)\bSigma_i^*.
\end{align*}
The matrix $\bSigma_{i,t}^*$ satisfies the following equation
\begin{align}
	\l(\frac{\pi_t}{K}+(1-\pi_t)\r)\cdot\lambda_{\min}\cdot\I_{d}\preceq\bSigma_{i,t}^*\preceq \lambda_{\max}\cdot\I_{d},
	\label{eq1:bahadur}
\end{align}
and we have the convergence in mean that
\begin{align}
	\lim_{t\to+\infty}\EE\l\|\bSigma_{i,t}^*-\bSigma_{i}(\pi^*)\r\|=0.
	\label{eq4:bahadur}
\end{align}
Additionally, $\bSigma_{i,t}-\bSigma_{i,t}^*$ can be written into,
\begin{multline*}
	\l\|\bSigma_{i,t}- \bSigma_{i,t}^*\r\|
	=(1-\pi_t)\l\|\EE\l\{\X_t\X_t^\top\cdot\l(\II\l\{\X\in\calX_t\l(i,\l\{\Bbeta_{j}^{(t)}\r\}_{i\in[K]}\r) \r\}-\II\l\{\X_t\in\calU_i^*\r\} \r)\bigg| \calF_t\r\} \r\|
\end{multline*}
Under Assumption~\ref{assm:density asymp} and Lemma~\ref{lem:decision region}, $ \l\|\bSigma_{i,t}- \bSigma_{i,t}^*\r\|$ can be upper bounded with,
\begin{align}
	\l\|\bSigma_{i,t}- \bSigma_{i,t}^*\r\|\leq\left\{
	\begin{array}{ll}
		(1-\pi_t)\cdot \kappa_0\cdot\max_{i\in\calA}\|\Bbeta_{i}^{(t)}-\Bbeta_{i}^*\|, &\quad\max_{i\in[K]}\|\Bbeta_{i}^{(t)}-\Bbeta_{i}^*\|\leq h_0\\
		(1-\pi_t)\cdot\lambda_{\max},& \quad\text{for all }\{\Bbeta_i\}_{i\in[K]}
	\end{array}.
	\right.
	\label{eq3:bahadur}
\end{align}
We first prove the following proposition.
\begin{proposition}
    Under the same assumptions and conditions as Proposition~\ref{prop:Bahadur}, we have
    \begin{enumerate}
        \item[(1)] if $\pi^*=0$, then under the event $\cup_{t=t_1}\{\max\|\Bbeta_i^{(t)}-\Bbeta_i^*\|\leq \frac{h}{2}\}$, the following holds for $i\in\calA$,
         \begin{align*}
        \EE\|\Bbeta_i^{(t)}-\Bbeta_i^*\|^2\leq C\frac{d}{t+\Cb d}\frac{\sigma^2}{\lambda_{\min}},
    \end{align*}
        \item[(2)] if $\pi^*>0$ and $\Ca\geq 4K/\pi^*$, the following holds for all $i\in[K]$,
        \begin{align*}
        \EE\|\Bbeta_{i}^{(t+1)} - \Bbeta_{i}^*\|^2 \leq \l(\frac{N+1+\Cb d}{t+1+\Cb d}\r)^{\Ca\pi^*/K}\|\Bbeta_i^{(N)}-\Bbeta_i^*\|^2+\frac{Cd}{t+1+\Cb d}\frac{\sigma^2}{\lambda_{\min}},
    \end{align*}
    where $N$ is some integer such that for all $t\geq N$, $\pi_t\geq \pi^*/2$.
    \end{enumerate}
    \label{prop:LBD_expectation_square}
\end{proposition}
\begin{proof}
    We first prove the $\pi^*=0$ case and then we discuss the $\pi^*>0$ scheme. 
    
    \noindent \underline{\underline{When $\pi^*=0$.}} For $i\in\calA$, Proposition~\ref{prop: conditional expectation} proves that
    \begin{multline*}
        \EE\l\{\|\Bbeta_{i}^{(t+1)} - \Bbeta_{i}^*\|^2 \big| \calF_t\r\}\leq \|\Bbeta_{i}^{(t)}-\Bbeta_{i}^*\|^2-2\eta_{t}\frac{\pi_t}{K}\lambda_{\min}\|\Bbeta_{i}^{(t)}-\Bbeta_{i}^*\|^2+2\eta_{t}^2d\lambda_{\max}^2\|\Bbeta_{i}^{(t)}-\Bbeta_{i}^*\|^2\\-2\eta_{t}(1-\pi_t)\lambda_{\min}\|\Bbeta_{i}^{(t)}-\Bbeta_{i}^*\|^2\cdot\II\l\{\max_{j}\|\Bbeta_{j}^{(t)}-\Bbeta_{i}^*\|\leq\frac{h}{2}\r\}+2\eta_{t}^2d\lambda_{\max}\sigma^2.
    \end{multline*}
    Take expectation on each side of the above equation and it leads to
    \begin{multline*}
        \EE\l\{\|\Bbeta_{i}^{(t+1)} - \Bbeta_{i}^*\|^2 \r\}\leq \l(1-2\eta_t\frac{\pi_t}{K}\lambda_{\min}+2\eta_t^2d\lambda_{\max}^2\r)\EE\l\{\|\Bbeta_i^{(t)}-\Bbeta_i^*\|^2\r\}\\
        -2\eta_{t}(1-\pi_t)\lambda_{\min}\EE\l\{\|\Bbeta_{i}^{(t)}-\Bbeta_{i}^*\|^2\cdot\II\l\{\max_{j}\|\Bbeta_{j}^{(t)}-\Bbeta_{i}^*\|\leq \frac{h}{2}\r\}\r\}+2\eta_{t}^2d\lambda_{\max}\sigma^2.
    \end{multline*}
    For $t\leq t_1$, according to the exploration rate scheme we have $\pi_t\geq \pi$.
    Hence, we have
    \begin{multline*}
        \EE\l\{\|\Bbeta_{i}^{(t+1)} - \Bbeta_{i}^*\|^2 \r\}\leq \l(1-2\frac{\Ca}{t+\Cb d}\frac{\pi}{K}+2\frac{\Ca^2d}{(t+\Cb d)^2}\frac{\lambda_{\max}^2}{\lambda_{\min}^2}\r)\EE\l\{\|\Bbeta_i^{(t)}-\Bbeta_i^*\|^2\r\}\\
        +2\frac{\Ca^2d}{(t+\Cb d)^2}\frac{\lambda_{\max}}{\lambda_{\min}^2}\sigma^2.
    \end{multline*}
    According to the choice of $\Cb$, we have
    \begin{align*}
        2\frac{\Ca^2d}{(t+\Cb d)^2}\frac{\lambda_{\max}^2}{\lambda_{\min}^2}\leq \frac{\Ca}{t+\Cb d}\frac{\pi}{K}.
    \end{align*}
    Hence, we have
    \begin{align*}
        &\EE\l\{\|\Bbeta_{i}^{(t+1)} - \Bbeta_{i}^*\|^2 \r\}\leq \l(1-2\frac{\Ca}{t+\Cb d}\frac{\pi}{K}\r)\EE\l\{\|\Bbeta_i^{(t)}-\Bbeta_i^*\|^2\r\}
        +2\frac{\Ca^2d}{(t+\Cb d)^2}\frac{\lambda_{\max}}{\lambda_{\min}^2}\sigma^2\\
        &~~~\leq \prod_{l=0}^{t} \l(1-2\frac{\Ca}{l+\Cb d}\frac{\pi}{K}\r)\EE\l\{\|\Bbeta_i^{(0)}-\Bbeta_i^*\|^2\r\}
        +2\sum_{k=0}^t\prod_{l=k}^t\l(1-2\frac{\Ca}{l+\Cb d}\frac{\pi}{K}\r)\frac{\Ca^2d}{(k+\Cb d)^2}\frac{\lambda_{\max}}{\lambda_{\min}^2}\sigma^2.
    \end{align*}
    By Lemma~\ref{teclem:prod upper bound} and Lemma~\ref{teclem:sum upper bound}, we have
    \begin{align*}
        \EE\l\{\|\Bbeta_{i}^{(t+1)} - \Bbeta_{i}^*\|^2 \r\}\leq\frac{Cd}{t+1+\Cb d}\frac{\sigma^2}{\lambda_{\min}}.
    \end{align*}
    On the other hand, for $t\geq t_1$, under the conditions $\cup_{t=t_1}^{+\infty}\{\max_{j\in[k]} \|\Bbeta_j^{(t)}-\Bbeta_j^*\|\leq\frac{h}{2}\}$, we have
    \begin{multline*}
        \EE\l\{\|\Bbeta_{i}^{(t+1)} - \Bbeta_{i}^*\|^2 \r\}\leq \l(1-2\frac{\Ca}{t+\Cb d}\frac{1}{K}+2\frac{\Ca^2d}{(t+\Cb d)^2}\frac{\lambda_{\max}^2}{\lambda_{\min}^2}\r)\EE\l\{\|\Bbeta_i^{(t)}-\Bbeta_i^*\|^2\r\}\\
        +2\frac{\Ca^2d}{(t+\Cb d)^2}\frac{\lambda_{\max}}{\lambda_{\min}^2}\sigma^2.
    \end{multline*}
    In a similar way, we have $\EE\l\{\|\Bbeta_{i}^{(t+1)} - \Bbeta_{i}^*\|^2 \r\}\leq\frac{Cd}{t+1+\Cb d}\frac{\sigma^2}{\lambda_{\min}}$.

    \noindent \underline{\underline{When $\pi^*>0$.}} For all $i\in[K]$, we have
    \begin{multline*}
        \EE\l\{\|\Bbeta_{i}^{(t+1)} - \Bbeta_{i}^*\|^2 \r\}\leq \l(1-2\frac{\Ca}{t+\Cb d}\frac{\pi_t}{K}+2\frac{\Ca^2d}{(t+\Cb d)^2}\frac{\lambda_{\max}^2}{\lambda_{\min}^2}\r)\EE\l\{\|\Bbeta_i^{(t)}-\Bbeta_i^*\|^2\r\}\\
        +2\frac{\Ca^2d}{(t+\Cb d)^2}\frac{\lambda_{\max}}{\lambda_{\min}^2}\sigma^2.
    \end{multline*}
    In this case, $\lim_{t\to+\infty} \pi_t=\pi^*$ and there exists $N>0$ such that for all $t\geq N$, we have $\pi_t>\frac{1}{2}\pi^*$. Then by accumulating the above inequality, we have
    \begin{align*}
        \EE\l\{\|\Bbeta_{i}^{(t+1)} - \Bbeta_{i}^*\|^2 \r\}&\leq \prod_{l=N}^{t} \l(1-\frac{\Ca}{l+\Cb d}\frac{\pi^*}{K}\r)\EE\l\{\|\Bbeta_i^{(N)}-\Bbeta_i^*\|^2\r\}\\
        &~~~~~~~~+2\sum_{k=N}^t\prod_{l=k}^t\l(1-\frac{\Ca}{l+\Cb d}\frac{\pi^*}{K}\r)\frac{\Ca^2d}{(k+\Cb d)^2}\frac{\lambda_{\max}}{\lambda_{\min}^2}\sigma^2\\
        &\leq \l(\frac{N+1+\Cb d}{t+1+\Cb d}\r)^{\Ca\pi^*/K}\|\Bbeta_i^{(N)}-\Bbeta_i^*\|^2+\frac{Cd}{t+1+\Cb d}\frac{\sigma^2}{\lambda_{\min}},
    \end{align*}
   where the last line uses Lemma~\ref{teclem:sum upper bound} and \ref{teclem:prod upper bound}.
    
\end{proof}

\begin{proposition}
    Under the same assumptions and conditions as Proposition~\ref{prop:Bahadur}, we have
    \begin{enumerate}
        \item[(1)] if $\pi^*=0$, then under the event $\cup_{t=t_1}\{\max\|\Bbeta_i^{(t)}-\Bbeta_i^*\|\leq \frac{h}{2}\}$, the following holds for $i\in\calA$,
         \begin{align*}
        \EE\|\Bbeta_i^{(t)}-\Bbeta_i^*\|^4\leq C\l(\frac{d}{t+\Cb d}\r)^2\frac{\sigma^4}{\lambda_{\min}^2},
    \end{align*}
        \item[(2)] if $\pi^*>0$ and $\Ca\geq 4K/\pi^*$, the following holds for all $i\in[K]$,
        \begin{align*}
        \EE\|\Bbeta_{i}^{(t+1)} - \Bbeta_{i}^*\|^4 \leq \l(\frac{N+1+\Cb d}{t+1+\Cb d}\r)^{2\Ca\pi^*/K}\|\Bbeta_N-\Bbeta^*\|^4+C\l(\frac{d}{t+1+\Cb d}\r)^2\frac{\sigma^4}{\lambda_{\min}^2},
    \end{align*}
    where $N$ is some integer such that for all $t\geq N$, $\pi_t\geq \pi^*/2$.
    \end{enumerate}
    \label{prop:LBD_expectation_four}
\end{proposition}

To prove Proposition~\ref{prop:LBD_expectation_four}, we need to expand the fourth moment based on the update of $\Bbeta_i^{(t+1)}$. The proof of Proposition~\ref{prop:LBD_expectation_four} is similar to Proposition~\ref{prop:LBD_expectation_square}, hence omitted.
\begin{proof}[Proof of Proposition~\ref{prop:Bahadur}]
	Under the conditions of Proposition~\ref{prop:Bahadur}, Proposition~\ref{prop:LBD_expectation_square} proves $$\|\Bbeta_i^{(t)}-\Bbeta_i^*\|^2=O_{\PP}\l(\frac{1}{t}\r).$$ Recall that the update is,
	\begin{align*}
		\Bbeta_{i}^{(t+1)}-\Bbeta_{i}^*&= \Bbeta_{i}^{(t)}-\Bbeta_{i}^* - \eta_{t}\cdot\II\{a_t=i\}\cdot(\X_t^{\top}\Bbeta_{i}^{(t)}-Y_t)\X_t\\
		&=\l(\I-\eta_{t}\cdot\II\{a_t=i\}\cdot\X_t\X_t^{\top}\r)\l(\Bbeta_{i}^{(t)}-\Bbeta_{i}^*\r)+\eta_{t}\cdot\II\{a_t=i\}\cdot\xi_t\cdot\X_t,
	\end{align*}
	which can be decomposed into,
	\begin{align*}
		\Bbeta_{i}^{(t+1)}-\Bbeta_{i}^*&=\l(\I -\eta_{t}\cdot\EE\l\{\II\{a_t=i\}\cdot\X_t\X_t^{\top}\bigg|\calF_t\r\}\r)\l(\Bbeta_{i}^{(t)}-\Bbeta_{i}^* \r)\\
		&-\eta_{t}\cdot\l(\II\{a_t=i\}\cdot\X_t\X_t^{\top} -\EE\l\{\II\{a_t=i\}\cdot\X_t\X_t^{\top}\big|\calF_t\r\} \r)\l(\Bbeta_{i}^{(t)}-\Bbeta_{i}^*\r)\\
		&+\eta_{t}\cdot \X_t\cdot\II\{a_t=i\}\cdot\xi_t.
	\end{align*}
	 With the notations defined above, the update can be further written as,
	\begin{multline*}
		\Bbeta_{i}^{(t+1)}-\Bbeta_{i}^*=\l(\I-\eta_{t}\cdot\bSigma_{i,t}^*\r)\l(\Bbeta_{i}^{(t)}-\Bbeta_{i}^*\r)+\eta_{t}\cdot\II\{a_t=i\}\cdot\xi_t\cdot\X_{t}\\
		-\eta_{t}\l(\bSigma_{i,t}-\bSigma_{i,t}^*\r) \l(\Bbeta_{i}^{(t)}-\Bbeta_{i}^*\r)-\eta_{t}\cdot\l(\II\{a_t=i\}\cdot\X_t\X_t^{\top}-\bSigma_{i,t} \r)\l(\Bbeta_{i}^{(t)}-\Bbeta_{i}^*\r).
	\end{multline*}
	We accumulate the above update starting from $\Bbeta_{i}^{(0)}$ and it arrives at
	\begin{multline*}
		\Bbeta_{i}^{(t+1)}-\Bbeta_{i}^*=\underbrace{\prod_{s=1}^{t} \l(\I-\eta_{s}\cdot\bSigma_{i,s}^*\r)\l(\Bbeta_{i}^{(0)}-\Bbeta_{i}^*\r)}_{\B_1}+\underbrace{\sum_{s=0}^{t}\prod_{l=s+1}^{t}\l(\I-\eta_{l}\cdot\bSigma_{i,l}^* \r)\eta_{s}\cdot\xi_s\cdot\X_{s}\cdot\II\{a_s=i\}}_{\B_2} \\
		-\underbrace{\sum_{s=0}^{t} \prod_{l=s+1}^{t}\l(\I-\eta_{l}\cdot\bSigma_{i,l}^* \r)\eta_{s}\cdot\l(\bSigma_{i,s}-\bSigma_{i,s}^*\r)\l(\Bbeta_{i}^{(s)}-\Bbeta_{i}^*\r)}_{\B_3}\\
		-\underbrace{\sum_{s=0}^{t} \prod_{l=s+1}^{t}\l(\I-\eta_{l}\cdot\bSigma_{i,l}^* \r)\eta_{s}\cdot\l(\II\{a_s=i\}\cdot\X_s\X_s^{\top}-\bSigma_{i,s} \r)\l(\Bbeta_{i}^{(s)}-\Bbeta_{i}^*\r)}_{\B_4}
%		+\underbrace{\sum_{s=0}^{t} \prod_{l=s+1}^{t}\l(\I-\eta_{l}\cdot\bSigma_{i,l}^* \r)\eta_{s}\xi_s\l(\X_{i,s}-\X_{i,s}^*\r)}_{\B_5}+\underbrace{\sum_{s=0}^{t} \prod_{l=s+1}^{t}\l(\I-\eta_{l}\cdot\bSigma_{i,l}^* \r)\eta_{s}\xi_s\l(\X_s\cdot\II\{a_s=i\}-\X_{i,s} \r)}_{\B_6}.
	\end{multline*}
It is worth noting that for $s=0,1,2,\ldots,t$, the matrices $\bSigma_{i,s}^*$ may not have the common eigenvectors. In what follows, we analyze each of the term $\B_1,\ldots,\B_4$ for $i\in\calA$, respectively. 

\noindent Bound of $\B_1$. First consider the products of the matrices,
\begin{align*}
	0\preceq\prod_{s=1}^{t} \l(\I-\eta_{s}\cdot\bSigma_{i,s}^*\r),
\end{align*}
and under the event $\{\calE_t\}$, we have
\begin{align*}
	\l\|\prod_{s=1}^{t} \l(\I-\eta_{s}\cdot\bSigma_{i,s}^*\r)\r\|&\leq\exp\l(\sum_{s=1}^{t}\log\l(\l\|\I-\eta_{s}\cdot\bSigma_{i,s}^* \r\|\r)\r)\leq\exp\l(-\sum_{s=1}^{t}\eta_{s}\cdot\lambda_{\min}(\bSigma_{i,s}^*) \r)\\
	&\leq \exp\l(-\sum_{s=1}^{t_1}\frac{\Ca}{s+\Cb d}\cdot\frac{\pi_s}{2K} - \sum_{s=t_1+1}^{t}\frac{\Ca}{s+\Cb d}\frac{1}{K}\r)\\
	&\leq\exp\l(-\frac{\Ca}{4K}\ln\l(\frac{t_1+\Cb d}{\Cb d}\r)-\frac{\Ca}{K}\ln\l(\frac{t+\Cb d}{t_1+\Cb d}\r)\r),
\end{align*}
where the first line uses Equation~\ref{eq1:bahadur} and the last line is the integral bound for the sequence. The above equation implies an upper bound for $\|\B_1\|$,
\begin{align*}
	\l\|\B_1\r\|\leq \l(\frac{\Cb d}{t+\Cb d}\r)^{\Ca/4K}\|\Bbeta_{i}^{(0)}-\Bbeta_{i}^*\|,
\end{align*}
which suggests $\B_1=O\l(\frac{1}{t}\r)$ due to $\Ca\geq 4K$.

\medspace

\noindent Analysis of $\B_2$. Notice that $\B_2$ is the sum of $t+1$ martingale difference. We are going to prove $\sqrt{t}\B_2$ converges to a multivariate Gaussian in distribution. To do so, we use Cramer-Wold Theorem \citep{cramer1936some} and martingale central limit theorem \citep{brown1971martingale,durrett2019probability}. For any $d$ dimensional vector $\v=(v_1,\ldots,v_d)$, we are going to prove $\sqrt{t}\v^{\top}\B_2$ is asymptotically Gaussian distributed. Proposition~\ref{prop:LBD_expectation_four} verifies that $\v^{\top}\B_2$ satisfies the Lindeberg's condition. We only need to prove the convergence of conditional variance. Specifically, due to $\xi_s\perp\X_s|\calF_s$, we have
\begin{multline*}
	\EE\l\{\l(\v^{\top}\prod_{l=s+1}^{t}\l(\I-\eta_{l}\cdot\bSigma_{i,l}^* \r)\eta_{s}\cdot\xi_s\cdot\X_{s}\cdot\II\{a_s=i\}\r)^2\bigg|\calF_s \r\}\\
%	=\eta_s^2\sigma^2\v^{\top}\prod_{l=s+1}^{t}\l(\I-\eta_{l}\cdot\bSigma_{i,l}^* \r)\EE\l\{ \X_s\X_s^{\top} \cdot\II\l\{a_s=i\r\} \r\}\prod_{l=s+1}^{t}\l(\I-\eta_{l}\cdot\bSigma_{i,l}^* \r)\v\\
	=\eta_s^2\sigma_s^2\v^{\top}\prod_{l=s+1}^{t}\l(\I-\eta_{l}\cdot\bSigma_{i,l}^* \r)\bSigma_{i,s}\prod_{l=s+1}^{t}\l(\I-\eta_{l}\cdot\bSigma_{i,l}^* \r)\v,
\end{multline*}
where $\sigma_s^2:=\EE\{\xi_s^2|\calF_s\}. $ We then consider the sum of the conditional variance and we are going to find its convergence in probability
\begin{multline*}
	V_{\B2}:=\sum_{s=0}^t \EE\l\{\l(\v^{\top}\prod_{l=s+1}^{t}\l(\I-\eta_{l}\cdot\bSigma_{i,l}^* \r)\eta_{s}\cdot\xi_s\cdot\X_{s}\cdot\II\{a_s=i\}\r)^2\bigg|\calF_s \r\}\\
	=	\sum_{s=0}^t \eta_{s}^2\sigma_s^2\v^{\top}\prod_{l=s+1}^{t}\l(\I-\eta_{l}\cdot\bSigma_{i,l}^* \r)\bSigma_{i,s}\prod_{l=s+1}^{t}\l(\I-\eta_{l}\cdot\bSigma_{i,l}^* \r)\v.
\end{multline*}
We decompose the variance $V_{\B_2}$ into two terms,
\begin{multline*}
	V_{\B2}=	\sum_{s=0}^t \eta_{s}^2\sigma_s^2\v^{\top}\prod_{l=s+1}^{t}\l(\I-\eta_{l}\cdot\bSigma_{i,l}^* \r)\bSigma_{i,s}^*\prod_{l=s+1}^{t}\l(\I-\eta_{l}\cdot\bSigma_{i,l}^* \r)\v\\
	+	\sum_{s=0}^t \eta_{s}^2\sigma_s^2\v^{\top}\prod_{l=s+1}^{t}\l(\I-\eta_{l}\cdot\bSigma_{i,l}^* \r)\l(\bSigma_{i,s}-\bSigma_{i,s}^*\r)\prod_{l=s+1}^{t}\l(\I-\eta_{l}\cdot\bSigma_{i,l}^* \r)\v.
\end{multline*}
We first prove that the second term of the above equation goes to $o(1/t)$. Specifically, we have the following upper bound,
\begin{multline}
	\l| \sum_{s=0}^t \eta_{s}^2\sigma_s^2\v^{\top}\prod_{l=s+1}^{t}\l(\I-\eta_{l}\cdot\bSigma_{i,l}^* \r)\l(\bSigma_{i,s}-\bSigma_{i,s}^*\r)\prod_{l=s+1}^{t}\l(\I-\eta_{l}\cdot\bSigma_{i,l}^* \r)\v\r|\\
	\leq\sigma^2\|\v\|^2\sum_{s=0}^{t} \l(\frac{\Ca}{s+\Cb d}\r)^2\l(\frac{s+\Cb d}{t+\Cb d}\r)^{2\Ca/3K}\l\|\bSigma_{i,s}-\bSigma_{i,s}^* \r\|.
    \label{eq:asymp10}
\end{multline}
We first bound $\EE\l\|\bSigma_{i,s}-\bSigma_{i,s}^* \r\|$ using Proposition~\ref{prop:LBD_expectation_square}, namely
\begin{align*}
    &~~~~\EE\l\|\bSigma_{i,s}-\bSigma_{i,s}^* \r\|\\
    &\leq \EE\l\|\bSigma_{i,s}-\bSigma_{i,s}^* \r\|\cdot\II\l\{\max_j\|\Bbeta_j^{(s)}-\Bbeta_j^*\|\leq h\r\}+\EE\l\|\bSigma_{i,s}-\bSigma_{i,s}^* \r\|\cdot\II\l\{\max_j\|\Bbeta_j^{(s)}-\Bbeta_j^*\|\geq h\r\}\\
    &\leq \kappa_0\EE\max_{j\in\calA}\|\Bbeta_j^{(s)}-\Bbeta_i^*\|+\lambda_{\max}\frac{\EE\max_{j\in[K]}\|\Bbeta_j^{(s)}-\Bbeta_j^*\|}{h}\leq \frac{C}{\sqrt{s}}
\end{align*}
 Hence, as $t\to+\infty$, the expectation of right hand side for Equation~\ref{eq:asymp10} is $O(t^{-3/2})$. Thus we finish showing the second term of $V_{\B_2}$ is $O(t^{-3/2})$. We then analyze the first term of $V_{\B_2}$ and prove its convergence under $\lim_{t\to+\infty}\pi_t=\pi^*\in[0,1]$ and under in probability convergence of $\lim_{t\to+\infty}\sigma_t^2=\sigma_*^2$. %For any $\eps>0$, there exists $T_{\eps}>0$, such that for all $t>T_{\eps}$, the following holds $$|\pi_t-\pi^*|\leq \eps,\quad |\sigma_t^2-\sigma_*^2|\leq\eps,\quad \| \bSigma_{i,t}^*-\bSigma_{i}(\pi^*)\|\leq\eps.$$ %Recall that $$\bSigma_{i}=\pi\bSigma+(1-\pi)\bSigma_{i}^*=\lim_{t\to+\infty}\bSigma_{i,t}^*.$$
%Consider for any $t,s\geq T_{\eps}$,
\begin{multline*}
	\l|\eta_{s}^2\sigma_s^2\v^{\top}\prod_{l=s+1}^{t}\l(\I-\eta_{l}\cdot\bSigma_{i,l}^* \r)\bSigma_{i,s}^*\prod_{l=s+1}^{t}\l(\I-\eta_{l}\cdot\bSigma_{i,l}^* \r)\v\r.\\
	\l.-\eta_{s}^2\sigma_*^2\v^{\top}\prod_{l=s+1}^{t}\l(\I-\eta_{l}\cdot\bSigma_{i}(\pi^*) \r)\bSigma_{i}(\pi^*)\prod_{l=s+1}^{t}\l(\I-\eta_{l}\cdot\bSigma_{i}(\pi^*) \r)\v\r|\\
	\leq \eta_{s}^2\|\v\|^2|\sigma_s^2-\sigma_*^2|\lambda_{\max}\l(\frac{s+\Cb d}{t+\Cb d}\r)^{2\Ca/3K}+2(t-s)\eta_{s}^3\|\v\|^2\lambda_{\max}\l(\frac{s+\Cb d}{t+\Cb d}\r)^{2\Ca/3K}\l\|\bSigma_i(\pi^*)-\bSigma_{i,s}^*\r\|\\
    +\eta_{s}^2\sigma^2\|\v\|^2\lambda_{\max}\l(\frac{s+\Cb d}{t+\Cb d}\r)^{2\Ca/3K}\l\|\bSigma_i(\pi^*)-\bSigma_{i,s}^*\r\|.
\end{multline*}
For any $\eps>0$, Assumption~\ref{assm:cov asymp} guarantees there exists $t_\eps$ such that for all $s\geq t_{\eps}$,
\begin{align*}
    \EE|\sigma_s^2-\sigma_*^2|\leq \eps,\quad   \EE\l\|\bSigma_i(\pi^*)-\bSigma_{i,s}^*\r\|\leq\eps.
\end{align*}
Thus, summing up the above equations and take expectation outsize, for all sufficiently large $t$, we have
\begin{multline*}
	\EE\l|	\sum_{s=0}^t \eta_{s}^2\sigma_s^2\v^{\top}\prod_{l=s+1}^{t}\l(\I-\eta_{l}\cdot\bSigma_{i,l}^* \r)\bSigma_{i,s}^*\prod_{l=s+1}^{t}\l(\I-\eta_{l}\cdot\bSigma_{i,l}^* \r)\v\r. \\
	\l.-	\sum_{s=0}^t \eta_{s}^2\sigma_*^2\v^{\top}\prod_{l=s+1}^{t}\l(\I-\eta_{l}\cdot\bSigma_{i}(\pi^*) \r)\bSigma_{i}(\pi^*)\prod_{l=s+1}^{t}\l(\I-\eta_{l}\cdot\bSigma_{i}(\pi^*) \r)\v\r|\leq C\frac{\eps}{t},
\end{multline*}
which implies
\begin{multline*}
	\lim_{t\to+\infty}t\EE\l|	\sum_{s=0}^t \eta_{s}^2\sigma_s^2\v^{\top}\prod_{l=s+1}^{t}\l(\I-\eta_{l}\cdot\bSigma_{i,l}^* \r)\bSigma_{i,s}^*\prod_{l=s+1}^{t}\l(\I-\eta_{l}\cdot\bSigma_{i,l}^* \r)\v\r. \\
	\l.-	\sum_{s=0}^t \eta_{s}^2\sigma_*^2\v^{\top}\prod_{l=s+1}^{t}\l(\I-\eta_{l}\cdot\bSigma_{i}(\pi^*) \r)\bSigma_{i}(\pi^*)\prod_{l=s+1}^{t}\l(\I-\eta_{l}\cdot\bSigma_{i}(\pi^*) \r)\v\r|=0.
\end{multline*}
Thus, we have convergence in probability,
\begin{multline}
	\lim_{t\to+\infty} t	\sum_{s=0}^t \eta_{s}^2\sigma_s^2\v^{\top}\prod_{l=s+1}^{t}\l(\I-\eta_{l}\cdot\bSigma_{i,l}^* \r)\bSigma_{i,s}^*\prod_{l=s+1}^{t}\l(\I-\eta_{l}\cdot\bSigma_{i,l}^* \r)\v\\
	=\lim_{t\to+\infty}t	\sum_{s=0}^t \eta_{s}^2\sigma_*^2\v^{\top}\prod_{l=s+1}^{t}\l(\I-\eta_{l}\cdot\bSigma_{i}(\pi^*) \r)\bSigma_{i}(\pi^*)\prod_{l=s+1}^{t}\l(\I-\eta_{l}\cdot\bSigma_{i}(\pi^*) \r)\v.
	\label{eq2:bahadur}
\end{multline}
Hence, in order to have $\lim_{t\to+\infty} tV_{\B_2}$, we only need to study the right hand side of Equation~\ref{eq2:bahadur}. Notice that $ \I-\eta_{l}\cdot\bSigma_{i}(\pi^*)$ and $\bSigma_{i}(\pi^*)$ have the common eigenvectors, thus they can be diagonalized simultaneously. Recall that $\bSigma_{i}(\pi^*)$ has the singular value decomposition $\bSigma_{i}(\pi^*)=\U_i(\pi^*)\boldsymbol{\Lambda}_i(\pi^*)\U_i(\pi^*)^{\top}$, where $\U_i(\pi^*)$ is an orthogonal matrix and $\boldsymbol{\Lambda}_i(\pi^*)=\operatorname{diag}(\lambda_1(\pi^*),\ldots,\lambda_d(\pi^*))$ is a diagonal matrix. Then the matrix part for  Equation~\ref{eq2:bahadur} RHS can be rewritten as
\begin{multline*}
	\U(\pi^*)^{\top}\sum_{s=0}^t \eta_{s}^2\prod_{l=s+1}^{t}\l(\I-\eta_{l}\cdot\bSigma_{i}(\pi^*) \r)\bSigma_{i}(\pi^*)\prod_{l=s+1}^{t}\l(\I-\eta_{l}\cdot\bSigma_{i}(\pi^*) \r)\U(\pi^*)\\
	=\l( \begin{matrix}
		\displaystyle \sum_{s=0}^t \eta_{s}^2\prod_{l=s+1}^{t}(1-\eta_{l}\lambda_{1}(\pi^*))^2\lambda_{1}(\pi^*)& &  \\
	%	&\sum_{s=0}^t \eta_{s}^2\prod_{l=s+1}^{t}(1-\eta_{l}\lambda_{2})^2\lambda_{2}&&\\
		&\ddots&\\
		&&\sum_{s=0}^t \eta_{s}^2\prod_{l=s+1}^{t}(1-\eta_{l}\lambda_{d}(\pi^*))^2\lambda_{d}(\pi^*)
	\end{matrix}\r).
\end{multline*}
Equivalently, we only need to analyze the convergence of the diagonal entries. For the simplicity of presentation, we write $\lambda_{j}$ representing $\lambda_{j}(\pi^*)$. To start with, we employ the Taylor's expansion of function $\log(1-x)$ and it arrives at
\begin{align*}
	\sum_{s=0}^t \eta_{s}^2\prod_{l=s+1}^{t}(1-\eta_{l}\lambda_{j})^2\lambda_{j}&=\sum_{s=0}^t \eta_{s}^2\exp\l(2\sum_{l=s+1}^{t}\log\l(1-\eta_{l}\lambda_j\r)\r)\lambda_{j}\\
	&=\sum_{s=0}^t \eta_{s}^2\exp\l(-2\sum_{l=s+1}^{t}\eta_{l}\lambda_j+O\l(\sum_{l=s+1}^{t}\eta_{l}^2\lambda_j^2 \r)\r)\lambda_{j}.
\end{align*}
We then insert the stepsize values into the equation. Specifically, the summation is
\begin{align*}
	\sum_{l=s+1}^{t}\eta_{l}\lambda_j&=\frac{\Ca\lambda_{j}}{\lambda_{\min}}\sum_{l=s+1}^{t}\frac{1}{l+\Cb d}=\frac{\Ca\lambda_{j}}{\lambda_{\min}}\sum_{l=s+1}^{t}\frac{1}{l+\Cb d}\\
	&=\frac{\Ca\lambda_{j}}{\lambda_{\min}}\int_{s+\Cb d+1}^{t+\Cb d+1}\frac{1}{x}  {\rm d}x+\frac{\Ca\lambda_{j}}{\lambda_{\min}}O\l(\sum_{l=s+1}^{t}\l(\frac{1}{l+\Cb d}\r)^2 \r)\\
	&=\frac{\Ca\lambda_{j}}{\lambda_{\min}}\log\l(\frac{t+\Cb d+1}{s+\Cb d+1}\r)+\frac{\Ca\lambda_{j}}{\lambda_{\min}}O\l(\frac{1}{s} \r),
\end{align*}
where the second line uses the Taylor's expansion of $f(\tau)=\int_{a}^{\tau} x^{-1} {\rm d}x$. Moreover, similarly, the square term can be bounded with,
\begin{align*}
	O\l(\sum_{l=s+1}^{t}\eta_{l}^2\lambda_j^2 \r)=O\l(\frac{1}{s}\r).
\end{align*}
Combining the above two equations, it arrives at
\begin{align*}
	\sum_{s=0}^t \eta_{s}^2\prod_{l=s+1}^{t}(1-\eta_{l}\lambda_{j})^2\lambda_{j}&=\sum_{s=0}^t \eta_{s}^2\exp\l( -2\frac{\Ca\lambda_{j}}{\lambda_{\min}} \log\l(\frac{t+\Cb d+1}{s+\Cb d+1}\r)+ O\l(\frac{1}{s}\r) \r)\\
	&=\sum_{s=0}^t \eta_{s}^2\l(\frac{s+\Cb d+1}{t+\Cb d+1}\r)^{2\Ca \lambda_{j}/\lambda_{\min}} \exp\l( O\l(\frac{1}{s}\r) \r)\\
	&=\sum_{s=0}^t \eta_{s}^2\l(\frac{s+\Cb d+1}{t+\Cb d+1}\r)^{2\Ca \lambda_{j}/\lambda_{\min}} \l(1+O\l(\frac{1}{s}\r)\r),
\end{align*}
where the last equation is due to the Taylor's expansion of $\exp(x)$ at $x=0$. We then bound each of the two terms respectively. The first term can be analyzed through Taylor's expansion of function $\int_{a}^{\tau} x^{\alpha}  {\rm d}x$,
\begin{align*}
	&\sum_{s=0}^t \eta_{s}^2\l(\frac{s+\Cb d+1}{t+\Cb d+1}\r)^{2\Ca \lambda_{j}/\lambda_{\min}}=\frac{\Ca^2}{\lambda_{\min}^2}\l(\frac{1}{t+\Cb d+1 }\r)^{2\Ca \lambda_{j}/\lambda_{\min} }\sum_{s=0}^{t}(s+\Cb d+1)^{2\Ca\lambda_{j}/\lambda_{\min}-2}\\
	&~~~~~~~~~~~~~~=\frac{\Ca^2}{\lambda_{\min}^2}\l(\frac{1}{t+\Cb d+1 }\r)^{2\Ca \lambda_{j}/\lambda_{\min} }\l(\int_{\Cb d+1}^{t+\Cb d+1}x^{2\Ca \lambda_{j}/\lambda_{\min}-2}  {\rm d}x +O\l(t^{2\Ca \lambda_{j}/\lambda_{\min}-2}\r)\r)\\
	&~~~~~~~~~~~~~~=\frac{\Ca^2}{\lambda_{\min}^2}\frac{1}{2\Ca\lambda_{j}/\lambda_{\min}-1}\frac{1}{t+\Cb d+1 }+O\l(\l(\frac{1}{t+\Cb d+1 }\r)^2 \r). 
\end{align*}
In a similar fashion, we can bound the $O(1/s)$ term and it's at scale of $O(1/t^2)$. Thus all over, we have the convergence in probability for Equation~\ref{eq2:bahadur} and $tV_{\B_2}$,
\begin{multline*}
		\lim_{t\to+\infty}tV_{\B_2}=\lim_{t\to+\infty}t\sigma_*^2	\sum_{s=0}^t \eta_{s}^2\v^{\top}\prod_{l=s+1}^{t}\l(\I-\eta_{l}\cdot\bSigma_{i}(\pi^*) \r)\bSigma_{i}(\pi^*)\prod_{l=s+1}^{t}\l(\I-\eta_{l}\cdot\bSigma_{i}(\pi^*) \r)\v\\
	=v^{\top}\U \l( \begin{matrix} \frac{\Ca\sigma_*^2}{\lambda_{\min}}\frac{\Ca\lambda_{1}/\lambda_{\min}}{2\Ca\lambda_{1}/\lambda_{\min}-1} & & \\
		%&\frac{\Ca\sigma^2}{\lambda_{\min}}\frac{\Ca\lambda_{2}/\lambda_{\min}}{2\Ca\lambda_{2}/\lambda_{\min}-1}&&\\
		&\ddots&\\
		&&\frac{\Ca\sigma_*^2}{\lambda_{\min}}\frac{\Ca\lambda_{d}/\lambda_{\min}}{2\Ca\lambda_{d}/\lambda_{\min}-1}
	\end{matrix}\r) \U^{\top}\v.
\end{multline*}
Thus, by martingale CLT \citep{brown1971martingale} and the Cramer-Wold Theorem \citep{cramer1936some}, we conclude that the term $\sqrt{t}\B_2$ converges in distribution to multidimensional Gaussian with mean zero and covariance defined in Proposition~\ref{prop:Bahadur}.

\medspace

\noindent Bound of $\B_3$. We still first consider the product of the matrices. For $t>s>t_1$, with similar tricks in Section~\ref{sec:proof-bandit}, we have
\begin{align*}
	\l\|\prod_{l=s+1}^{t}\l(\I-\eta_{l}\cdot\bSigma_{i,l}^* \r)\r\|\leq \exp\l(-\Ca\ln\l(\frac{t+\Cb d}{s+\Cb d}\r)\r)=\l(\frac{s+\Cb d}{t+\Cb d}\r)^{\Ca}.
\end{align*}
On the other hand, tricks for $t>t_1>s$, the matrix product has
\begin{align*}
	\l\|\prod_{l=s+1}^{t}\l(\I-\eta_{l}\cdot\bSigma_{i,l}^* \r)\r\|&\leq\exp\l(-\frac{\Ca}{K}\ln\l(\frac{t_1+\Cb d}{s+\Cb d}\r)-\Ca\ln\l(\frac{t+\Cb d}{t_1+\Cb d}\r)\r)\\
	&\leq\l(\frac{s+\Cb d}{t_1+\Cb d}\r)^{\Ca/K}\cdot\l(\frac{t_1+\Cb d}{t+\Cb d}\r)^{\Ca}\cdot\I_{d}.
\end{align*}
Then we can bound the norm of $\|\B_3\|$,
\begin{align*}
	\|\B_3\|&\leq \sum_{s=0}^{t}\frac{\l( s+\Cb d\r)^{\Ca/K-1}}{(t+\Cb d)^{\Ca/K}}\|\bSigma_{i,s}-\bSigma_{i,s}^*\|\cdot\|\Bbeta_i^{(s)}-\Bbeta_i^*\|.
\end{align*}
Moreover, Equation~\ref{eq3:bahadur} provides the bound for $\|\bSigma_{i,s}-\bSigma_{i,s}^*\|$, and under assumptions of Proposition~\ref{prop:LBD_expectation_square}, we further have
\begin{multline*}
	\EE\l\{\|\B_3\|\r\}
	\leq \sum_{s=0}^{t_1}\frac{\l( s+\Cb d\r)^{\Ca/K-1}}{(t+\Cb d)^{\Ca/K}}\cdot \lambda_{\max}\cdot\EE\l\{\|\Bbeta_i^{(t)}-\Bbeta_i^*\|\r\}\\+\sum_{s=t_1+1}^{t}\frac{\l( s+\Cb d\r)^{\Ca/K-1}}{(t+\Cb d)^{\Ca/K}}\cdot\kappa_0\cdot\EE\l\{\|\Bbeta_i^{(s)}-\Bbeta_i^*\|^2\cdot \II\l\{\|\Bbeta_i^{(s)}-\Bbeta_i^*\|\leq h\r\}\r\}\\
    +\sum_{s=t_1+1}^{t}\frac{\l( s+\Cb d\r)^{\Ca/K-1}}{(t+\Cb d)^{\Ca/K}}\cdot\lambda_{\max}\cdot\EE\l\{\|\Bbeta_i^{(s)}-\Bbeta_i^*\|\cdot \II\l\{\|\Bbeta_i^{(s)}-\Bbeta_i^*\|\geq h\r\}\r\}
\end{multline*}
which implies $\|\B_3\|\leq O_{\PP}(1/t)$.

\medspace

\noindent Bound of $\B_4$. It's worth noting that $\B_4$ is the sum of martingale differences. We consider $\EE\|\B_4\|^2$,
\begin{multline*}
    \EE\l\{\|\B_4\|^2\r\}=\sum_{s=0}^t \prod_{l=s+1}^{t}\l\|\I-\eta_l\cdot\bSigma_{i,l}^*\r\|^2\cdot\eta_s^2\cdot\EE\l\|\l( \II\{a_s=i\}\cdot\X_s\X_s^{\top}-\bSigma_{i,s}\r)\l(\Bbeta_i^{(s)}-\Bbeta_i^*\r)\r\|^2\\
    \leq \sum_{s=0}^t \prod_{l=s+1}^{t}\l\|\I-\eta_l\cdot\bSigma_{i,l}^*\r\|^2\cdot\eta_s^2\cdot\lambda_{\max}^2\EE\l\|\Bbeta_i^{(s)}-\Bbeta_i^* \r\|^2.
\end{multline*}
Hence, by Proposition~\ref{prop:LBD_expectation_square}, Lemma~\ref{teclem:sum upper bound} and Lemma~\ref{teclem:prod upper bound}, we have
\begin{align*}
	\|\B_4\|=O_{\PP}\l(\frac{1}{t}\r).
\end{align*}
Thus, all over, under the conditions of Proposition~\ref{prop:Bahadur}, $\|\B_i\|=O_{\PP}(1/t)$ for $i=1,3,4$ and $\sqrt{t}\B_2$ converges to the multivariate Gaussian in distribution.

\end{proof}

\subsection{Proof of Lemma~\ref{lem:asymp estimate cov}}
First of all, we have
\begin{align*}
    \PP\l(\l\|\frac{1}{t}\sum_{l=1}^t \l\{\X_l\X_l^\top-\EE \l\{\X_l\X_l^\top|\calF_l \r\}\r\}\r\|_F^2\geq \eps\r)\leq \frac{1}{t^2\eps}\sum_{l=1}^t \EE\l\|\X_l\X_l^\top-\EE \l\{\X_l\X_l^\top|\calF_l \r\} \r\|_F^2,
\end{align*}
implying that $\frac{1}{t}\sum_{l=1}^t\X_l\X_l^\top -\frac{1}{t}\sum_{l=1}^t\EE \l\{\X_l\X_l^\top|\calF_l \r\}$ converges to zero in probability. On the other hand, by Assumption~\ref{assm:cov asymp}, we have $\frac{1}{t}\sum_{l=1}^t\EE \l\{\X_l\X_l^\top|\calF_l \r\}$ converges in probability to $\bSigma^*$. The other two arguments can be proved similarly.

\section{Technical Lemmas}
\begin{lemma}[Bernstein's Inequality for Sub-Exponential Martingales]
	Let $Z_1,Z_2,\ldots,Z_T$ be a real valued sequence of martingale difference with respect to $\calF_0,\ldots,\calF_{T}$. Suppose $Z_t|\calF_{t-1}$ is sub-exponential with Orlicz norm bounded by $\| Z_t|\calF_{t-1}\|_{\Psi_1}\leq K_t$, a.s. , where $K_t$ is a constant, then
	\begin{align*}
		\PP\l( \sum_{t=1}^T Z_t\geq s\r)\leq\exp\l(-\min \l\{\frac{s^2}{C\sum_{t=1}^T K_t^2},\frac{s}{2\max_t K_t}\r\}\r).
	\end{align*}
	\label{teclem:azuma}
\end{lemma}
\begin{proof}
	By Markov's inequality, it has
	\begin{align*}
		\PP\l( \sum_{t=1}^T Z_t\geq s\r)&\leq \exp(-\lambda s)\EE\l\{\exp\l(\lambda \sum_{t=1}^T Z_t\r)\r\}\\
		&=\exp(-\lambda s)\EE\l\{ \EE \l\{ \exp\l( \lambda\sum_{i=1}^{T} Z_t\r)\bigg|\calF_{T-1}\r\}\r\}\\
		&=\exp(-\lambda s)\EE\l\{ \exp\l(\lambda\sum_{t=1}^{T-1} Z_t\r)\EE \l\{ \exp\l( \lambda Z_T\r)\bigg| \calF_{T-1}\r\}\r\}.
	\end{align*}
	For $\lambda\leq \frac{1}{\max_t \|Z_t| \calF_{t-1}\|_{\Psi_1}}$, the conditional expectation has
	$$\EE\l\{\exp\l(\lambda Z_T\r)| \calF_{T-1}\r\}\leq\exp\l(C\lambda^2\|Z_T| \calF_{T-1}\|_{\Psi_1}^2\r).$$
	Apply the procedure to $X_{T-1},\ldots,X_1$ respectively and then we have
	\begin{align*}
		\PP\l(\sum_{t=1}^T Z_t\geq s\r)\leq \exp\l( -\lambda s+ C\lambda^2 \sum_{t=1}^{T} K_t^2\r).
	\end{align*}
We then insert the value choice for $\lambda$ $$\lambda=\min\l\{\frac{s}{2C\sum_{t=1}^T K_t^2},\frac{c}{\max_tK_t}\r\}$$ into the tail bound and then we have
	\begin{align*}
		\PP\l( \sum_{t=1}^T Z_t\geq s\r)\leq\exp\l(-\min \l\{\frac{s^2}{C\sum_{t=1}^T K_t^2},\frac{s}{2\max_t K_t}\r\}\r).
	\end{align*}
\end{proof}

\begin{lemma}[Bernstein Inequality based on Orlicz Norm of Random Vectors]
	Suppose $\Z_1,\Z_2,\ldots,\Z_T\in\RR^{d}$ are mean zero vectors, namely, $\EE[\Z_t]=0$ and $\|\Z_t\|_{\Psi_1}$ exists. Then we have the following holds for any $v>0$,
	\begin{align*}
		\PP\l(\l\|\sum_{t=1}^{T}\Z_t \r\| \geq 2v\r)\leq \exp\l(Cd-\min\l\{\frac{v^2}{\sum_{t=1}^T \|\Z_t\|_{\Psi_1}^2},\frac{v}{\max_{t}\|\Z_t\|_{\Psi_1}} \r\}\r),
	\end{align*}
	where $C$ is some constant.
	\label{teclem:Bernstein Orlicz Norm}
\end{lemma}
\begin{proof} We notice that
	\begin{align*}
		\l\|\sum_{t=1}^{T}\Z_t \r\|=\sup_{\u\in\SS^{d-1}} \langle\sum_{t=1}^{T}\Z_t,\u \rangle.
	\end{align*}
	We shall bound the Euclidean norm using $\eps$-net. Additionally, for a fixed $\u$ and for any $v>0$, we have
	\begin{align}
		\PP\l(\l|\l\langle\sum_{t=1}^{T}\Z_t,\u \r\rangle\r|\geq v \r)&\leq 2\exp\l(-\min\l\{\frac{v^2}{\sum_{t=1}^T \|\langle\Z_t,\u\rangle\|_{\Psi_1}^2},\frac{v}{\max_{t}\|\langle\Z_t,\u\rangle\|_{\Psi_1}} \r\} \r)\\
		&\leq2\exp\l(-\min\l\{\frac{v^2}{\sum_{t=1}^T \|\Z_t\|_{\Psi_1}^2},\frac{v}{\max_{t}\|\Z_t\|_{\Psi_1}} \r\} \r),
		\label{eq1}
	\end{align}
	which uses the definition of Orlicz norm for random vectors. Suppose $\calN$ is a $1/2$-net of $\SS^{d-1}$ with cardinality $|\calN|\leq 5^d$ and we take a union of Equation~\ref{eq1} on $\calN$,
	\begin{align*}
		\PP\l(\max_{\u\in\calN}\l|\l\langle\sum_{t=1}^{T}\Z_t,\u \r\rangle\r|\geq v \r)\leq 2|\calN|\exp\l(-\min\l\{\frac{v^2}{\sum_{t=1}^T \|\Z_t\|_{\Psi_1}^2},\frac{v}{\max_{t}\|\Z_t\|_{\Psi_1}} \r\} \r).
	\end{align*}
	For any $\u\in\SS^{d-1}$, there exists $\x\in\calN$, such that $\|\u-\x\|\leq1/2$. Thus, we have
	\begin{multline*}
		\l\|\sum_{t=1}^{T}\Z_t \r\|=\sup_{\u\in\SS^{d-1}} \l|\l\langle\sum_{t=1}^{T}\Z_t,\u \r\rangle\r|\leq \max_{\u\in\calN}\l|\l\langle\sum_{t=1}^{T}\Z_t,\u \r\rangle\r|+\sup_{\u:\; \|\u\|\leq\eps} \l|\l\langle\sum_{t=1}^{T}\Z_t,\u \r\rangle\r|\\
		=\max_{\u\in\calN}\l|\l\langle\sum_{t=1}^{T}\Z_t,\u \r\rangle\r|+\frac{1}{2}\l\|\sum_{t=1}^{T}\Z_t \r\|,
	\end{multline*}
	which shows that
	\begin{align*}
		\l\|\sum_{t=1}^{T}\Z_t \r\|=\sup_{\u\in\SS^{d-1}} \l|\l\langle\sum_{t=1}^{T}\Z_t,\u \r\rangle\r|\leq2\max_{\u\in\calN}\l|\l\langle\sum_{t=1}^{T}\Z_t,\u \r\rangle\r|.
	\end{align*}
	Thus we have
	\begin{align*}
		\PP\l( \l\|\sum_{t=1}^{T}\Z_t \r\|\geq 2v\r)\leq 2\times 5^{d} \exp\l(-\min\l\{\frac{v^2}{\sum_{t=1}^T \|\Z_t\|_{\Psi_1}^2},\frac{v}{\max_{t}\|\Z_t\|_{\Psi_1}} \r\} \r),
	\end{align*}
	which completes the proof.
\end{proof}

\begin{lemma}[Martingale Bernstein Inequality for Vectors]
	Suppose $\Z_1,\Z_2,\ldots,\Z_T$ are $d$-dimensional martingale difference with respect to $\calF_0,\calF_{1},\ldots,\calF_{T}$, namely, $\EE[\Z_t|\calF_{t-1}]=0$ a.s. and we suppose $\|\Z_t|\calF_{t-1}\|_{\Psi_1}\leq K_t$ a.s. for some contant $K_t$. Then for any $v>0$, we have
	\begin{align*}
		\PP\l(\l\|\sum_{t=1}^{T}\Z_t \r\| \geq 2v\r)\leq \exp\l(Cd-\min\l\{\frac{v^2}{\sum_{t=1}^T K_t^2},\frac{v}{\max_{t} K_t } \r\}\r),
	\end{align*}
	where $C$ is some constant.
	\label{teclem:Bernstein Martingale}
\end{lemma}
\begin{proof}
	The proof procedure would be similar to Lemma~\ref{teclem:Bernstein Orlicz Norm} and additionally, Lemma~\ref{teclem:azuma} would be used. Notice that for any fixed $\u$ and for any $v>0$, Lemma~\ref{teclem:azuma} indicates that
	\begin{align*}
		\PP\l(\l|\l\langle\sum_{t=1}^{T}\Z_t,\u \r\rangle\r|\geq v \r)%&\leq 2\exp\l(-\min\l\{\frac{v^2}{\sum_{t=1}^T \|\langle\Z_t,\u\rangle|\calF_{t-1}\|_{\Psi_1}^2},\frac{v}{\max_{t}\|\langle\Z_t,\u\rangle|\calF_{t-1}\|_{\Psi_1}} \r\} \r)\\
		&\leq2\exp\l(-\min\l\{\frac{v^2}{\sum_{t=1}^T K_t^2},\frac{v}{\max_{t}K_t} \r\} \r).
	\end{align*}
	Then with a similar routine to Lemma~\ref{teclem:Bernstein Orlicz Norm}, we have
	\begin{align*}
		\PP\l( \l\|\sum_{t=1}^{T}\Z_t \r\|\geq 2v\r)\leq 2\times 5^d\exp\l(-\min\l\{\frac{v^2}{\sum_{t=1}^T K_t^2},\frac{v}{\max_{t} K_t} \r\} \r),
	\end{align*}
which completes the proof.
\end{proof}
\begin{lemma}
	Suppose $\Z_1,\Z_2,\ldots,\Z_T$ are $d$-dimensional martingale difference with respect to $\calF_0,\calF_{1},\\\ldots,\calF_{T}$, namely, $\EE[\Z_t|\calF_{t-1}]=0$ a.s. and we suppose $\sup_{|\calS|\leq s,\, \calS\subseteq[d]}\|\calH_{\calS}(\Z_t)|\calF_{t-1}\|_{\Psi_1}\leq K_t$ a.s. for some contant $K_t$. Then for any $v>0$, we have
	\begin{align*}
		\PP\l(\sup_{|\calS|\leq s,\; \calS\subseteq[d]}\l\|\calH_{\calS}\l(\sum_{t=1}^{T}\Z_t\r) \r\| \geq 2v\r)\leq \exp\l(Cs\log(ed/s)-\min\l\{\frac{v^2}{\sum_{t=1}^T K_t^2},\frac{v}{\max_{t}K_t} \r\}\r),
	\end{align*}
	where $C$ is some constant.
	\label{teclem:Bernstein partial length}
\end{lemma}
\begin{proof}
	Lemma~\ref{teclem:Bernstein Martingale} indicates that for any fixed $\calS$,
	\begin{align*}
			\PP\l( \l\|\calH_{\calS}\l(\sum_{t=1}^{T}\Z_t \r)\r\|\geq 2v\r)\leq 2\times 5^s\exp\l(-\min\l\{\frac{v^2}{\sum_{t=1}^T K_t^2},\frac{v}{\max_{t} K_t} \r\} \r).
	\end{align*}
Then take the union over all $\calS$, and it leads to
\begin{align*}
	\PP\l( \l\|\calH_{\calS}\l(\sum_{t=1}^{T}\Z_t \r)\r\|\geq 2v\r)\leq \binom{d}{s}\times2\times 5^s\exp\l(-\min\l\{\frac{v^2}{\sum_{t=1}^T K_t^2},\frac{v}{\max_{t} K_t} \r\} \r),
\end{align*}
with $\binom{d}{s}\leq (ed/s)^s$ completing the proof.
\end{proof}
\begin{lemma}
	Suppose $g_1, \ldots, g_n$ are zero-mean sub-Gaussian random variables and the length $n$ vector $\g=(g_1,\ldots,g_n)$ $s-\Psi_2$ Orlicz-norm at most $\sigma$, namely $\sup_{\calS\subseteq[n],\;|\calS|\leq s}\|[\g]_{\calS}\|_{\Psi_2}\leq\sigma$. Denote by $\left(g_{(1)}, \ldots, g_{(n)}\right)$ be a non-increasing rearrangement of $\left(\left|g_1\right|, \ldots,\left|g_n\right|\right)$. Then
	\begin{align*}
		\mathbb{P}\left(\sqrt{\sum_{j=1}^s\left(g_{(j)}\right)^2}\geq C\sqrt{s\log(en/s)}\sigma+u\sigma\right) \leq \exp(-cu^2)
	\end{align*}
	for all $t>0$ and $s \in\{1, \ldots, p\}$.
	\label{teclem:maxsum tail}
\end{lemma}
\begin{proof}
	Based on the assumption, for any fixed set $\calS\subseteq[n]$ and $|\calS|=s$, we have
	\begin{align*}
		\PP\l( \sqrt{\sum_{i\in \calS}(g_i)^2} \geq C\sqrt{s}\sigma+u\r)\leq \exp\l(-\frac{cu^2}{\sigma^2}\r).
	\end{align*}
	We then take the uniform for all the $s$ subset of $[n]$. It's worth noting that
	\begin{align*}
		\#\{\calS:\; |\calS|=s,\; \calS\subseteq [n]\}=\binom{n}{s}\leq\l(\frac{en}{s}\r)^s.
	\end{align*}
	Then we have
	\begin{align*}
		\PP\l(\max_{|\calS|=s,\; \calS\subseteq [n] } \sqrt{\sum_{i\in \calS}(g_i)^2} \geq C\sqrt{s}\sigma+u\r)\leq \l(\frac{en}{s}\r)^s \exp\l(-\frac{cu^2}{\sigma^2}\r).
	\end{align*}
	It implies
	\begin{align*}
		\PP\l(\max_{|\calS|=s,\; \calS\subseteq [n] } \sqrt{\sum_{i\in \calS}(g_i)^2} \geq C\sqrt{s\log(en/s)}\sigma+u\sigma\r)\leq \exp\l(-cu^2\r),
	\end{align*}
	which finishes the proof.
\end{proof}
\begin{lemma}
	Suppose $\bSigma_1,\ldots,\bSigma_n$ are symmetric positive definite matrices with minimal eigenvalues greater than $\lambda_{\min}>0$, namely, $\lambda_{\min}(\bSigma_{i})\geq\lambda_{\min}$. Then we have $\lambda_{\min}(\sum_{i=1}^n \bSigma_{i})\geq n\lambda_{\min}$.
	\label{teclem:min eigenvalue}
\end{lemma}
\begin{proof}
	It's worth noting that $\lambda_{\min}(\sum_{i=1}^n \bSigma_{i})=\inf_{\u\in\SS^{d-1}}\u^{\top} (\sum_{i=1}^n \bSigma_{i}) \u$. On the other hand, for any $\u\in\SS^{d-1}$, we have $\u^{\top}  \bSigma_{i} \u\geq\lambda_{\min}$. Thus, we have $\inf_{\u\in\SS^{d-1}}\u^{\top} (\sum_{i=1}^n \bSigma_{i}) \u\geq \sum_{i=1}^{n} \inf_{\u\in\SS^{d-1}} \u^{\top}\bSigma_{i}\u\geq n\lambda_{\min}$, which completes the proof.
\end{proof}
\begin{lemma}
	Suppose $\Z_1,\Z_2,\ldots,\Z_T$ are $d$-dimensional martingale difference with respect to $\calF_0,\calF_{1},\ldots,\calF_{T}$, namely, $\EE[\Z_t|\calF_{t-1}]=0$ a.s. and we suppose $\|\Z_t|\calF_{t-1}\|_{\Psi_1}\leq K_t$ a.s. for some constant $K_t$. Then for any $v>0$, we have
	\begin{align*}
		\PP\l(\l\|\sum_{t=1}^{T}\Z_t \r\|_\infty \geq 2v\r)\leq d\exp\l(-\min\l\{\frac{v^2}{C\sum_{t=1}^T K_t^2},\frac{v}{\max_{t} K_t} \r\} \r),
	\end{align*}
	where $C$ is some constant.
	\label{teclem:vec max norm}
\end{lemma}
\begin{proof}
	For any $i=1,\ldots, d$, the term $\sum_{t=1}^{T}[\Z_t]_i$ is the sum of martingale differences, with $\|[\Z_t]_i\|_{\Psi_1}\leq K_t$. By Lemma~\ref{teclem:azuma}, for any $v$, we have,
	\begin{align*}
		\PP\l( \l|\sum_{t=1}^{T}[\Z_t]_i \r|\geq 2v\r)\leq \exp\l(-\min\l\{\frac{v^2}{C\sum_{t=1}^{T} K_t^2},\frac{v}{\max_t K_t} \r\}\r).
	\end{align*}
Take the uniform for all $i=1,\ldots,d$ and it leads to
\begin{align*}
	\PP\l( \max_{i=1,\ldots,d} \l|\sum_{t=1}^{T}[\Z_t]_i \r|\geq 2v\r)\leq d\exp\l(-\min\l\{\frac{v^2}{C\sum_{t=1}^{T} K_t^2},\frac{v}{\max_t K_t} \r\}\r),
\end{align*}
which completes the proof.
\end{proof}

\begin{lemma}
	For any positive values $\alpha,\beta>0$ and any positive integers $t\geq s$, we have
	\begin{align*}
		\prod_{l=s}^{t}\l(1-\frac{\alpha}{l+\beta}\r)\leq\l(\frac{s+\beta}{t+1+\beta}\r)^\alpha.
	\end{align*}
	\label{teclem:prod upper bound}
\end{lemma}
\begin{proof} Since $\log(1-x)\leq x$, we have
	\begin{align*}
		\prod_{l=s}^{t}\l(1-\frac{\alpha}{l+\beta}\r)=\exp\l(\sum_{l=s}^{t} \log\l(1-\frac{\alpha}{l+\beta} \r)\r)\leq\exp\l(-\sum_{l=s}^{t}\frac{\alpha}{l+\beta} \r).
	\end{align*}
	Moreover, the decreasing function has
	\begin{align*}
		\sum_{l=s}^{t}\frac{\alpha}{l+\beta}\geq\int_{s}^{t+1}\frac{\alpha}{x+\beta}  {\rm d}x =\alpha\log\l(\frac{t+1+\beta}{s+\beta}\r).
	\end{align*}
	Thus, it arrives at
	\begin{align*}
		\prod_{l=s}^{t}\l(1-\frac{\alpha}{l+\beta}\r)\leq \l(\frac{s+\beta}{t+1+\beta}\r)^\alpha.
	\end{align*}
\end{proof}
\begin{lemma}
	For any positive values $\alpha,\beta>0$ and any positive integers $t\geq s$, we have
	\begin{align*}
		\sum_{l=s}^{t}(l+\beta)^\alpha\leq \frac{1}{\alpha +1} (t+1+\beta)^{\alpha+1}.
	\end{align*}
	\label{teclem:sum upper bound}
\end{lemma}
\end{document}